\documentclass[11pt,a4paper]{article}
\pdfoutput=1

\usepackage[latin9]{inputenc}
\usepackage{amsmath}
\usepackage{amsfonts}
\usepackage{amssymb}
\usepackage{amsthm}
\usepackage{color}
\usepackage[english]{babel}

\usepackage{graphicx}

\theoremstyle{plain}
\newtheorem{thm}{Theorem}[section]
\newtheorem{prop}[thm]{Proposition}
\newtheorem{lem}[thm]{Lemma}
\newtheorem{cor}[thm]{Corollary}

\theoremstyle{definition}

\newcommand{\ind}[1]{\mathbf{1}_{\lbrace #1 \rbrace}}
\newcommand{\N}{\mathbb{N}}
\newcommand{\Z}{\mathbb{Z}}
\newcommand{\R}{\mathbb{R}}

\renewcommand{\(}{\left(}
\renewcommand{\)}{\right)}

\renewcommand{\{}{\left\lbrace}
\renewcommand{\}}{\right\rbrace}

\newcommand{\segs}[2]{[ \! [ #1 , #2 ] \! ]}
\newcommand{\abs}[1]{\left|#1\right|}

\newcommand{\ent}[1]{\lfloor #1 \rfloor}

\newcommand{\card}[1]{\# #1}

\newcommand{\cvg}[2]{\xrightarrow[#1]{#2}}

\newcommand{\E}[2][]{\mathbb{E}\if\relax#1\relax[#2]\else\left[#2\right]\fi}
\newcommand{\Efrom}[3][]{\mathbb{E}_{#2}\if\relax#1\relax[#3]\else\left[#3\right]\fi}
\renewcommand{\P}[2][]{\mathbb{P}\if\relax#1\relax(#2)\else\left(#2\right)\fi}
\newcommand{\Pfrom}[3][]{\mathbb{P}_{#2}\if\relax#1\relax(#3)\else\left(#3\right)\fi}

\newcommand{\as}{\mbox{a.s.}}

\newcommand{\ul}[1]{\underline{#1}}
\newcommand{\ol}[1]{\overline{#1}}
\newcommand{\ora}[1]{\overrightarrow{#1}}
\newcommand{\ola}[1]{\overleftarrow{#1}}
\newcommand{\olra}[1]{\overleftrightarrow{#1}}
\newcommand{\dbtilde}[1]{\tilde{\raisebox{0cm}[0.9\height]{$\tilde{#1}$}}}

\newcommand{\ltrees}[1]{\mathbb{T}_{#1}}
\newcommand{\lforests}[2]{\mathbb{F}_{#1,#2}}
\newcommand{\pt}[1]{\mathrm{#1}}
\newcommand{\spine}{\mathfrak{s}} 
\newcommand{\strees}{\textbf{S}} 
\newcommand{\lstrees}{\mathbb{S}} 
\newcommand{\uhpq}{\ola{\textbf{Q}}}
\newcommand{\lhpq}{\ora{\textbf{Q}}}
\DeclareMathOperator{\spq}{Sp}
\DeclareMathOperator{\geom}{Geom}
\newcommand{\GWgeom}[1]{\rho_{(#1)}}
\newcommand{\MX}{M\!X}
\newcommand{\MY}{M\!Y}
\newcommand{\MZ}{M\!Z}
\newcommand{\SumFW}[2]{H_{#1}(#2)}
\newcommand{\SumFWStar}[2]{H_{#1}^{\ast}(#2)}
\newcommand{\SumGV}[2]{H'_{#1}(#2)}
\newcommand{\LB}{LB}
\newcommand{\RB}{RB}

\renewcommand{\epsilon}{\varepsilon}

\title{The UIPQ seen from a point at infinity along its geodesic ray}
\date{\today}
\author{Daphné \textsc{Dieuleveut}}

\begin{document}

\maketitle

\textsc{Abstract:} 
We consider the uniform infinite quadrangulation of the plane (UIPQ). Curien, Ménard and Miermont recently established that in the UIPQ, all infinite geodesic rays originating from the root are essentially similar, in the sense that they have an infinite number of common vertices. In this work, we identify the limit quadrangulation obtained by rerooting the UIPQ at a point at infinity on one of these geodesics. 
More precisely, calling $v_k$ the $k$-th vertex on the ``leftmost'' geodesic ray originating from the root, and $Q_{\infty}^{(k)}$ the UIPQ re-rooted at $v_k$, we study the local limit of $Q_{\infty}^{(k)}$. To do this, we split the UIPQ along the geodesic ray $(v_k)_{k\geq 0}$. Using natural extensions of the Schaeffer correspondence with discrete trees, we study the quadrangulations obtained on each ``side'' of this geodesic ray. We finally show that the local limit of $Q_{\infty}^{(k)}$ is the quadrangulation obtained by gluing the limit quadrangulations back together.

\section{Introduction}

\label{S: Setting}

Finite and infinite planar maps are a popular model for random geometry. While finite maps have been studied since the sixties, infinite models were only introduced a decade ago, with the works of Angel and Schramm \cite{AngSchr,Ang_Peeling}. They were the first to define the uniform infinite planar triangulation, an infinite map which can be seen as the local limit (in distribution) of uniform finite triangulations. Krikun \cite{Kri} then studied its counterpart, the uniform infinite planar quadrangulation (UIPQ), defined as the limit of uniform rooted finite quadrangulations as the number of faces goes to infinity. In this article, we study what the UIPQ looks like seen from a point ``at infinity'' on a geodesic ray originating from the root.

One of the main advantages of quadrangulations over other classes of planar maps is the existence of the so-called Cori-Vauquelin-Schaeffer bijection. This bijection, introduced in \cite{CV} and developed thoroughly in \cite{Sch_Thesis,ChaSch}, gives a correspondence between finite quadrangulations and well-labeled finite trees. It was in particular used by Chassaing and Durhuus \cite{ChaDu} as a new approach to the UIPQ: they studied the infinite quadrangulation of the plane corresponding to an infinite positive labeled tree, and it was shown later by Ménard \cite{Men} that this quadrangulation has the same distribution as the one defined by Krikun.

Using another extension of the Cori-Vauquelin-Schaeffer bijection, Curien, Ménard and Miermont \cite{CuMenMi} recently showed that the UIPQ can also be obtained from a ``uniform'' infinite labeled tree, without the positivity constraint on the labels. This construction allowed them to prove new results on the UIPQ, and in particular to give a fine description of the geodesic arcs from a point to infinity. One of their main results states that all such geodesics are ``trapped'' between two distinguished geodesics, which have a simple description in terms of the corresponding labeled tree. Moreover, these two geodesics, called the maximal (or leftmost) and minimal (or rightmost) geodesics, are roughly similar, in the sense that they almost surely have an infinite number of common points.

Our main goal here is to study the local limit of $Q_{\infty}^{(k)}$ as $k \rightarrow \infty$, where $Q_{\infty}^{(k)}$ denotes the UIPQ re-rooted at a point at distance $k$ from the root, on the leftmost geodesic. Our methods are again based on bijective correspondences between trees and quadrangulations. Specifically, we show that $Q_{\infty}^{(k)}$ converges in distribution to a limit quadrangulation $\olra{Q}_{\infty}$, which can be obtained by gluing together two quadrangulations of the half-plane with geodesic boundaries; we give explicit expressions for the distribution of the corresponding trees. Note that the laws of the quadrangulations of the half-plane we consider (corresponding to the parts of the UIPQ which are ``on the left'' and ``on the right'' of the leftmost geodesic ray) are orthogonal to the law of the uniform infinite quadrangulation of the half-plane (UIHPQ) which was studied in \cite{Ang_UIHPT} and \cite{CuMi}.

Finally, note that the scaling limit of the uniform infinite quadrangulation, the Brownian plane, which was introduced and studied by Curien and Le Gall \cite{CuLG}, has a similar ``uniqueness'' property of infinite geodesic rays started from the root. We expect our result to have a natural analog in this context.

In the rest of this introduction, we give the necessary definitions to state our main results. In Section \ref{S: Definitions}, we first recall classical definitions on quadrangulations and labeled trees; we also describe the construction of the UIPQ given in \cite{CuMenMi} and the ``Schaeffer-type'' correspondence it relies on. Section \ref{S: Leftmost geodesic} gives more details on the UIPQ re-rooted at the $k$-th point on the leftmost infinite geodesic ray starting from the root. In particular, we explain why it is enough to study the local limit of the parts on each side of this geodesic. This leads us to extend the correspondence to a larger class of infinite labeled trees, which encode planar quadrangulations with a geodesic boundary (see Section 1.3). Finally, in Section 1.4, we state our main convergence results for these trees and the associated quadrangulations.

\subsection{Well-labeled trees and associated quadrangulations} \label{S: Definitions}

\subsubsection{First definitions on finite and infinite planar maps} \label{S: Planar maps}

A finite planar map is a proper embedding of a finite connected graph, possibly with multiple edges or loops, into the two-dimensional sphere (or more rigorously, the equivalence class of such a graph, modulo orientation-preserving homeomorphisms).

We first introduce some notation for such a map $\mathfrak{m}$. Let $V(\mathfrak{m})$, $E(\mathfrak{m})$ and $\ora{E}(\mathfrak{m})$ denote the sets of the vertices, edges and oriented edges of $\mathfrak{m}$, respectively. The faces of $\mathfrak{m}$ are the connected components of the complement of $E(\mathfrak{m})$. We say that a face is incident to $e \in \ora{E}(\mathfrak{m})$ if it is the face on the left of $e$. The degree of a face is the number of edges it is incident to. A corner of $\mathfrak{m}$ is an angular sector between two edges of $\mathfrak{m}$. Note that there is a bijective correspondence between the corners of $\mathfrak{m}$ and its oriented edges; we say that a corner is incident to $e \in \ora{E}(\mathfrak{m})$ if it is the corner on the left of $e$, next to its origin.

We say that a finite planar map is rooted if it comes with a distinguished oriented edge, called the root edge; the origin vertex of the root is called the root vertex, and the face which is incident to the root is called the root face. A planar map is a quadrangulation if all faces have degree $4$, and a tree if it has only one face. A quadrangulation with a boundary is a planar map with a distinguished face called the external face, such that the boundary of the external face is simple and all other faces have degree $4$. We let $\mathcal{Q}_f$, $\mathcal{Q}_{f,b}$ and $\mathcal{T}_f$ respectively denote the sets of finite quadrangulations, quadrangulations with a boundary and trees.

Let us now define the local limit topology on these sets. For any rooted map $\mathfrak{m}$, let $B_{\mathfrak{m}}(r)$ denote the ball of radius $r$ in $\mathfrak{m}$, centered at the root-vertex (i.e. the planar map defined by the edges of $\mathfrak{m}$ whose extremities are both at distance at most $r$ from the root-vertex, for the graph-distance on $\mathfrak{m}$). For all finite planar maps $\mathfrak{m}$, $\mathfrak{m}'$, we let
\begin{gather*}
D(\mathfrak{m},\mathfrak{m}') = (1+\sup \{ r \geq 0: B_{\mathfrak{m}}(r) = B_{\mathfrak{m}'}(r) \})^{-1}.
\end{gather*}
The local topology is the topology associated to this distance. Let $\mathcal{Q}$, $\mathcal{Q}_b$ and $\mathcal{T}$ denote the completions of $\mathcal{Q}_f$, $\mathcal{Q}_{f,b}$ and $\mathcal{T}_f$ for this topology. The elements of $\mathcal{Q}_{\infty} := \mathcal{Q} \setminus \mathcal{Q}_f$ (resp. $\mathcal{T}_{\infty} := \mathcal{T} \setminus \mathcal{T}_f$) are \emph{infinite} planar quadrangulations (resp. trees). All the notations introduced above for finite planar maps have natural extensions to the above sets. We let $\mathcal{Q_{\infty,\infty}}$ denote the set of the quadrangulations with an infinite boundary, i.e. the elements of $\mathcal{Q}_b$ which are defined as limits of sequences of maps in $\mathcal{Q}_{f,b}$ whose external faces have degrees going to infinity.

Any element $Q$ of $\mathcal{Q}_{\infty}$ or $\mathcal{Q}_{\infty,\infty}$ can be seen as a gluing of quadrangles which defines an orientable, connected, separable surface, with a boundary in the second case. See \cite[Appendix]{CuMenMi} for details. We are interested in two cases:
\begin{itemize}
\item If the corresponding surface is homeomorphic to $S=\R^2$, we say that $Q$ is an infinite quadrangulation of the plane.
\item If the corresponding surface is homeomorphic to $S=\R \times \R_{+}$, we say that $Q$ is an infinite quadrangulation of the half-plane.
\end{itemize}
In both of these cases, $Q$ can be drawn onto $S$ in such a way that every face is bounded, every compact subset of $S$ intersects only finitely many edges of $Q$, and in the second case, the union of the boundary edges is $\R \times \{0\}$. By convention, if the root edge belongs to this boundary and is oriented from left to right, we say that $Q$ is a quadrangulation of the upper half-plane, and if it is oriented from right to left, we say that $Q$ is a quadrangulation of the lower half-plane. We let $\textbf{Q}$ denote the set of the quadrangulations of the plane, and $\uhpq$ (resp. $\lhpq$) denote the set of the infinite quadrangulations $Q$ of the upper half-plane (resp. lower half-plane) such that the boundary of $Q$ is a geodesic path in $Q$.

As explained in \cite[Appendix]{CuMenMi}, an element $Q$ of $\mathcal{Q}_{\infty}$ is a quadrangulation of the plane if and only if it has exactly one \emph{end} - which means, in terms of maps, that for all $r \in \N$, the map $Q \setminus B_{Q}(r)$ has exactly one infinite connected component. For an element $Q$ of $\mathcal{Q}_{\infty,\infty}$, one can check that $Q$ is a quadrangulation of the half-plane if and only if the same condition holds. Indeed, the infinite quadrangulation obtained by gluing a copy of the lattice $\Z \times \Z_-$ along the boundary also has one end (the number of ends can only decrease when we perform this operation), so it is a quadrangulation of the plane.

In what follows, the trees and quadrangulations we consider will be elements of $\textbf{T} := \mathcal{T}$, $\textbf{Q}$, $\uhpq$ and $\lhpq$. The uniform infinite quadrangulation (UIPQ) is a random variable in $\textbf{Q}$ whose distribution is the limit of the uniform distribution on planar quadrangulations with $n$ faces, as $n \rightarrow \infty$.

\subsubsection{Well-labeled trees} \label{S: Labelled trees}

We say that $(T,l)$ is a well-labeled (plane, rooted) tree if $T$ is an element of $\textbf{T}$ and $l$ is a mapping from $V(T)$ into $\Z$ such that $\abs{l(\pt{u}) - l(\pt{v})} \leq 1$ for every pair of neighbouring vertices $\pt{u},\pt{v}$. Let $\ltrees{}$ be the set of such trees. More precisely, for all $x \in \Z$ and $n \geq 0$, let $\ltrees{n}(x)$ be the set of well-labeled plane rooted trees with $n$ edges and root-label $x$, and
\begin{gather*}
\ltrees{n} = \bigcup_{x \in \Z} \ltrees{n}(x).
\end{gather*}
Similarly, for all $x \in \Z$, let $\ltrees{\infty}(x)$ denote the set of infinite well-labeled plane rooted trees with root-label $x$, and
\begin{gather*}
\ltrees{\infty} = \bigcup_{x \in \Z} \ltrees{\infty}(x).
\end{gather*}
We thus have
\begin{gather*}
\ltrees{} = \bigcup_{n \in \N\cup \{0,\infty\}} \ltrees{n} = \bigcup_{n \in \N\cup \{0,\infty\}} \bigcup_{x \in \Z} \ltrees{n}(x).
\end{gather*}

For any (infinite) plane rooted tree $T$, we say that $(\pt{u}_i)_{i \geq 0}$ is a spine in $T$ if $\pt{u}_0$ is the root of $T$ and if for all $i \geq 0$, $\pt{u}_i$ is the parent of $\pt{u}_{i+1}$. We let $\strees$ be the set of all plane rooted trees having exactly one spine, and consider the corresponding sets of labeled trees:
\begin{gather*}
\lstrees(x) = \{(T,l) \in \ltrees{\infty}(x): T \in \strees \} \qquad \forall x \in \Z, \\
\lstrees = \{(T,l) \in \ltrees{\infty}: T \in \strees \}.
\end{gather*}
For every $T \in \strees$, we let $(\spine_i(T))_{i \geq 0}$ be the spine of $T$. Any vertex $\spine_i(T)$ has a subtree ``to its left'' and a subtree ``to its right'' in $T$, which we denote by $L_i(T)$ and $R_i(T)$ respectively. To give a formal definition of these subtrees, we consider two orders on $V(T)$: the depth-first order, denoted by $<$, and the partial order $\prec$ induced by the genealogy, defined for all $\pt{u},\pt{v} \in V(T)$ by $\pt{u} \prec \pt{v}$ if $\pt{u}$ is an ancestor of $\pt{v}$ in $T$. With this notation:
\begin{itemize}
\item $L_i(T)$ is the subtree of $T$ containing the vertices $\pt{v}$ such that $\spine_i \leq \pt{v} < \spine_{i+1}$.
\item $R_i(T)$ is the subtree of $T$ containing $\spine_i$ and the vertices $\pt{v}$ such that $\spine_{i+1} < \pt{v}$ and $\spine_{i+1} \nprec \pt{v}$.
\end{itemize}
We also use the natural extensions of these notations to well-labeled trees.

\subsubsection{The Schaeffer correspondence between infinite trees and quadrangulations} \label{S: Schaeffer}

In this section, we recall the definition of the Schaeffer correspondence used in \cite{CuMenMi}, which matches infinite well-labeled trees with infinite quadrangulations of the plane.

For all $x \in \Z$, let
\begin{gather*}
\lstrees^{\ast}(x) = \{ (T,l) \in \lstrees(x): \inf_{i \geq 0} l(\spine_i(T)) = -\infty \}.
\end{gather*}
We fix $\theta=(T,l) \in \lstrees^{\ast}(0)$. Let $\pt{c}_n$, $n \in \Z$ denote the corners of $T$, taken in the clockwise order, with $\pt{c}_0$ the root-corner. For all $n$, we say that the label of $\pt{c}_n$ is the label of the vertex which is incident to $\pt{c}_n$, and we define the successor $\sigma_{\theta}(\pt{c}_n)$ of $\pt{c}_n$ as the first corner among $\pt{c}_{n+1}, \pt{c}_{n+2}, \ldots$ such that
\begin{gather*}
l(\sigma_{\theta}(\pt{c}_n)) = l(\pt{c}_n)-1.
\end{gather*}
We now let $\Phi(\theta)$ denote the graph whose set of vertices is $V(T)$, whose edges are the pairs $\{\pt{c},\sigma_{\theta}(\pt{c})\}$ for all corners $\pt{c}$ of $T$, and whose root-edge is $(\pt{c}_0,\sigma_{\theta}(\pt{c}_0))$. Figure \ref{F: Schaeffer Bijection} gives an example of this construction. Note that $\Phi(\theta)$ can be embedded naturally in the plane, by considering a specific embedding of $T$ and drawing arcs between every corner and its successor in a non-crossing way. Moreover, Proposition 2 of \cite{CuMenMi} shows that for all $\theta \in \lstrees^{\ast}(0)$, $\Phi(\theta)$ is an infinite quadrangulation of the plane.
\begin{figure}[t]
\begin{center}
\includegraphics{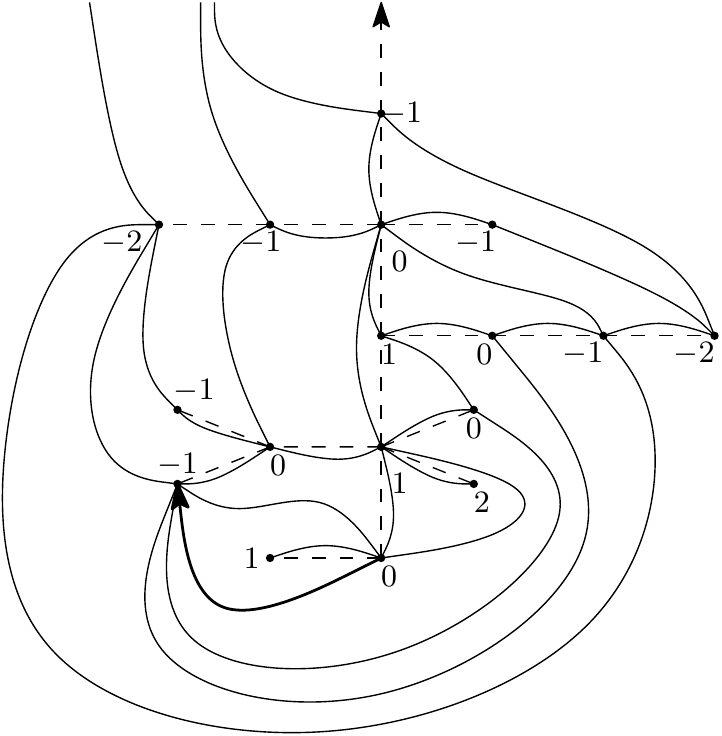}
\caption{The quadrangulation $\Phi(\theta)$ obtained by applying the Schaeffer correspondence to a labeled tree $\theta$. The edges of $\theta$ are represented by dashed lines.}
\label{F: Schaeffer Bijection}
\end{center}
\end{figure}

For a technical reason, we extend this definition to trees $\theta \in \lstrees^{\ast}(1)$ by keeping the same vertices and edges, and choosing $(\sigma_{\theta}(\pt{c}_0), \sigma_{\theta} (\sigma_{\theta}(\pt{c}_0)))$ as the root. (Thus the root edge of $\Phi(\theta)$ always goes from vertices with labels $0$ and $-1$ in $\theta$.) For all $\theta \in \lstrees^{\ast}(1)$, we still have $\Phi(\theta) \in \textbf{Q}$.

\subsubsection{Uniform infinite labeled tree and quadrangulation} \label{S: UIPQ}

For all $x \in \Z$, let $\GWgeom{x}$ be the law of a Galton--Watson tree with offspring distribution $\geom(1/2)$, such that the root has label $x$ and, for any vertex $\pt{v}$ other than the root, the label of $\pt{v}$ is uniform in $\{\ell-1,\ell,\ell+1\}$, with $\ell$ the label of its parent. The uniform infinite labeled tree is the random variable $\theta_{\infty} = (T_{\infty}, l_{\infty}) \in \lstrees(0)$ whose distribution is characterized by the following properties:
\begin{itemize}
\item the process of the spine-labels $(S_i(\theta_{\infty}))_{i \geq 0} := (l_{\infty}(\spine_i(T_{\infty})))_{i \geq 0}$ is a random walk with independent uniform steps in $\{-1,0,1\}$,
\item conditionally on $(S_i(\theta_{\infty}))_{i \geq 0}$, the trees $L_i(\theta_{\infty})$ and $R_i(\theta_{\infty})$ are independent labeled trees distributed according to $\GWgeom{S_i}$.
\end{itemize}
For all $n \in \N$, we also let $\theta_n = (T_n,l_n)$ be a uniform random element of $\ltrees{n}(0)$. It is known that $\theta_n$ converges to $\theta_{\infty}$ for the local limit topology, as $n \rightarrow \infty$ (as noted in \cite{CuMenMi}, it is a consequence of \cite[Lemma 1.14]{Kes}). Note that we have $\theta_{\infty} \in \lstrees^{\ast}(0)$ almost surely, and let $Q_{\infty} := \Phi(\theta_{\infty})$.

It was shown in \cite{CuMenMi} that the UIPQ can be seen as the random quadrangulation $\widetilde{Q}_{\infty}$ equal to $Q_{\infty}$ with probability $1/2$, and to the quadrangulation obtained by reversing the root edge of $Q_{\infty}$ with probability $1/2$.

\subsection{Re-rooting the UIPQ at the $k$-th point on the leftmost geodesic ray} \label{S: Leftmost geodesic}

Let us first clarify what we mean by the leftmost geodesic originating form the root in the UIPQ. It is known from \cite{CuMenMi} that for all vertices $\pt{u},\pt{v}$ of $\widetilde{Q}_{\infty}$, the quantity
\begin{gather*}
\lim_{\pt{w} \rightarrow \infty} \(d_{\widetilde{Q}_{\infty}}(\pt{u},\pt{w}) - d_{\widetilde{Q}_{\infty}}(\pt{v},\pt{w})\)
\end{gather*}
is well defined (in the sense that the difference of those distances is the same except for a finite number of vertices $\pt{w}$), and equal to the difference of the labels of $\pt{u}$ and $\pt{v}$ in the corresponding tree. As a consequence, letting $e$ denote the root edge of $Q_{\infty}$, with $e^-$ its origin and $e^+$ its other extremity, we have
\begin{gather*}
\lim_{\pt{w} \rightarrow \infty} \(d_{\widetilde{Q}_{\infty}}(e^-,\pt{w}) - d_{\widetilde{Q}_{\infty}}(e^+,\pt{w})\) = 1.
\end{gather*}
In other words, the extremity of the root edge of $\widetilde{Q}_{\infty}$ which is ``closest to infinity'' is well defined, and equal to $e^+$. Therefore, it is natural to say that the leftmost geodesic ray started from the root in $\widetilde{Q}_{\infty}$ is the unique path $\gamma_L=(\gamma_L(i))_{i \geq 0}$ such that $\gamma_L(0)=e^-$, $\gamma_L(1)=e^+$ and for all $i \geq 1$, $\gamma_L(i+1)$ is the first neighbour of $\gamma_L(i)$ after $\gamma_L(i-1)$ (in the clockwise order) such that
\begin{gather*}
\lim_{\pt{w} \rightarrow \infty} \(d_{\widetilde{Q}_{\infty}}(\gamma_L(i),\pt{w}) - d_{\widetilde{Q}_{\infty}}(\gamma_L(i+1),\pt{w})\) = 1.
\end{gather*}
Note that the definition of the leftmost geodesic ray does not depend on whether the root edge of $\widetilde{Q}_{\infty}$ has the same orientation as that of $Q_{\infty}$ or not, so it is sufficient to work with $Q_{\infty}$ in the rest of the article.

The leftmost geodesic also has a natural definition in terms of the tree $\theta_{\infty}$. For all $k \geq 0$, let $\pt{e}_k$ be the $k$-th corner on the chain of the iterated successors of $\pt{e}_0$, where $\pt{e}_0$ is the root corner of $\theta_{\infty}$. Equivalently, $\pt{e}_k$ can be seen as the first corner with label $-k$ after the root, in the clockwise order. We use the same notation for the corresponding vertex in $Q_{\infty}$. The path $\gamma_{\max} := (\pt{e}_k)_{k \geq 0}$ is a geodesic ray in $Q_{\infty}$, called the maximal geodesic in \cite{CuMenMi}, and equal to $\gamma_L$.

Curien, Ménard and Miermont proved in \cite{CuMenMi} that all other geodesic rays from $\pt{e}_0$ to infinity are essentially similar to $\gamma_{\max}$: almost surely, there exists an infinite sequence of distinct vertices of $Q_{\infty}$ such that every geodesic ray from $\pt{e}_0$ to infinity passes through all these vertices. Our main goal is to study the local limit of $Q_\infty^{(k)}$ as $k \rightarrow \infty$, where $Q_\infty^{(k)}$ denotes the quadrangulation $Q_{\infty}$ re-rooted at $(\pt{e}_k,\pt{e}_{k+1})$.

More precisely, we will study what the quadrangulation looks like on the left and on the right of the geodesic ray $\gamma_{\max}$. This leads us to introduce the ``split'' quadrangulation $\spq(Q_{\infty})$ obtained by ``cutting'' $Q_{\infty}$ along $\gamma_{\max}$; formally, $\spq(Q_{\infty})$ is an infinite quadrangulation of the (lower) half-plane whose boundary is formed by the edges $(\pt{e}_k,\pt{e}_{k+1})$ on the left of $\pt{e}_0$, and by copies $(\pt{e}'_k,\pt{e}'_{k+1})$ of these edges on the right of $\pt{e}_0$. This construction is illustrated in Figure \ref{F: Split UIPQ}. For all $k \geq 0$, we let $\ora{Q}_{\infty}^{(k)}$ denote the quadrangulation having the same vertices and edges as $\spq(Q_{\infty})$, with root $(\pt{e}_k,\pt{e}_{k+1})$, and $\ola{Q}_{\infty}^{(k)}$ denote the quadrangulation having the same vertices and edges as $\spq(Q_{\infty})$, with root $(\pt{e}'_k,\pt{e}'_{k+1})$. Thus, since $(\pt{e}_k)_{k \geq 0}$ and $(\pt{e}'_k)_{k \geq 0}$ are geodesics in $\ora{Q}_{\infty}^{(k)}$ and $\ola{Q}_{\infty}^{(k)}$, we have the following property:

\begin{lem}
For all $r \leq k$, the ball of radius $r$ in $Q_\infty^{(k)}$ is the same as the union of the balls of radius $r$ in $\ora{Q}_{\infty}^{(k)}$ and $\ola{Q}_{\infty}^{(k)}$.
\end{lem}

The main idea now consists in studying the limit of the trees encoding $\ora{Q}_{\infty}^{(k)}$ and $\ola{Q}_{\infty}^{(k)}$, and then going back to the associated quadrangulations.

\begin{figure}[!t]
\begin{center}
\includegraphics{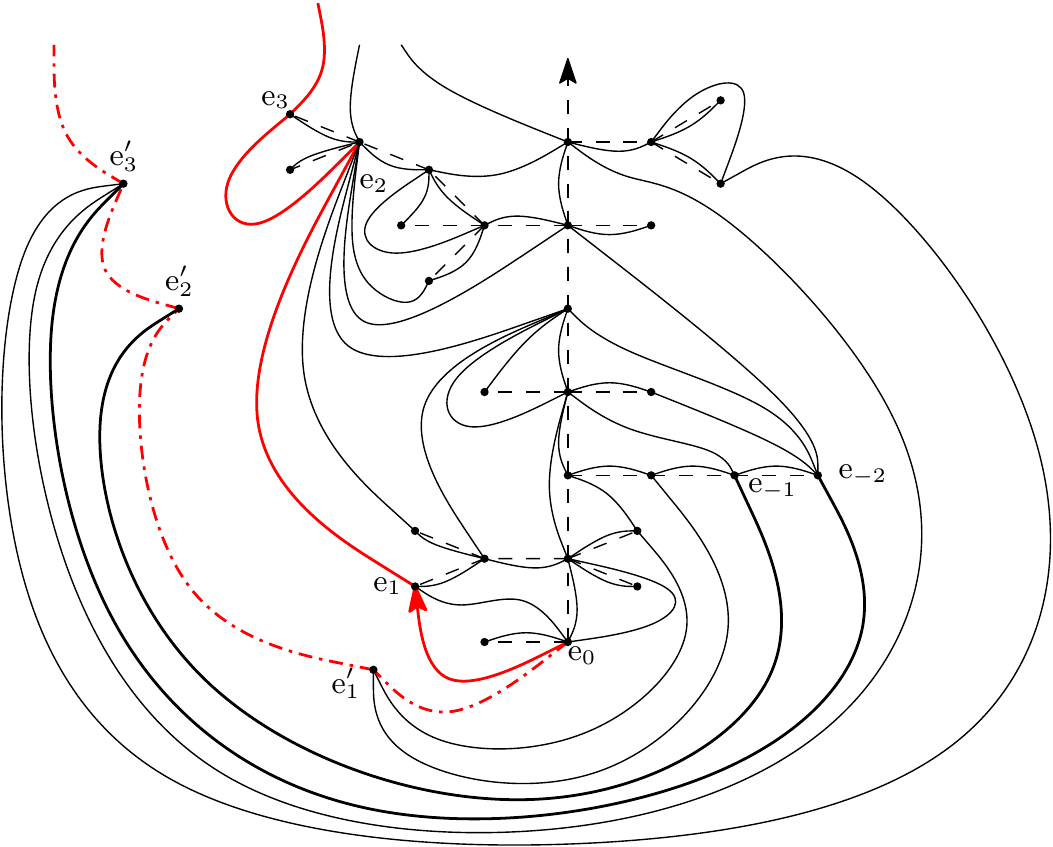}
\caption{The ``split'' quadrangulation $\spq(Q_{\infty})$ obtained from $\theta_{\infty}$. The edges of the underlying tree $\theta_{\infty}$ are represented in dashed lines, and the geodesic ray $\gamma_{\max}$ is represented in red. The labels are omitted to keep the figure readable.}
\label{F: Split UIPQ}
\end{center}
\end{figure}

To this end, for all $k \in \N$, we introduce the tree $\theta_{\infty}^{(k)} = (T_{\infty}^{(k)}, l_{\infty}^{(k)})$, where $T_{\infty}^{(k)}$ is the tree $T_{\infty}$ re-rooted at $\pt{e}_k$, and $l_{\infty}^{(k)}:= l_{\infty} + k$. Note that the vertices $\pt{e}'_k$, $k \geq 1$, contrary to the $\pt{e}_k$, do not correspond to corners of the tree $\theta_{\infty}$. Therefore, for all $k \in \N$, we let $\pt{e}_{-k+1}$ denote the last corner of $\theta_{\infty}$ before the root (still in the clockwise order) such that $\sigma_{\theta_{\infty}} (\pt{e}_{-k+1}) = \pt{e}_k$. Equivalently, $\pt{e}_{-k+1}$ can be seen as the last corner with label $-k+1$ before the root (hence the choice of the index). Now, for all $k \in \N$, we let $\theta_{\infty}^{(-k+1)} = (T_{\infty}^{(-k+1)}, l_{\infty}^{(-k+1)})$, where $T_{\infty}^{(-k+1)}$ is the tree $T_{\infty}$ re-rooted at $\pt{e}_{-k+1}$, and $l_{\infty}^{(-k+1)}:= l_{\infty} + k$. With this notation, for all $k \in \N$, we have $\theta_{\infty}^{(k)} \in \lstrees^{\ast}(0)$, $\theta_{\infty}^{(-k+1)} \in \lstrees^{\ast}(1)$, $\Phi(\theta_{\infty}^{(k)}) = \ora{Q}_{\infty}^{(k)}$ and $\Phi(\theta_{\infty}^{(-k+1)})= \ola{Q}_{\infty}^{(k)}$; but more importantly, we will show in Section \ref{S: Joint CV of the quadrangulations} that the local limits of $\ora{Q}_{\infty}^{(k)}$ and $\ola{Q}_{\infty}^{(k)}$ can be determined using the local limits of $\theta_{\infty}^{(k)}$ and $\theta_{\infty}^{(-k+1)}$.

Intuitively, one can anticipate that the local limit of $\theta_{\infty}^{(k)}$ will be a tree in which the right-hand side only has positive labels, and the local limit of $\theta_{\infty}^{(-k+1)}$ will be a tree in which the left-hand side only has labels greater than $1$. This leads us to extend the domain of $\Phi$ to such trees.

\subsection{Extending the Schaeffer correspondence} \label{S: Extended Schaeffer}

Consider the following subsets of $\lstrees$:
\begin{gather*}
\ora{\lstrees} = \{ (T,l) \in \lstrees(0): \min_{n \leq -1} l(\pt{c}_n(T)) = 1, \lim_{n \rightarrow -\infty} l(\pt{c}_n(T)) = +\infty \mbox{ and } \inf_{n \geq 0} l(\pt{c}_n(T)) = -\infty \} \\
\ola{\lstrees} = \{ (T,l) \in \lstrees(1): \min_{n \geq 1} l(\pt{c}_n(T)) = 2, \lim_{n \rightarrow +\infty} l(\pt{c}_n(T)) = +\infty \mbox{ and } \inf_{n \leq 0} l(\pt{c}_n(T)) = -\infty \}.
\end{gather*}
Here, we show that ``Schaeffer-type'' constructions yield natural associations between  the trees in these sets and quadrangulations of the lower and upper half-planes. Examples of quadrangulations obtained this way are given on Figure \ref{F: extended Schaeffer}.

In the case where $\theta \in \ora{\lstrees}$, the construction is exactly the same as for $\theta \in \lstrees^{\ast}(0)$: for all $n$, we define the successor $\sigma_{\theta}(\pt{c}_n)$ of $\pt{c}_n$ as the first corner among $\pt{c}_{n+1}, \pt{c}_{n+2}, \ldots$ such that
\begin{gather*}
l(\sigma_{\theta}(\pt{c}_n)) = l(\pt{c}_n)-1,
\end{gather*}
and we let $\Phi(\theta)$ denote the graph whose set of vertices is $V(T)$, whose edges are the pairs $\{\pt{c},\sigma_{\theta}(\pt{c})\}$ for all corners $\pt{c}$ of $T$, and whose root-edge is $(\pt{c}_0,\sigma_{\theta}(\pt{c}_0))$.

Now, consider the case where $\theta \in \ola{\lstrees}$. If we use the above construction, then for example, for all $i$, the last corner with label $i$ has no successor. We therefore add a ``shuttle'' $\Lambda$, i.e. a line of new points $\lambda_i$, $i \in \Z$ on which the corners with no successor will be attached. More precisely, for all $n$, the successor of $\pt{c}_n$ is defined as
\begin{gather*}
\sigma_{\theta}(\pt{c}_n) = \left\lbrace \begin{array}{ll}
\pt{c}_{n'} & \mbox{for the smallest } n'\geq n \mbox{ such that } l(\pt{c}_{n'}) = l(\pt{c}_n)-1, \mbox{ if it exists,} \\
\lambda_{l(\pt{c}_n)-1} & \mbox{otherwise,}
\end{array}\right.
\end{gather*}
and we extend this notation to the points of $\Lambda$ by letting $\sigma_{\theta} (\lambda_i) = \lambda_{i-1}$ for all $i \in \Z$.  We let $\Phi(\theta)$ be the graph whose set of vertices is $V(T) \sqcup \Lambda$, whose edges are the pairs $\{\pt{c},\sigma_{\theta}(\pt{c})\}$ for all corners $\pt{c}$ of $T$, and the pairs $\{\lambda_i,\lambda_{i-1}\}$ for all $i \in \Z$, and whose root-edge is $(\lambda_0,\lambda_{-1})$. (Note that the rooting convention is consistent with the one we used to define $\Phi$ on $\lstrees^{\ast}(1)$.)

\begin{figure}[t]
\begin{center}
\includegraphics{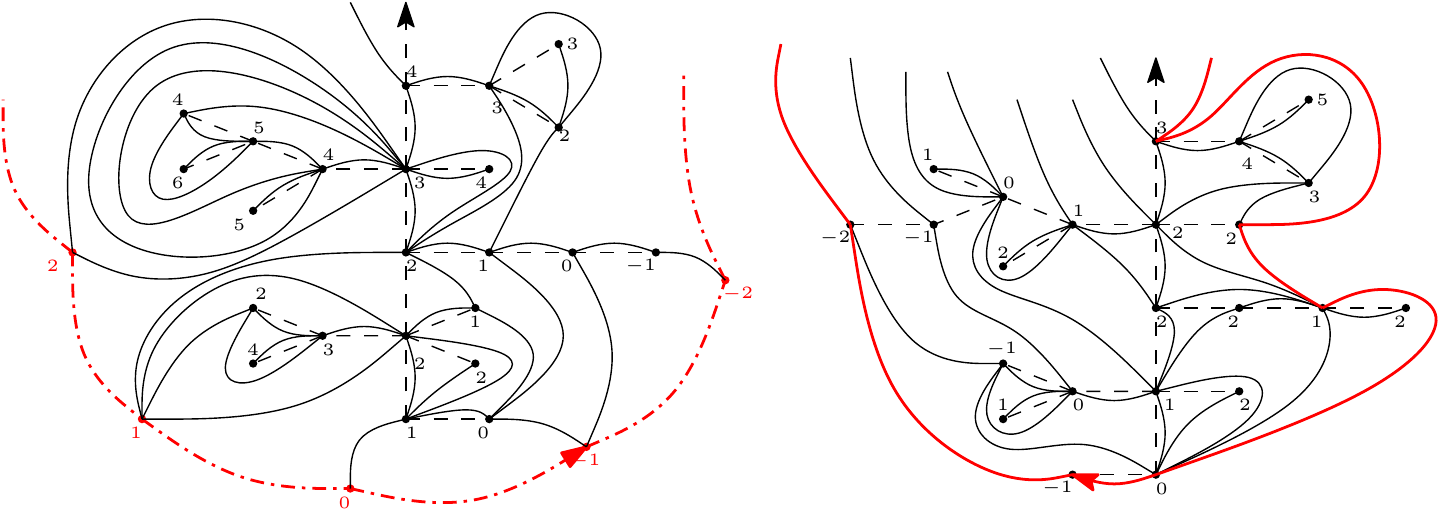}
\caption{Examples of quadrangulations $\Phi(\theta)$ obtained for $\theta \in \protect\ola{\lstrees}$ (on the left-hand side) and $\theta \in \protect\ora{\lstrees}$ (on the right-hand side).}
\label{F: extended Schaeffer}
\end{center}
\end{figure}

\begin{lem}
We have the following properties:
\begin{itemize}
\item If $\theta \in \ora{\lstrees}$, then $\Phi(\theta) \in \lhpq$.
\item If $\theta \in \ola{\lstrees}$, then $\Phi(\theta) \in \uhpq$.
\end{itemize}
\end{lem}

\begin{proof}
In both cases, it is clear that the graph $\Phi(\theta)$ has a natural embedding into the plane, and the conditions on $\liminf_{n \rightarrow -\infty} l(\pt{c}_n)$ ensure that every corner is the successor of a finite number of other corners. Thus every vertex of $\Phi(\theta)$ has finite degree: $\Phi(\theta)$ is an infinite planar map.

As in Schaeffer's usual construction, a simple case study shows that for every corner $\pt{c}$ of $\theta$:
\begin{itemize}
\item The face which is on the right of $(\pt{c}, \sigma_{\theta}(\pt{c}))$ is a quadrangle.
\item If there exists a corner $\pt{c}' < \pt{c}$ such that $l(\pt{c}')=l(\pt{c})$, then the face which is on the left of $(\pt{c}, \sigma_{\theta}(\pt{c}))$ is a quadrangle. If $\theta \in \ola{\lstrees}$, this is always true. If $\theta \in \ora{\lstrees}$, then the only corners for which it is not true are the $\pt{c}_{n_i}$, $i \in \Z$, with $n_i = \min \{ n \in \Z: l(\pt{c}_n)=i \}$. For all $i$, we have $\sigma(\pt{c}_{n_i})=\pt{c}_{n_{i-1}}$, and the face which is on the left of $(\pt{c}_{n_i},\pt{c}_{n_{i-1}})$ is the root face of $\Phi(\theta)$.
\end{itemize}
For $\theta \in \ola{\lstrees}$, we also have to study the faces which are on the left and on the right of the edges $(\lambda_i, \lambda_{i-1})$: we easily see that the first one is always a quadrangle, and that the second one is the same for all $i$. Thus:
\begin{itemize}
\item For all $\theta \in \ora{\lstrees}$, we have $\Phi(\theta) \in \mathcal{Q}_{\infty,\infty}$.
\item For all $\theta \in \ola{\lstrees}$, letting $\ol{\Phi(\theta)}$ denote the map obtained by reversing the root edge of $\Phi(\theta)$, we have $\ol{\Phi(\theta)} \in \mathcal{Q}_{\infty,\infty}$.
\end{itemize}
Note that the construction ensures that the classical bound
\begin{gather} \label{E: Lower bound on d_Phi(.)}
d_{\Phi(\theta)}(\pt{u},\pt{v}) \geq \abs{l(\pt{u},\pt{v})}
\end{gather}
still holds. As a consequence, in both cases, the boundary is a geodesic path. Moreover, the fact that $\theta$ has exactly one spine implies, by construction, that $\Phi(\theta)$ is one-ended.
\end{proof}

Note that for $\theta \in \ora{\lstrees}$, for all $i<i'$, the path $(\lambda_j)_{i \leq j \leq i'}$ is the unique geodesic between $\lambda_i$ and $\lambda_{i'}$. Indeed, all neighbours of $\lambda_{i'}$ different of $\lambda_{i'-1}$ have labels equal to $i+1$, so they are at distance (at least) $i'-i+1$ from $\lambda_i$. In other words, the boundary is the \emph{unique} geodesic path between vertices of $\Lambda$.

\subsection{Main results} \label{S: Results}

The first part of our work is the identification of the limit of the joint distribution of $(\theta_{\infty}^{(k)}, \theta_{\infty}^{(-k+1)})$ as $k \rightarrow \infty$. We begin by using the convergence of $\theta_n$ towards $\theta_{\infty}$ to give an explicit description of this joint distribution.

To give a more precise idea of these results, we adapt the notation of Section \ref{S: Leftmost geodesic} to possibly finite trees. For all $\theta = (T,l) \in \ltrees{}(0)$ and $k \geq 0$ such that $\min_{V(T)} l \leq -k$, let $\pt{e}_k (\theta)$ be the first corner having label $-k$ after the root, in clockwise order, $\pt{e}_{-k} (\theta)$ be the last corner having label $-k$ before the root, and $\pt{v}_k(\theta)$ be the most recent common ancestor of $\pt{e}_k(\theta)$ and $\pt{e}_{-k+1}(\theta)$. Note that for $k=0$, this is well defined since $\pt{e}_0 (\theta) = \pt{e}_{-0} (\theta)$. Finally, we define the finite analogs of $\theta_{\infty}^{(k)}$ and $\theta_{\infty}^{(-k+1)}$: conditionally on $\min_{V(T_n)} l_n \leq - \abs{k}$, for all $n,k \in \N$, we let
\begin{itemize}
\item $\theta_n^{(k)} = (T_n^{(k)}, l_n^{(k)})$, where $T_n^{(k)}$ is the tree $T_n$ re-rooted at $\pt{e}_k (\theta_n)$, and $l_n^{(k)}= l_n + k$,
\item $\theta_n^{(-k+1)} = (T_n^{(-k+1)}, l_n^{(-k+1)})$, where $T_n^{(-k+1)}$ is the tree $T_n$ re-rooted at $\pt{e}_{-k+1} (\theta_n)$, and $l_n^{(-k+1)}= l_n + k$.
\end{itemize}
It is easy to see that:

\begin{lem}
We have the joint convergence in distribution
\begin{gather} \label{E: Convergence of theta_n,k in n}
(\theta_n^{(k)},\theta_n^{(-k+1)}) \cvg{n \rightarrow \infty}{} (\theta_{\infty}^{(k)},\theta_{\infty}^{(-k+1)})
\end{gather}
for the local limit topology.
\end{lem}

Indeed, the operations which consist in re-rooting a tree $\theta \in \ltrees{}$ at $\pt{e}_k(\theta)$ and $\pt{e}_{-k+1}(\theta)$ are both continuous for the local limit topology on $\lstrees^{\ast}(0)$. Since $\theta_{\infty}$ belongs to this set, this yields the conclusion. This lemma will allow us to give an explicit description of the joint distribution of $\theta_{\infty}^{(k)}$ and $\theta_{\infty}^{(-k+1)}$ (see Proposition \ref{T: Distribution of theta(infty,k)} for the distribution of $\theta_{\infty}^{(k)}$ alone, and Corollary \ref{T: Distribution of theta(infty,k) with two marked points} for the joint distribution).

We use these results to prove the convergence theorem below. Recall that $\GWgeom{x}$ denotes the distribution of a Galton--Watson tree with $\geom(1/2)$ offspring distribution and ``uniform'' labels, with root label $x$. If $x$ is positive, we let $\GWgeom{x}^+$ denote the same distribution, conditioned to have only positive labels. We also introduce a Markov chain $\tilde{X}$ taking values in $\N$, with transition probabilities
\begin{gather*}
p_x := \P{\tilde{X}_1 = x+1 | \tilde{X}_0=x} = \frac{(x+4)(2x+5)}{3(x+2)(2x+3)} \\
r_x := \P{\tilde{X}_1 = x | \tilde{X}_0=x} = \frac{x(x+3)}{3(x+1)(x+2)} \\
q_x := \P{\tilde{X}_1 = x-1 | \tilde{X}_0=x} = \frac{(x-1)(2x+1)}{3(x+1)(2x+3)}.
\end{gather*}
Note that $\tilde{X}$ can be seen as a discrete version of a seven-dimensional Bessel process. Indeed, a theorem of Lamperti \cite{Lamp} shows that, under some easily checked conditions, the rescaled process $((1/\sqrt{n}) \cdot \tilde{X}_{\ent{nt}})_{t \geq 0}$ converges in distribution to a diffusion process with generator
\begin{gather*}
L = \frac{\alpha}{x}\frac{d}{dx} + \frac{\beta}{2} \frac{d^2}{dx^2},
\end{gather*}
where
\begin{gather*}
\alpha = \lim_{x \rightarrow \infty} x \E{\tilde{X}_1-\tilde{X}_0 | \tilde{X}_0=x} = 2
\end{gather*}
and
\begin{gather*}
\beta = \lim_{x \rightarrow \infty} \E{(\tilde{X}_1-\tilde{X}_0)^2 | \tilde{X}_0=x} =  \frac{2}{3},
\end{gather*}
hence in our case
\begin{gather*}
L = \frac{2}{3} \(\frac{3}{x}\frac{d}{dx} + \frac{1}{2} \frac{d^2}{dx^2}\).
\end{gather*}
Thus, $((1/\sqrt{n}) \cdot \tilde{X}_{\ent{nt}})_{t \geq 0}$ converges to $(Z_{2t/3})_{t \geq 0}$, where $Z$ denotes a Bessel(7) process started from 0.

\begin{thm} \label{T: Joint cv of the rerooted trees}
We have the joint convergence in distribution
\begin{gather}
(\theta_{\infty}^{(k)},\theta_{\infty}^{(-k+1)}) \cvg{k \rightarrow \infty}{} (\ora{\theta_{\infty}},\ola{\theta_{\infty}})
\end{gather}
for the local topology, where $\ora{\theta_{\infty}} = (\ora{T_{\infty}}, \ora{l_{\infty}})$ and $\ola{\theta_{\infty}} = (\ola{T_{\infty}}, \ola{l_{\infty}})$ are independent random variables in $\lstrees(0)$ and $\lstrees(1)$, whose distributions are characterized by the following properties:
\begin{itemize}
\item The process $(S_i(\ora{\theta_{\infty}}))_{i \geq 1}$ has the same law as the Markov chain $\tilde{X}$ started from $1$.
\item Conditionally on $(S_i(\ora{\theta_{\infty}}))_{i \geq 0}$, the subtrees $L_i(\ora{\theta_{\infty}})$, $i \geq 0$ and $R_i(\ora{\theta_{\infty}})$, $i \geq 1$ are independent random variables, with respective distributions $\GWgeom{S_i(\ora{\theta_{\infty}})}$ and $\GWgeom{S_i(\ora{\theta_{\infty}})}^+$.
\item We have the joint distributional identities:
\begin{gather*}
(S_i(\ola{\theta_{\infty}})-1)_{i \geq 0} = (S_i(\ora{\theta_{\infty}}))_{i \geq 0} \\
(L_i(\ola{T_{\infty}}), \ola{l_{\infty}}-1)_{i \geq 0} = (R_i(\ora{\theta_{\infty}}))_{i \geq 0} \\
(R_i(\ola{T_{\infty}}), \ola{l_{\infty}}-1)_{i \geq 0} = (L_i(\ora{\theta_{\infty}}))_{i \geq 0}.
\end{gather*}
\end{itemize}
\end{thm}

We finally extend this convergence to the associated quadrangulations:
\begin{thm} \label{T: Joint CV of the quadrangulations}
Let $\ora{Q}_{\infty} = \Phi(\ora{\theta_{\infty}})$ and $\ola{Q}_{\infty} = \Phi(\ola{\theta_{\infty}})$. We have the joint convergence in distribution
\begin{gather*}
(\ora{Q}_{\infty}^{(k)}, \ola{Q}_{\infty}^{(k)}) \cvg{k \rightarrow \infty}{} (\ora{Q}_{\infty}, \ola{Q}_{\infty})
\end{gather*}
for the local topology. As a consequence, $Q_{\infty}^{(k)}$ converges in distribution towards the quadrangulation of the plane $\olra{Q}_{\infty}$ obtained by gluing together the boundaries of $\ora{Q}_{\infty}$ and $\ola{Q}_{\infty}$ in such a way that their root edges are identified.
\end{thm}

Note that $\Phi$ is not continuous at points $\ora{Q}_{\infty}$ and $\ola{Q}_{\infty}$, so this result is not a straightforward consequence of the previous theorem. In the same spirit as Ménard in \cite{Men}, we have to show that the balls of radius $r$ in $\ora{Q}_{\infty}$ and $\ola{Q}_{\infty}$ are included into balls of radius $h(r)$ in the corresponding trees with high probability, uniformly in $k$. This is done in Proposition \ref{T: Ball inclusions}.

The distribution of $\olra{Q}_{\infty}$ could be the subject of further study, in particular concerning its symmetries. Informally, it would be interesting to see if it is invariant under the two following transformations:
\begin{itemize}
\item Rerooting $\olra{Q}_{\infty}$ at the ``lowest'' edge $e$ belonging to an infinite geodesic $(\gamma(i))_{i \in \Z}$, such that $l(e^-)=0$ and $l(e^+)=1$; then taking the quadrangulation obtained by reflection with respect to the root edge.
\item Rerooting $\olra{Q}_{\infty}$ at the ``lowest'' edge $e$ belonging to an infinite geodesic $(\gamma(i))_{i \in \Z}$, such that $l(e^-)=0$ and $l(e^+)=1$; then reversing the root edge.
\end{itemize}
In the first case, the invariance should be easy to derive from symmetries of the UIPQ. The second question appears more difficult and is work in progress.

The paper is organized in the following way. In Sections \ref{S: First convergence} and \ref{S: Joint CV of the trees}, we focus on the convergence of the trees $\theta_{\infty}^{(k)}$ and $\theta_{\infty}^{(-k+1)}$. We first give the proof of the convergence of $\theta_{\infty}^{(k)}$ alone, and then show how the same methods can be applied to derive the joint convergence. Note that the convergence results of Section \ref{S: First convergence} are not necessary in the proof of the joint convergence, but should make the structure of the proof easier to understand. Finally, Section \ref{S: Joint CV of the quadrangulations} is devoted to the proof of Theorem \ref{T: Joint CV of the quadrangulations}.

\paragraph{Acknowledgements:} This work is part of a larger project in collaboration with Grégory Miermont and Erich Baur. The author would like to thank them for the inspiring discussions and their careful proofreadings, as well as Nicolas Curien, whose idea it first was to study this transformation of the UIPQ.

\section{Convergence of $\theta_{\infty}^{(k)}$} \label{S: First convergence}

\subsection{Explicit expressions for the distribution of $\theta_{\infty}^{(k)}$} \label{S: First dist eqns}

In this section, we work with a fixed value of $k \in \N$. Let us introduce some notation for particular vertices and subtrees of $\theta_n^{(k)}$, for $n \in \N \cup \{\infty\}$. All the variables we consider also depend on $k$, and should therefore be denoted with an exponent $^{(k)}$, but we omit it as long as $k$ is fixed, to keep the notation readable. First, let $m_n$ be the graph-distance between $\pt{e}_0(\theta_n)$ and $\pt{e}_k(\theta_n)$, and $\pt{x}_{n,0}, \ldots, \pt{x}_{n,m_n}$ denote the sequence of the vertices which appear on the path from $\pt{e}_k(\theta_n)$ to $\pt{e}_0(\theta_n)$. For all  $i \in \{0,\ldots, m_n\}$, let $X_{n,i} = l_n^{(k)} (\pt{x}_{n,i})$. We also consider the subtrees which appear on each ``side'' of the path $(\pt{x}_{n,0}, \ldots, \pt{x}_{n,m_n})$:
\begin{itemize}
\item For all $i \in \{1, \ldots, m_n\}$, let $\tau_{n,i}$ be the subtree of $\theta_n^{(k)}$ containing the vertices $\pt{v}$ such that in $\theta_n$, we had $\pt{x}_{n,i} \leq \pt{v} < \pt{x}_{n,i-1}$.
\item For all $i \in \{0,\ldots, m_n\}$, let $\tau'_{n,i}$ be the subtree of $\theta_n^{(k)}$ containing the vertices $\pt{v}$ such that in $\theta_n$, we had $\pt{v}=\pt{x}_{n,i}$, or $\pt{x}_{n,i} \prec \pt{v}$, $\pt{x}_{n,i-1} < \pt{v}$ and $\pt{x}_{n,i-1} \nprec \pt{v}$.
\end{itemize}
We emphasize that these subtrees inherit the labels $l_n^{(k)}$ instead of $l_n$, even if we have to use the orders $<$ and $\prec$ on $T_n$ (instead of $T_n^{(k)}$) to define them. The fact that we have to use these orders may seem a bit clumsy since the subtrees are numbered starting from the root $\pt{x}_{n,0}$ of $\theta_n^{(k)}$, but it is necessary to get the distinction between $\tau_{n,m_n}$ and $\tau'_{n,m_n}$. Figure \ref{F: Notation on theta_n,k} sums up the above notation.

\begin{figure}[b]
\begin{center}
\includegraphics{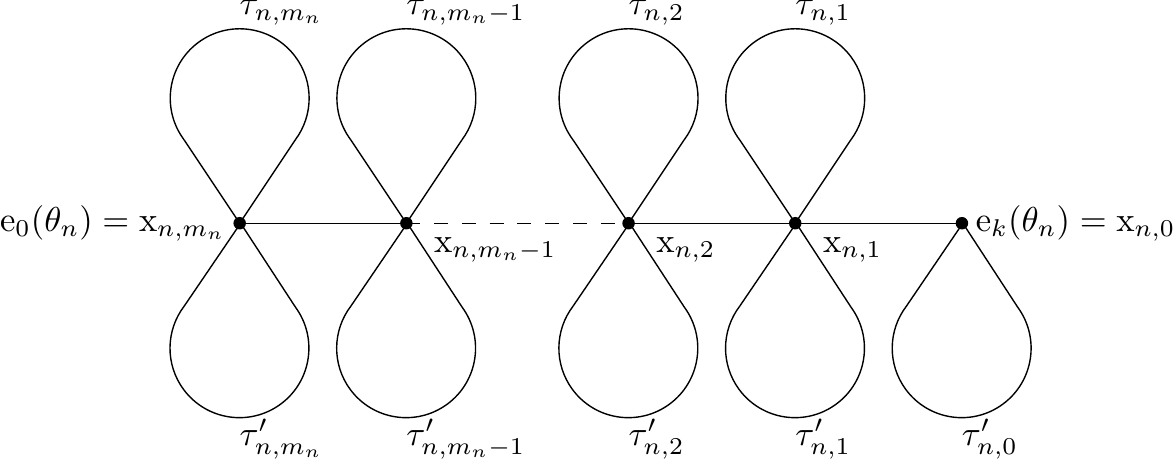}
\caption{Notation for the vertices and subtrees of $\theta_n^{(k)}$.}
\label{F: Notation on theta_n,k}
\end{center}
\end{figure}

Our first step is to characterize the joint distribution of $m_{\infty}$, $(X_{\infty,i})_{0 \leq i \leq m_{\infty}}$, $(\tau_{\infty,i})_{1 \leq i \leq m_{\infty}}$ and $(\tau'_{\infty,i})_{0 \leq i \leq m_{\infty}}$. We introduce some more notation for the sets in which these random variables take their values. For all $m,x,x' \in \N$, let $\mathcal{M}^+_{m, x \rightarrow x'}$ denote the set of the walks $(x_1, \ldots, x_m) \in \N^m$ such that $x_1=x$, $x_m = x'$ and for all $i \leq m-1$, $\abs{x_{i+1} - x_i} \leq 1$. Also let
\begin{gather*}
\ltrees{n}^+(x) = \{(T,l) \in \ltrees{n}(x): l>0 \} \qquad \forall x \in \N, \\
\ltrees{n}^+ = \bigcup_{x \in \N} \ltrees{n}^+(x) \qquad \mbox{and} \qquad
\ltrees{}^+ = \bigcup_{n \geq 0} \ltrees{n}^+.
\end{gather*}
We also use the following facts on the distributions $\GWgeom{x}$ and $\GWgeom{x}^+$: for all $n \geq 0$, it is known that
\begin{gather*}
\GWgeom{x} (\{\theta\}) = \frac{1}{2 \cdot 12^n} \qquad \forall x \in \Z,\ \theta \in \ltrees{n}(x)
\end{gather*}
and Proposition 2.4 of \cite{ChaDu} shows that
\begin{gather} \label{E: w [Cha-Dur]}
\sum_{n \geq 0} \frac{1}{2 \cdot 12^n}\card{\ltrees{n}^+(x)} = w(x) := \frac{x(x+3)}{(x+1)(x+2)} \qquad \forall x \in \N.
\end{gather}
In particular, for all $n \geq 0$ and $x \in \N$, we have $\GWgeom{x} (\ltrees{}^+) = w(x)$ and
\begin{gather*}
\GWgeom{x}^+ (\{\theta\}) = \frac{1}{2 w(x) 12^n} \qquad \forall \theta \in \ltrees{n}^+(x).
\end{gather*}
Finally, for all $m \in \N$ and $(x_0,\ldots,x_m) \in \Z^{m+1}$, we let $\mu_{(x_0,\ldots,x_m)}$ denote the distribution of the forest $(\tilde{\tau}_i)_{0 \leq i \leq m}$ defined as follows. Let $I$ be a uniform random variable in $\{0,\ldots,m\}$. Let $\tilde{\tau}_I$ be a random tree distributed as $(T_{\infty},l_{\infty}+x_I)$, and $\tilde{\tau}_i$, $i \in \{0,\ldots,m\} \setminus \lbrace I \rbrace$ be independent random trees distributed according to $\GWgeom{x_i}$, independent of $\tilde{\tau}_I$.

We can now state the proposition:

\begin{prop} \label{T: Distribution of theta(infty,k)}
We have $X_{\infty,0} = 0\ \as$, and for all $m \in \N$, $\ul{x} \in \mathcal{M}^+_{m, 1 \rightarrow k}$,
\begin{gather*}
\P{m_{\infty}=m, (X_{\infty,1}, \ldots, X_{\infty,m}) = \ul{x}}
= \frac{m+1}{3^m} \prod_{i=1}^m w(x_i).
\end{gather*}
Moreover, conditionally on $m_{\infty}=m$ and $(X_{\infty,1}, \ldots, X_{\infty,m}) = \ul{x}$:
\begin{itemize}
\item The forests $(\tau_{\infty,i})_{1 \leq i \leq m}$ and $(\tau'_{\infty,i})_{0 \leq i \leq m}$ are independent.
\item The trees $\tau_{\infty,i}$, $1 \leq i \leq m$ are independent random variables distributed according to $\GWgeom{x_i}^+$.
\item The forest $(\tau'_{\infty,i})_{0 \leq i \leq m}$ is distributed according to $\mu_{(0,x_1,\ldots,x_m)}$.
\end{itemize}
\end{prop}

The proof of this proposition relies on counting the well-labeled trees with $n$ edges such that the corresponding $m_n$, $(\tau_{n,1},\ldots,\tau_{n,m})$ take a certain value, and using the convergence \eqref{E: Convergence of theta_n,k in n}.

\begin{proof}
We say that a well-labeled forest with $m$ trees is a $m$-tuple of well-labeled plane rooted trees $(t_1,\ldots,t_m)$, such that for all $i \in \{1,\ldots,m-1\}$, the labels of the roots of $t_i$ and $t_{i+1}$ differ by at most $1$. The number of edges of such a forest is the sum of the numbers of edges of the trees $t_1\ldots,t_m$. Let $\lforests{m}{n}$ be the set of well-labeled plane forests with $m$ trees and $n$ edges.

Fix $m,N \geq 0$, $\ul{t} = (t_1,\ldots,t_m) \in \lforests{m}{N}$ such that the root of $t_1$ has label $1$, and all the labels in $\ul{t}$ are positive. For all $n \in \N \cup \{\infty\}$, let
\begin{gather*}
P_n^{(k)} (m,\ul{t}) = \P{m_n=m,(\tau_{n,1}, \ldots, \tau_{n,m})=\ul{t} | \min l_n \leq -k}.
\end{gather*}
We are interested in the behaviour of $P_n^{(k)} (m,\ul{t})$ as $n \rightarrow \infty$, for fixed $k$. Since $\theta_n$ is uniform in $\ltrees{n}(0)$, we have
\begin{gather*}
P_n^{(k)} (m,\ul{t})
= \frac{\mathcal{F}_{m+1,n-(m+N)}}{\card{\{(T,l) \in \ltrees{n}(0): \min l \leq -k\}}},
\end{gather*}
where for all $n' \geq 0$,
\begin{align*}
\mathcal{F}_{m+1,n'}
& = \card{\{ (t'_0,\ldots,t'_m) \in \lforests{m+1}{n'}: \begin{array}{l}
\mbox{the root of } t'_0 \mbox{ has label } 0 \mbox{ and for all } i \geq 1,\\
\mbox{the root of } t'_0 \mbox{ has the same label as the root of } t_i
\end{array} \}} \\
& = \card{\{ (t'_0,\ldots,t'_m) \in \lforests{m+1}{n'}: \mbox{for all } i \geq 1, \mbox{the root of } t'_0 \mbox{ has label } 0 \}}.
\end{align*}
First note that
\begin{gather*}
\card{\(\{(T,l) \in \ltrees{n}(0): \min l \leq -k\}\)} \sim_{n \rightarrow \infty} \card{\ltrees{n}(0)}.
\end{gather*}
Moreover, it can be seen from the well-known cyclic lemma (see \cite{Pit_CSP}) that
\begin{gather} \label{E: Number of ltrees}
\card{\ltrees{n}(0)} = \frac{3^n}{2n+1} \binom{2n+1}{n}
\end{gather}
and
\begin{gather} \label{E: Number of lforests}
\mathcal{F}_{m,n} = \frac{3^n m}{2n+m} \binom{2n+m}{n}.
\end{gather}
Applying these formulas to our case gives
\begin{gather*}
\mathcal{F}_{m+1,n-(m+N)}
= \frac{3^{n-(m+N)} (m+1)}{2n+1-(m+2N)} \binom{2n+1-(m+2N)}{n-(m+N)},
\end{gather*}
and therefore
\begin{gather*}
P_n^{(k)} (m,\ul{t})
\sim_{n \rightarrow \infty} \frac{m+1}{3^{m+N}} \binom{2n+1}{n}^{-1} \binom{2n+1-(m+2N)}{n-(m+N)}.
\end{gather*}
We now use Stirling's formula to get an estimate of the binomial coefficients involved:
\begin{gather*}
\binom{2n+1}{n}
\sim_{n \rightarrow \infty} \frac{2 \cdot 4^n}{\sqrt{\pi n}},
\end{gather*} 
and
\begin{gather*}
\binom{2n+1-(m+2N)}{n-(m+N)} \sim_{n \rightarrow \infty} \frac{4^n}{2^{m+2N-1} \sqrt{\pi n}}.
\end{gather*}
Putting these together, we obtain
\begin{gather*}
P_n^{(k)} (m,\ul{t})
\sim_{n \rightarrow \infty} \frac{m+1}{3^{m+N}2^{m+2N}} = \frac{m+1}{6^m 12^N},
\end{gather*}
so the local convergence \eqref{E: Convergence of theta_n,k in n} implies that
\begin{gather*}
P_{\infty}^{(k)} (m,\ul{t}) = \frac{m+1}{6^m 12^N}.
\end{gather*}
As a consequence, for all $m \in \N$, $\ul{x} \in \mathcal{M}^+_{m, 1 \rightarrow k}$, we have
\begin{gather*}
\P{m_{\infty}=m, (X_{\infty,1}, \ldots, X_{\infty,m}) = \ul{x}}
= \frac{m+1}{6^m} \prod_{i=1}^m \( \sum_{n_i \geq 0} \frac{1}{12^{n_i}} \card{\ltrees{n_i}^+(x_i)} \).
\end{gather*}
Recalling equation \eqref{E: w [Cha-Dur]}, we get
\begin{gather*}
\P{m_{\infty}=m, (X_{\infty,1}, \ldots, X_{\infty,m}) = \ul{x}}
= \frac{m+1}{6^m} \prod_{i=1}^m 2 w(x_i)
= \frac{m+1}{3^m} \prod_{i=1}^m w(x_i).
\end{gather*}
Furthermore, for all $\ul{t} = (t_1,\ldots,t_m) \in \lforests{m}{N}$ such that all the labels in $\ul{t}$ are positive, the conditional probability
\begin{gather*}
\P{(\tau_{\infty,1}, \ldots, \tau_{\infty,m}) = \ul{t}\ |\ m_{\infty}=m, (X_{\infty,1}, \ldots, X_{\infty,m}) = \ul{x}}
\end{gather*}
is equal to
\begin{gather*}
\frac{m+1}{6^m 12^N} \cdot \(\frac{m+1}{3^m} \prod_{i=1}^m w(x_i)\)^{-1}
= \prod_{i=1}^m \frac{1}{2 w(x_i) 12^{|t_i|}}
= \prod_{i=1}^m \GWgeom{x_i}^+ (t_i),
\end{gather*}
hence the conditional distribution of $(\tau_{\infty,1}, \ldots, \tau_{\infty,m_{\infty}})$.

Finally, conditionally on $m_{\infty}=m$, $(X_{\infty,1}, \ldots, X_{\infty,m}) = \ul{x}$ and $(\tau_{\infty,1}, \ldots, \tau_{\infty,m}) = \ul{t} \in \lforests{m}{N}$, the trees $(\tau'_{n,0}, \ldots, \tau'_{n,m})$ form a uniform labeled forest with $m+1$ trees and $n-m-N$ edges, hence the distribution of the limit given in the statement.
\end{proof}

To get the limit of $\theta_{\infty}^{(k)}$, the main step will consist in showing that for any $r \in \N$, the labels $X_{\infty,1}^{(k)},\ldots,X_{\infty,r}^{(k)}$ converge in distribution to the $r$ first steps of the Markov chain $\tilde{X}$ started at $1$, as $k \rightarrow \infty$. For the moment, we show how to make $\tilde{X}$ appear in the above expression; the fact that it is indeed the limit is the purpose of Proposition \ref{T: First labels convergence}.

We first introduce the random walk $(\hat{X}_i)_{i \geq 0}$ with uniform random steps in $\{-1,0,1\}$. From now on, we also adopt the usual notation $\Efrom{x}{\ \cdot\ }$ for the conditional expectation $\E{\ \cdot\ | \hat{X}_0 = x}$, for all $x$. The expression of the lemma implies that
\begin{gather*}
\P{m_{\infty}=m}
= \frac{m+1}{3} \Efrom[sz]{1}{\prod_{i=0}^{m-1} w(\hat{X}_i) \ind{\hat{X}_{m-1}=k}}.
\end{gather*}
(Note that the term in the expectation is zero if we do not have $\hat{X}_i \geq 1$ for all $i \leq m-1$.) Let $f(x) = x(x+3)(2x+3)$ for all $x \in \R$, and
\begin{gather*}
M_j = \frac{f(\hat{X}_j)}{f(\hat{X}_0)} \prod_{i=0}^{j-1} w(\hat{X}_i) \qquad \forall j \geq 0.
\end{gather*}
Under the assumption $\hat{X}_0 =1$, the process $(M_i)_{i \geq 0}$ is a martingale. Using this new process, we get
\begin{align*}
\P{m_{\infty}=m}
& = \frac{m+1}{3} \Efrom[sz]{1}{ \frac{f(1)w(k)}{f(k)} M_{m-1} \ind{\hat{X}_{m-1}=k}} \\
& = \frac{f(1)w(k)}{3f(k)} (m+1) \Pfrom{1}{\tilde{X}_{m-1}=k},
\end{align*}
where $\tilde{X}$ is defined as the image of $\hat{X}$ under the measure-change given by the martingale $M$, i.e. the Markov process such that $\E{\phi(\tilde{X}_i)} = \E{M_i \phi(\hat{X}_i)}$ for every continuous bounded function $\phi$. Computing the transition probabilities of $\tilde{X}$ gives:
\begin{gather*}
p_x = \Pfrom{x}{\tilde{X_1} = x+1} = \frac{f(x+1)w(x)}{3f(x)} \\
r_x = \Pfrom{x}{\tilde{X_1} = x} = \frac{w(x)}{3} \\
q_x = \Pfrom{x}{\tilde{X_1} = x-1} = \frac{f(x-1)w(x)}{3f(x)},
\end{gather*}
hence the expressions given in the Introduction.

\subsection{Two useful quantities}

To prove of the convergence of $\theta_{\infty}^{(k)}$, we will need estimates for the quantities
\begin{gather*}
\SumFW{x}{k} = \sum_{m \geq 1} \Pfrom{x}{\tilde{X}_{m-1}=k} \\
\SumFWStar{x}{k} = \sum_{m \geq 1} (m+1) \Pfrom{x}{\tilde{X}_{m-1}=k},
\end{gather*}
depending on the values of $k, x \in \N$. In practice, these estimates are best obtained through explicit computation; the expressions we get are given in the two following lemmas. We use the notations $\tilde{T}_y = \inf \lbrace t \geq 1: \tilde{X}_t=y \rbrace$ and $h(y)= y(y+1)(y+2)(y+3)(2y+3)$, for all $y \in \N$. (In this section, we mainly work on Markov processes, and use the letters $t$ and $T$ for associated times instead of trees.)

\begin{lem} \label{T: Values of SumFW}
Fix $k \geq 2$, $x \in \N$. We have the following equalities:
\begin{itemize}
\item if $x \leq k$,
\begin{gather*}
\SumFW{x}{k} = \frac{3}{10} (2k+3),
\end{gather*}
\item if $x>k$,
\begin{gather*}
\SumFW{x}{k} = \frac{3 h(k)}{10 h(x)}(2k+3).
\end{gather*}
\end{itemize}
\end{lem}

\begin{proof}
Fix $x,k \in \N$. First note that we can write $\SumFW{x}{k}$ as
\begin{gather*}
\SumFW{x}{k}
= \sum_{m \geq 0} \Pfrom{x}{\tilde{X}_m = k}
= \Efrom[sz]{x}{\sum_{m \geq 0} \ind{\tilde{X}_m=k}}
= \ind{x=k} + \Efrom[sz]{x}{\ind{\tilde{T}_k < \infty} \sum_{m \geq 0} \ind{\tilde{X}_{\tilde{T}_k+m}=k}}.
\end{gather*}
Now, applying the Markov property at the stopping time $\tilde{T}_k$ yields
\begin{gather*}
\SumFW{x}{k}
= \ind{x=k} + \Pfrom{x}{\tilde{T}_k < \infty} \Efrom[sz]{k}{ \sum_{m \geq 0} \ind{\tilde{X}_m=k}}.
\end{gather*}
For all $y \geq 0$, let
\begin{gather*}
K_{y+1,j} = \frac{q_{y+1} \ldots q_{y+j}}{p_{y+1} \ldots p_{y+j}} = \prod_{z=y+1}^{y+j} \frac{f(z-1)}{f(z+1)} = \frac{f(y)f(y+1)}{f(y+j)f(y+j+1)}.
\end{gather*}
Since $K_{1,j}$ is the general term of a converging series, the Markov chain $\tilde{X}$ is transient, and as a consequence, we have 
\begin{gather} \label{E: SumFW_x(k) using T_k}
\SumFW{x}{k}
= \ind{x=k} + \frac{\Pfrom{x}{\tilde{T}_k < \infty}}{\Pfrom{k}{\tilde{T}_k = \infty}}.
\end{gather}
To compute these quantities, it is enough to know the expression of $\Pfrom{y+1}{\tilde{T}_y = \infty}$ for all $y \geq k$, which is a well-known property of birth-and-death processes:
\begin{gather*}
\Pfrom{y+1}{\tilde{T}_y = \infty} = \frac{1}{\sum_{j \geq 0} K_{y+1,j}}
\end{gather*}
Computing the sum $\sum_{j \geq 0} K_{y+1,j}$ yields
\begin{gather} \label{E: P(no return to y from y+1)}
\Pfrom{y+1}{\tilde{T}_y = \infty} = \frac{10 (y+2)}{(y+4)(2y+5)}.
\end{gather}
As a consequence, we get the following results: 
\begin{itemize}
\item If $x<k$, then
\begin{gather*}
\Pfrom{x}{\tilde{T}_k < \infty} = 1.
\end{gather*}
\item If $x=k$, then
\begin{gather*}
\Pfrom{x}{\tilde{T}_k < \infty}
= 1-p_k \Pfrom{k+1}{\tilde{T}_k = \infty}
= \frac{6k-1}{3(2k+3)}.
\end{gather*}
\item If $x>k$, then
\begin{gather*}
\Pfrom{x}{\tilde{T}_k < \infty}
= \prod_{y=k}^{x-1} \Pfrom{y+1}{\tilde{T}_y < \infty}
= \frac{h(k)}{h(x)}.
\end{gather*}
\end{itemize}
Together with \eqref{E: SumFW_x(k) using T_k}, this completes the proof of the lemma.
\end{proof}

Note that the values we obtain can also be computed using the recurrence relations
\begin{gather} \label{E: Rec SumFW_1(k)}
\left\lbrace \begin{array}{l}
\SumFW{1}{1} = 1 + r_1 \SumFW{1}{1} + q_2 \SumFW{1}{2} \\
\SumFW{1}{k} = p_{k-1} \SumFW{1}{k-1} + r_k \SumFW{1}{k} + q_{k+1} \SumFW{1}{k+1} \qquad \forall k \geq 2
\end{array} \right.
\end{gather}
and, for all $k \in \N$,
\begin{gather} \label{E: Rec SumFW_x(k) in x}
\left\lbrace \begin{array}{l}
\SumFW{1}{k} = \ind{k=1} + r_1 \SumFW{1}{k} + p_1 \SumFW{2}{k} \\
\SumFW{x}{k} = \ind{k=x} + p_x \SumFW{x+1}{k} + r_x \SumFW{x}{k} + q_x \SumFW{x-1}{k} \qquad \forall x \geq 2,
\end{array} \right.
\end{gather}
which stem from the Markov property of $\tilde{X}$. Nevertheless, we would still have to go through part of the previous calculations to get the value of $\SumFW{1}{1}$. In the proof of the following lemma, we will find it easier to use this approach.

\begin{lem} \label{T: Values of SumFWStar}
Fix $k \geq 2$, $x \in \N$, and let $C_x = \frac{3}{14}((x+1)(x+2)-6)$. We have the following equalities:
\begin{itemize}
\item if $x<k$,
\begin{gather*}
\SumFWStar{x}{k} = \frac{3f(k)}{f(1)w(k)} - \frac{3 C_x}{10} (2k+3),
\end{gather*}
\item if $x \geq k$,
\begin{gather*}
\SumFWStar{x}{k} = \frac{3f(k)}{f(1)w(k)} - \frac{3}{10} (2k+3)\(C_x + 1-\frac{h(k)}{h(x)}\).
\end{gather*}
\end{itemize}
\end{lem}


\begin{proof}
The first step of the proof consists in computing $\SumFWStar{1}{1}$. We will then obtain $\SumFWStar{x}{k}$ as the unique solution of recursive systems having this initial value. Note that since $\Pfrom{1}{\tilde{X}_0=1}=1$, we have
\begin{gather*}
\SumFWStar{1}{1}
= 2 + \sum_{m \geq 1} (m+2) \Pfrom{1}{\tilde{X}_m = 1}.
\end{gather*}
Let us rewrite the second term using the first return time in 1, as in the proof of the previous lemma:
\begin{align*}
\SumFWStar{1}{1}
&= 2 + \sum_{t \geq 1} \Pfrom{1}{\tilde{T}_1 = t} \sum_{m \geq t} (m+2) \Pfrom{1}{\tilde{X}_m=1 \mid \tilde{T}_1 = t} \\
&= 2 + \sum_{t \geq 1} \Pfrom{1}{\tilde{T}_1 = t} \sum_{m \geq 0} (m+t+2) \Pfrom{1}{\tilde{X}_m=1} \\
&= 2 + \SumFWStar{1}{1} \sum_{t \geq 1} \Pfrom{1}{\tilde{T}_1 = t} + \SumFW{1}{1} \sum_{t \geq 1} t \Pfrom{1}{\tilde{T}_1 = t} \\
&= 2 + \SumFWStar{1}{1} \Pfrom{1}{\tilde{T}_1 < \infty} + \SumFW{1}{1} \Efrom[sz]{1}{\tilde{T}_1 \ind{\tilde{T}_1 < \infty}}.
\end{align*}
Thus, we have
\begin{gather*}
\SumFWStar{1}{1}
= \frac{1}{\Pfrom{1}{\tilde{T_1} = \infty}} \( 2 + \frac{3}{2} \Efrom{1}{\tilde{T}_1 \mid \tilde{T}_1 < \infty} \Pfrom{1}{\tilde{T}_1 < \infty} \).
\end{gather*}
Using the value of $\Pfrom{1}{\tilde{T}_1 < \infty}$ obtained in the previous proof, we get
\begin{gather}
\SumFWStar{1}{1}
= \frac{3}{2} \( 2 + \frac{1}{2} \Efrom{1}{\tilde{T}_1 \mid \tilde{T}_1 < \infty}\) \label{E: SumFWStar_1(1)}.
\end{gather}
To work out the value of the above expectation, we study the process $\tilde{X}^{\ast}$ having the law of $\tilde{X}$ conditioned on returning to $1$ infinitely often. This process is a recurrent Markov chain whose transition probabilities can be computed explicitly. Indeed, letting
\begin{gather*}
p^{\ast}_x
:= \Pfrom{x}{\tilde{X}_1 = x+1 \mid \tilde{T}_1 < \infty} \\
r^{\ast}_x
:= \Pfrom{x}{\tilde{X}_1 = x \mid \tilde{T}_1 < \infty} \\
q^{\ast}_x
:= \Pfrom{x}{\tilde{X}_1 = x-1 \mid \tilde{T}_1 < \infty},
\end{gather*}
Bayes' law yields
\begin{align*}
p^{\ast}_x
&= \frac{\Pfrom{x}{\tilde{X}_1 = x+1} \Pfrom{x}{\tilde{T_1} < \infty \mid \tilde{X}_1 = x+1}}{\Pfrom{x}{\tilde{T}_1 < \infty}}
= \frac{p_x \Pfrom{x+1}{\tilde{T}_1 < \infty}}{\Pfrom{x}{\tilde{T}_1 < \infty}}, \\
r^{\ast}_x
&= \frac{\Pfrom{x}{\tilde{X}_1 = x} \Pfrom{x}{\tilde{T_1} < \infty \mid \tilde{X}_1 = x}}{\Pfrom{x}{\tilde{T}_1 < \infty}}
= \left\lbrace \begin{array}{l}
r_x \mbox{ if } x \neq 1 \\
\frac{r_1}{\Pfrom{1}{\tilde{T}_1 < \infty}} \mbox{ if } x=1,
\end{array} \right. \\
q^{\ast}_x
&= \frac{\Pfrom{x}{\tilde{X}_1 = x-1} \Pfrom{x}{\tilde{T_1} < \infty \mid \tilde{X}_1 = x-1}}{\Pfrom{x}{\tilde{T}_1 < \infty}}
= \left\lbrace \begin{array}{l}
\frac{q_x \Pfrom{x-1}{\tilde{T}_1 < \infty}}{\Pfrom{x}{\tilde{T}_1 < \infty}} \mbox{ if } x \neq 2 \\
\frac{q_2}{\Pfrom{2}{\tilde{T}_1 < \infty}} \mbox{ if } x=1.
\end{array} \right.
\end{align*}
Note that, for all $x \geq 2$,
\begin{gather*}
\Pfrom{x+1}{\tilde{T}_1 < \infty} = \Pfrom{x+1}{\tilde{T}_x < \infty} \Pfrom{x}{\tilde{T}_1 < \infty},
\end{gather*}
so we can again use equation \eqref{E: P(no return to y from y+1)}. Finally, we get $p^{\ast}_1 = \frac{1}{3}$, $r^{\ast}_1 = \frac{2}{3}$, $q^{\ast}_1 = 0$, and for all $x \geq 2$
\begin{gather*}
p^{\ast}_x = \frac{x}{3(x+2)} \\
r^{\ast}_x = \frac{x(x+3)}{3(x+1)(x+2)} \\
q^{\ast}_x = \frac{x+3}{3(x+1)}.
\end{gather*}
To get the value of $\Efrom{1}{\tilde{T}_1 \mid \tilde{T}_1 < \infty}$, it is now enough to compute the invariant measure $\Pi$ of $\tilde{X}^{\ast}$. We do so by using reversibility: the detailed balanced equation $\Pi(x) p^{\ast}_x = \Pi(x+1) q^{\ast}_{x+1}$ implies
\begin{gather*}
\frac{\Pi(x+1)}{\Pi(x)} = \left\lbrace \begin{array}{l}
\frac{x}{x+4} \mbox{ if } x \geq 2 \\
\frac{3}{5} \mbox{ if } x = 1.
\end{array}\right.
\end{gather*}
As a consequence,
\begin{gather*}
\sum_{x \geq 1} \Pi(x) = \Pi(1) \( 1 + \frac{3}{5} \sum_{x \geq 2} \frac{2 \times 3 \times 4 \times 5}{x(x+1)(x+2)(x+3)} \) = \Pi(1) \( 1 + \frac{3}{5} \times 120 \times \frac{1}{72} \),
\end{gather*}
so $\Pi$ is a probability measure if and only if $\Pi(1) = \frac{1}{2}$. This implies
\begin{gather*}
\Efrom{1}{\tilde{T}_1 \mid \tilde{T}_1 < \infty} = \frac{1}{\Pi(1)} = 2.
\end{gather*}
Injecting this value into \eqref{E: SumFWStar_1(1)} gives $\SumFWStar{1}{1} = \frac{9}{2}$.

For the second step, we keep $x=1$, and compute the values of $\SumFWStar{1}{k}$ for $k \in \N$. As above, we first shift indices and set the first term aside:
\begin{gather*}
\SumFWStar{1}{k}
= 2 \ind{k=1} + \sum_{m \geq 0} (m+3) \Pfrom{1}{\tilde{X}_{m+1}=k}.
\end{gather*}
Applying the Markov property at time $m$ in each of the terms gives the following recurrence relations:
\begin{itemize}
\item For $k=1$,
\begin{align*}
\SumFWStar{1}{1}
&= 2 + \sum_{m \geq 0} (m+3) \(r_1 \Pfrom{1}{\tilde{X}_m=1} + q_2 \Pfrom{1}{\tilde{X}_m=2}\) \\
&= 2 + r_1 \SumFWStar{1}{1} + q_2 \SumFWStar{1}{2} + \SumFW{1}{1}-1.
\end{align*}
\item For all $k \geq 2$,
\begin{align*}
\SumFWStar{1}{k}
&= \sum_{m \geq 0} (m+3) \(p_{k-1}\Pfrom{1}{\tilde{X}_m=k-1} + r_k\Pfrom{1}{\tilde{X}_m=k} + q_{k+1}\Pfrom{1}{\tilde{X}_m=k+1}\) \\
&= p_{k-1} (\SumFWStar{1}{k-1} + \SumFW{1}{k-1}) + r_k (\SumFWStar{1}{k} + \SumFW{1}{k}) + q_{k+1} (\SumFWStar{1}{k+1} + \SumFW{1}{k+1}) \\
&= p_{k-1} \SumFWStar{1}{k-1} + r_k \SumFWStar{1}{k} + q_{k+1} \SumFWStar{1}{k+1} + \SumFW{1}{k}.
\end{align*}
\end{itemize}
(Note that we have used implicitly the fact that $\SumFW{1}{k}$ verifies the similar system \eqref{E: Rec SumFW_1(k)}). Using the values obtained in Lemma \ref{T: Values of SumFW}, we get the recursive system
\begin{gather*}
\left\lbrace
\begin{array}{l}
\SumFWStar{1}{1} = \frac{9}{2} \\
\SumFWStar{1}{2} = \frac{1}{q_2} \( (1-r_1)\SumFWStar{1}{1} - \frac{5}{2} \) \\
\SumFWStar{1}{k+1} = \frac{1}{q_{k+1}} \( (1-r_k)\SumFWStar{1}{k} - p_{k-1} \SumFWStar{1}{k-1} - \frac{3}{10} (2k+3) \).
\end{array}
\right.
\end{gather*}
It is now easy to check that $\frac{3 f(k)}{f(1)w(k)}$ is also a solution of this system, and therefore is equal to $\SumFWStar{1}{k}$.

In the third and last step, we fix the value of $k$, and write recurrence relations for $\SumFWStar{x}{k}$, $x \in \N$. To this end, we again use the Markov property, but at time 1 (with the convention that $\SumFWStar{0}{k}=0$, to keep the setting general):
\begin{align*}
\SumFWStar{x}{k}
&= 2 \ind{x=k} + \sum_{m \geq 0} (m+3) \Pfrom{x}{\tilde{X}_{m+1}=k} \\
&= 2 \ind{x=k} + \sum_{m \geq 0} (m+3) (p_x \Pfrom{x+1}{\tilde{X}_m=k} + r_x \Pfrom{x}{\tilde{X}_{m+1}=k} + q_x \Pfrom{x-1}{\tilde{X}_m=k}) \\
&= 2 \ind{x=k} + p_x \SumFWStar{x+1}{k} + r_x \SumFWStar{x}{k} + q_x \SumFWStar{x-1}{k} + p_x \SumFW{x+1}{k} + r_x \SumFW{x}{k} + q_x \SumFW{x-1}{k} \\
&= \ind{x=k} + p_x \SumFWStar{x+1}{k} + r_x \SumFWStar{x}{k} + q_x \SumFWStar{x-1}{k} + \SumFW{x}{k}.
\end{align*}
This gives the system
\begin{gather*}
\left\lbrace \begin{array}{l}
\SumFWStar{1}{k} = \frac{3 f(k)}{f(1)w(k)} \\
\SumFWStar{x+1}{k} = \frac{1}{p_x} \( (1-r_x)\SumFWStar{x}{k} - q_x \SumFWStar{x-1}{k} - \SumFW{x}{k} - \ind{x=k}\).
\end{array} \right.
\end{gather*}
We first solve these equations for $x < k$, so that the last term is zero. The solution is of the form given in the lemma if and only if $C_x$ is such that
\begin{gather*}
\left\lbrace \begin{array}{l}
C_0 = C_1 = 0 \\
C_{x+1} = \frac{1}{p_x} ((1-r_x) C_x - q_x C_{x-1} + 1).
\end{array} \right.
\end{gather*}
This is indeed the case for $C_x = \frac{3}{14}((x+1)(x+2)-6)$. Now, for $x \geq k$, we seek a solution of the form
\begin{gather*}
\SumFWStar{x}{k} = \frac{3f(k)}{f(1)w(k)} - \frac{3 C_x}{10} (2k+3) - C'_{k,x}.
\end{gather*}
The recursive system can be translated into $C'_{k,k} = 0$, $C'_{k,k+1} = \frac{1}{p_x}$ and
\begin{gather*}
C'_{k,x+1} = \frac{1}{p_x} ((1-r_x)C'_{k,x} - q_x C'_{k,x-1}),
\end{gather*}
or equivalently
\begin{gather*}
p_x(C'_{k,x+1}-C'_{k,x}) = q_x (C'_{k,x}-C'_{k,x-1}).
\end{gather*}
Thus, for $x\geq k+1$, we get
\begin{align*}
C'_{k,x}
&= \sum_{y=k}^{x-1} \frac{q_{k+1} \ldots q_y}{p_{k+1} \ldots p_y} \frac{1}{p_k} \\
&= \sum_{y=k}^{x-1} \frac{f(k)f(k+1)}{f(y)f(y+1)} \frac{1}{p_k} \\
&= \frac{f(k)f(k+1)}{p_k} \frac{h(x)-h(k)}{10 h(k)h(x)}.
\end{align*}
Using the expressions of $f$, $h$ and $p_k$, we conclude that
\begin{align*}
&= \frac{(k+4)(2k+5)}{10 (k+2) p_k} \( 1-\frac{h(k)}{h(x)} \) \\
&= \frac{3(2k+3)}{10} \( 1-\frac{h(k)}{h(x)} \).
\end{align*}
This ends the proof.
\end{proof}

\subsection{Proof of the convergence}

We are now ready to give the proof of the convergence of $\theta_{\infty}^{(k)}$. We begin with the convergence of the labels $(X_{\infty,i}^{(k)})$ towards the Markov chain $\tilde{X}$.

\begin{prop} \label{T: First labels convergence}
Fix $r \in \N$. For any continuous bounded function $F$ from $\R^r$ into $\R$, we have
\begin{gather*}
\E{F(X_{\infty,1}^{(k)},\ldots,X_{\infty,r}^{(k)})} \cvg{k \rightarrow \infty}{} \Efrom{1}{F(\tilde{X}_0,\ldots,\tilde{X}_{r-1})}.
\end{gather*}
\end{prop}

\begin{proof}
Let $k \geq r$. The computations of Section \ref{S: First dist eqns} show that
\begin{gather*}
\E{F(X_{\infty,1}^{(k)},\ldots,X_{\infty,r}^{(k)})}
= \sum_{m \geq 1} \frac{m+1}{3} \Efrom[sz]{1}{F(\hat{X}_0,\ldots,\hat{X}_{r-1}) \ind{\hat{X}_{m-1}=k} M_{m-1} \frac{f(1)w(k)}{f(k)}}.
\end{gather*}
Since $k \geq r$, the term $\ind{\hat{X}_{m-1}=k}$ is zero for $m < r$. Applying the Markov property allows us to write $\E{F(X_{\infty,1}^{(k)},\ldots,X_{\infty,r}^{(k)})}$ as
\begin{gather*}
\sum_{m \geq r} \frac{m+1}{3} \Efrom[sz]{1}{F(\hat{X}_0,\ldots,\hat{X}_{r-1}) \frac{f(1)}{f(\hat{X}_{r-1})} M_{r-1} \Efrom[sz]{\hat{X}_{r-1}}{\ind{\hat{X}'_{m-r}=k} M'_{m-r} \frac{f(\hat{X}_{r-1})w(k)}{f(k)}}},
\end{gather*}
where $\hat{X}'$ is an independent copy of the process $\hat{X}$, and for all $j \in \N$
\begin{gather*}
M'_j = \frac{f(\hat{X}'_j)}{f(\hat{X}'_0)} \prod_{i=0}^{j-1} w(\hat{X}'_i).
\end{gather*}
Therefore, $\E{F(X_{\infty,1}^{(k)},\ldots,X_{\infty,r}^{(k)})}$ is equal to
\begin{gather*}
\Efrom[sz]{1}{M_{r-1} F(\hat{X}_0,\ldots,\hat{X}_{r-1}) \frac{f(1) w(k)}{3 f(k)} \sum_{m \geq 0} (m+r+1) \Efrom[sz]{\hat{X}_{r-1}}{\ind{\hat{X}'_m=k} M'_m}}.
\end{gather*}
Since $\hat{X}_{r-1} \leq r$ $\as$, we now have to estimate the sum $\sum_{m \geq 0} (m+r+1) \Efrom[sz]{x}{\ind{\hat{X}'_m=k} M'_m}$, for all $x \leq r$. We first express this quantity using $\SumFWStar{x}{k}$ and $\SumFW{x}{k}$:
\begin{align*}
\sum_{m \geq 0} (m+r+1) \Efrom[sz]{x}{\ind{\hat{X}'_m=k} M'_m}
&= \sum_{m \geq 0} (m+r+1) \Pfrom{x}{\tilde{X}'_m=k} \\
&= \SumFWStar{x}{k} + (r-1) \SumFW{x}{k}.
\end{align*}
Now, the results of Lemmas \ref{T: Values of SumFW} and \ref{T: Values of SumFWStar} yield
\begin{gather*}
\sum_{m \geq 0} (m+r+1) \Efrom[sz]{x}{\ind{\hat{X}'_m=k} M'_m}
= \frac{3f(k)}{f(1)w(k)} + (r-1-C_x) \frac{3}{10} (2k+3).
\end{gather*}
As a consequence, we have
\begin{gather*}
\frac{f(1) w(k)}{3 f(k)} \sum_{m \geq 0} (m+r+1) \Efrom[sz]{x}{\ind{\hat{X}'_m=k} M'_m} \cvg{k \rightarrow \infty}{} 1
\end{gather*}
uniformly in $x \leq r$, hence the result.
\end{proof}

Note that we have only used part of the results of Lemmas \ref{T: Values of SumFW} and \ref{T: Values of SumFWStar} (namely, the case where $x \leq k$). The remaining expressions will play a role in the proof of the joint convergence.

The convergence of $\theta_{\infty}^{(k)}$ towards $\ora{\theta_{\infty}}$ can now be obtained by putting together the results of Proposition \ref{T: Distribution of theta(infty,k)} and Proposition \ref{T: First labels convergence}. Indeed, letting $I^{(k)}$ denote the unique index $i$ such that $\tau'_{\infty,i}$ is infinite, conditionally on $I^{(k)} \geq r$, we have that:
\begin{itemize}
\item The points $\spine_i(\theta_{\infty}^{(k)})$ and $\pt{x}_{\infty,i}^{(k)}$ are the same for all $i \leq r$, hence the equalities $R_i(\theta_{\infty}^{(k)}) = \tau_{\infty,i}^{(k)}$ and $L_i(\theta_{\infty}^{(k)}) = (\tau'_{\infty,i})^{(k)}$ for all $i < r$.
\item As a consequence, $(S_i(\theta_{\infty}^{(k)}))_{1 \leq i \leq r}$ converges in distribution to $(\tilde{X}_i)_{0 \leq i \leq r-1}$ for $\tilde{X}_0=1$.
\item Conditionally on $(S_i(\theta_{\infty}^{(k)}))_{0 \leq i < r}$, the subtrees $L_i(\theta_{\infty}^{(k)})$, $0 \leq i < r$ and $R_i(\theta_{\infty}^{(k)})$, $1 \leq i < r$ are independent random variables, with respective distributions $\GWgeom{S_i(\theta_{\infty}^{(k)})}$ and $\GWgeom{S_i(\theta_{\infty}^{(k)})}^+$.
\end{itemize}
Since
\begin{gather*}
\P{I^{(k)} < r} = \E[sz]{\frac{r}{m_{\infty}^{(k)}+1}} \leq \frac{r}{k+1} \cvg{k \rightarrow \infty}{} 0,
\end{gather*}
this gives the desired convergence.

\section{Joint convergence of $(\theta_{\infty}^{(k)},\theta_{\infty}^{(-k+1)})$} \label{S: Joint CV of the trees}

\subsection{Explicit expressions for the joint distribution} \label{S: Bi-marked theta(n)}

As in the previous section, we first fix $k$, and use the convergence of $(\theta_n^{(k)},\theta_n^{(-k+1)})$ to study $(\theta_{\infty}^{(k)},\theta_{\infty}^{(-k+1)})$. Let $n \in \N \cup \{\infty\}$. We introduce some new notation, summed-up in Figure \ref{F: Notation on theta_n bi-rerooted}. To simplify what follows, we write $\pt{e}_0$, $\pt{e}_k$, $\pt{e}_{-k+1}$ and $\pt{v}_k$ instead of $\pt{e}_0(\theta_n)$, $\pt{e}_k(\theta_n)$, $\pt{e}_{-k+1}(\theta_n)$ and $\pt{v}_k(\theta_n)$.

We first deal with the branches between $\pt{e}_0$, $\pt{e}_k$ and $\pt{e}_{-k+1}$. Let $a_n = d_n (\pt{v}_k, \pt{e}_k)$, $b_n = d_n (\pt{v}_k, \pt{e}_{-k+1})$ and $c_n = d_n (\pt{e}_0, \pt{v}_k)$, where $d_n$ denotes the graph-distance on $\theta_n$. Let $\pt{x}_{n,0}, \ldots, \pt{x}_{n,a_n}$ be the vertices on the path from $\pt{e}_k$ to $\pt{v}_k$, $\pt{y}_{n,0}, \ldots, \pt{y}_{n,b_n}$ the ones on the path from $\pt{e}_{-k+1}$ to $\pt{v}_k$, and $\pt{z}_{n,0}, \ldots, \pt{z}_{n,c_n}$ the ones on the path from $\pt{v}_k$ to $\pt{e}_0$. For the corresponding labels, we use capital letters: $X_{n,i} = l_n^{(k)} (\pt{x}_{n,i})$, $Y_{n,i} = l_n^{(k)} (\pt{y}_{n,i})$ and $Z_{n,i} = l_n^{(k)} (\pt{z}_{n,i})$ for all $i$.

We now add notation for the subtrees which are grafted on these branches. Again, we use the orders $\prec$ and $<$ on the vertices of $\theta_n$ in these definitions, even if we think of these trees as subtrees of $\theta_n^{(k)}$ (in particular, they inherit the labels $l_n^{(k)}$).
\begin{itemize}
\item For all $i \in \{1, \ldots, a_n+c_n\}$, let $\tau_{n,i}$ be the subtree containing the vertices $\pt{v}$ such that :
\begin{itemize}
\item if $i \leq a_n$, then $\pt{x}_{n,i} \leq \pt{v} < \pt{x}_{n,i-1}$,
\item if $a_n+1 \leq i \leq a_n+c_n$, then $\pt{z}_{n,i-a_n} \leq \pt{v} < \pt{z}_{n,i-a_n-1}$.
\end{itemize}
\item For all $i \in \{1,\ldots, b_n+c_n\}$, let $\tau'_{n,i}$ be the subtree containing the vertices $\pt{v}$ such that :
\begin{itemize}
\item if $i \leq b_n$, either $\pt{v}=\pt{y}_{n,i}$, or $\pt{y}_{n,i} \prec \pt{v}$, $\pt{y}_{n,i-1} < \pt{v}$ and $\pt{y}_{n,i-1} \nprec \pt{v}$,
\item if $b_n+1 \leq i \leq b_n + c_n$, either $\pt{v}=\pt{z}_{n,i-b_n}$, or $\pt{z}_{n,i-b_n} \prec \pt{v}$, $\pt{z}_{n,i-b_n-1} < \pt{v}$ and $\pt{z}_{n,i-b_n-1} \nprec \pt{v}$.
\end{itemize}
\item For all $i \in \{0,\ldots,a_n\}$, let $\ol{\tau}_{n,i}$ be the subtree containing the vertices $\pt{v}$ such that :
\begin{itemize}
\item if $i=0$, then $\pt{x}_0 \preceq \pt{v}$,
\item otherwise, either $\pt{v}=\pt{x}_{n,i}$, or $\pt{x}_{n,i} \prec \pt{v}$, $\pt{x}_{n,i-1} < \pt{v}$ and $\pt{x}_{n,i-1} \nprec \pt{v}$.
\end{itemize}
\item For all $i \in \{0,\ldots,b_n-1\}$, let $\ol{\tau}'_{n,i}$ be the subtree containing the vertices $\pt{v}$ such that :
\begin{itemize}
\item if $i = 0$, then $\pt{y}_0 \preceq \pt{v}$.
\item otherwise, $\pt{y}_{n,i} \leq \pt{v} < \pt{y}_{n,i-1}$,
\end{itemize}
\end{itemize}
As in section \ref{S: First dist eqns}, for all these variables, there should be an exponent $^{(k)}$ in the notation, but we omit this precision as long as $k$ remains constant.

\begin{figure}[t]
\begin{center}
\includegraphics{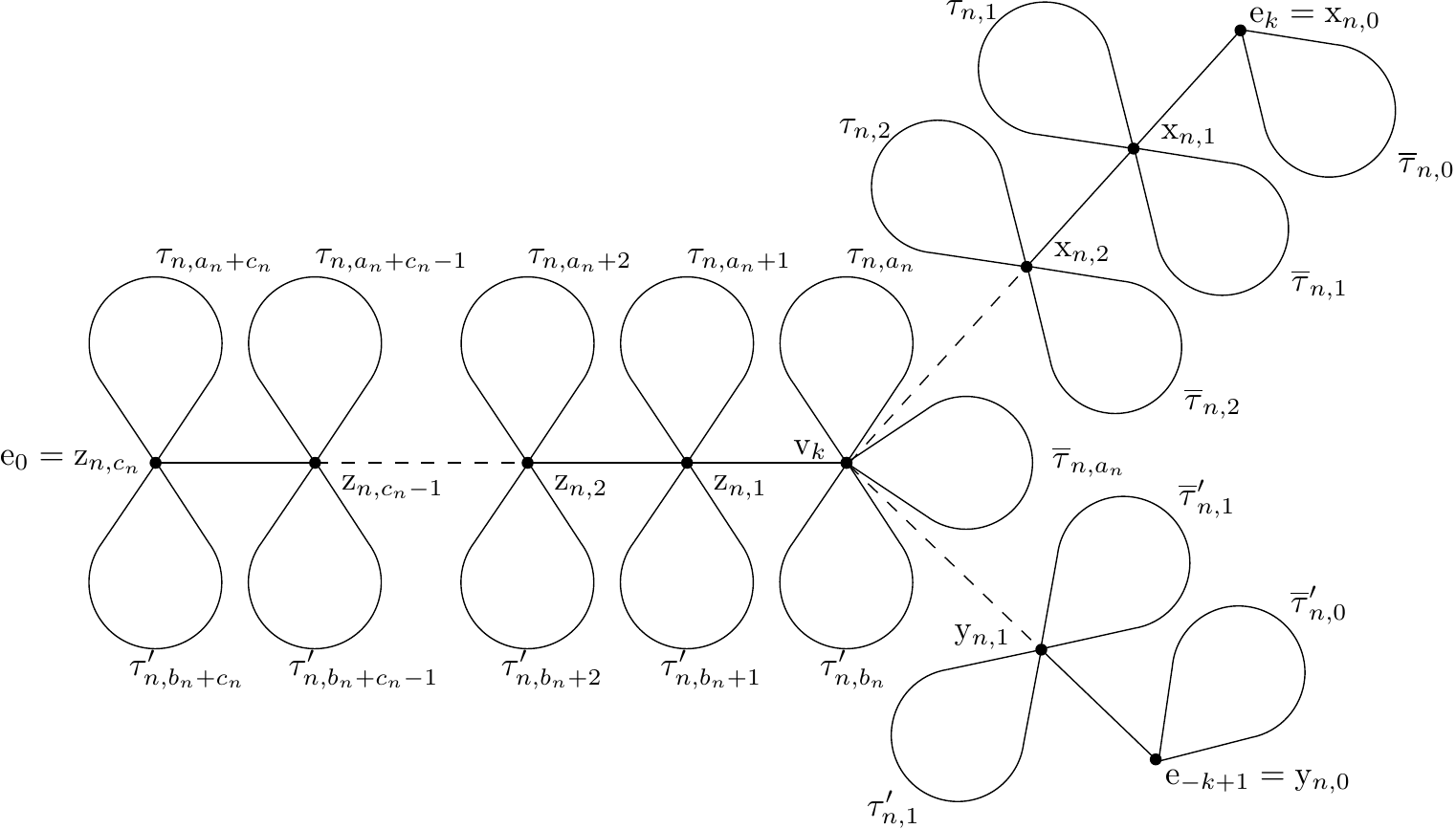}
\caption{Notation for the vertices and subtrees of $\theta_n$, with distinguished points $\pt{e}_k$ and $\pt{e}_{-k+1}$.}
\label{F: Notation on theta_n bi-rerooted}
\end{center}
\end{figure}

Fix $a,b,c,N,N' \geq 0$, $\ul{t} = (t_1,\ldots,t_{a+c}) \in \lforests{a+c}{N}$ and $\ul{t'} = (t'_1,\ldots,t'_{b+c}) \in \lforests{b+c}{N'}$ such that:
\begin{itemize}
\item the root of $t_1$ has label $1$, and all labels in $\ul{t}$ are positive,
\item the root of $t'_1$ has label $2$, and all labels in $\ul{t'}$ are greater than $1$,
\item for all $i \leq c$, the labels of the roots of $t_{a+c-i}$ and $t'_{b+c-i}$ are the same.
\end{itemize}
Let
\begin{gather*}
P_n^{(k)} (a,b,c,\ul{t},\ul{t'}) = \P{a_n=a,b_n=b,c_n=c,(\tau_1, \ldots, \tau_{a+c})=\ul{t}, (\tau'_1, \ldots, \tau'_{b+c})=\ul{t'}}.
\end{gather*}
We are once again interested in the behaviour of $P_n^{(k)} (a,b,c,\ul{t},\ul{t'})$ as $n \rightarrow \infty$, for fixed $k$.

\begin{lem}
Using the above notation, we have
\begin{gather*}
P_n^{(k)} (a,b,c,\ul{t},\ul{t'})
\sim_{n \rightarrow \infty} \frac{a+b+1}{6^{a+b} 12^{c+N+N'}}.
\end{gather*}
\end{lem}

\begin{proof}
Recall that $\mathcal{F}_{m,n'}$ denotes the number of well-labeled forests with $m$ trees, $n'$ edges and prescribed root labels. Since $\theta_n$ in uniform in $\ltrees{n}(0)$, we have
\begin{gather*}
P_n^{(k)} (a,b,c,\ul{t},\ul{t'})
= \frac{\mathcal{F}_{a+b+1,n-(a+b+c+N+N')}}{\card{\{(T,l) \in \ltrees{n}(0): \min_{V(T)} l \leq -k\}}}.
\end{gather*}
Using equations \eqref{E: Number of ltrees} and \eqref{E: Number of lforests} yields
\begin{align*}
\mathcal{F}_{a+b+1,n-(a+b+c+N+N')}
= &\frac{3^{n-(a+b+c+N+N')} (a+b+1)}{2n+1-(a+b+2c+2N+2N')} \\ & \qquad \binom{2n+1-(a+b+2c+2N+2N')}{n-(a+b+c+N+N')},
\end{align*}
and
\begin{gather*}
P_n^{(k)} (a,b,c,\ul{t},\ul{t'})
\sim_{n \rightarrow \infty} \frac{a+b+1}{3^{a+b+c+N+N'}} \binom{2n+1}{n}^{-1} \binom{2n+1-(a+b+2c+2N+2N')}{n-(a+b+c+N+N')}.
\end{gather*}
Stirling's formula now gives
\begin{gather*}
P_n^{(k)} (a,b,c,\ul{t},\ul{t'})
\sim_{n \rightarrow \infty} \frac{a+b+1}{3^{a+b+c+N+N'}2^{a+b+2c+2N+2N'}},
\end{gather*}
hence the lemma.
\end{proof}

Recall that for all $m,x,x' \in \N$, $\ltrees{m}^+(x)$ is the set of the labeled trees $(T,l) \in \ltrees{m}$ such that $l>0$ and the root of $T$ has label $x$, and $\mathcal{M}^+_{m, x \rightarrow x'}$ is the set of the walks $(x_1, \ldots, x_m) \in \N^{m+1}$ such that $x_0=x$, $x_m = x'$ and for all $i \leq m-1$, $\abs{x_{i+1} - x_i} \leq 1$. Similarly, we let $\ltrees{m}^{>1}(x)$ be the set of the labeled trees $(T,l) \in \ltrees{m}^+(x)$ such that $l>1$, and $\mathcal{M}^{>1}_{m, x \rightarrow x'}$ be the set of the walks $(x_1, \ldots, x_m) \in \mathcal{M}^+_{m, x \rightarrow x'}$ such that $x_0, \ldots, x_m >1$. Also recall that $\mu_{(x_0,\ldots,x_m)}$ denotes the distribution of a ``uniform infinite'' forest with root labels $x_0,\ldots,x_m$.

For all $a,b,c,k' \geq 1$, $\ul{x} \in \mathcal{M}_{a, 1\rightarrow k'}^+$, $\ul{y} \in \mathcal{M}_{b, 2\rightarrow k'}^{>1}$, $\ul{z} \in \mathcal{M}_{c+1, k'\rightarrow k}^{>1}$, let $A_{\infty}^{(k)}(a,b,c,\ul{x},\ul{y},\ul{z})$ denote the event:
\begin{gather*}
a_{\infty}^{(k)}=a, b_{\infty}^{(k)}=b, c_{\infty}^{(k)}=c, \quad (X_{\infty,1}^{(k)},\ldots,X_{\infty,a}^{(k)})=\ul{x}, \\ (Y_{\infty,1}^{(k)},\ldots,Y_{\infty,b}^{(k)})=\ul{y}, \quad \quad (Z_{\infty,0}^{(k)},\ldots,Z_{\infty,c}^{(k)})=\ul{z}.
\end{gather*}

\begin{cor} \label{T: Distribution of theta(infty,k) with two marked points}
For all $a,b,c,k' \geq 1$, $\ul{x} \in \mathcal{M}_{a, 1\rightarrow k'}^+$, $\ul{y} \in \mathcal{M}_{b, 2\rightarrow k'}^{>1}$, $\ul{z} \in \mathcal{M}_{c+1, k'\rightarrow k}^{>1}$, we have
\begin{gather*}
\P[sz]{A_{\infty}^{(k)} (a,b,c,\ul{x},\ul{y},\ul{z})}
= \frac{a+b+1}{3^{a+b+c}} \(\prod_{i=1}^a w(x_i)\) \(\prod_{i=1}^b w(y_i-1)\) \(\prod_{i=1}^{c+1} w(z_i)w(z_i-1)\).
\end{gather*}
Moreover, conditionally on $A_{\infty}^{(k)} (a,b,c,\ul{x},\ul{y},\ul{z})$, with the conventions $x_0=y_0=0$:
\begin{itemize}
\item The forests $(\tau_{\infty,i})_{1 \leq i \leq a+c}$, $(\tau'_{\infty,i})_{1 \leq i \leq b+c}$ and $(\ol{\tau}_{\infty,0},\ldots,\ol{\tau}_{\infty,a},\ol{\tau}'_{\infty,0},\ldots,\ol{\tau}'_{\infty,b-1})$ are independent.
\item The trees $\tau_{\infty,i}$, $1 \leq i \leq a+c$ are independent random variables, respectively distributed according to $\GWgeom{x_i}^+$, $1 \leq i \leq a$ and $\GWgeom{z_{a+i}}^+$, $a+1 \leq i \leq a+c$.
\item The trees $\tau'_{\infty,i}$, $1 \leq i \leq b+c$ are independent random variables, obtained by adding $1$ to the labels of trees distributed according to $\GWgeom{y_i-1}^+$, $1 \leq i \leq b$ and $\GWgeom{z_{b+i}}^+$, $b+1 \leq i \leq a+c$, respectively.
\item The forest $(\ol{\tau}_{\infty,0},\ldots,\ol{\tau}_{\infty,a},\ol{\tau}'_{\infty,0},\ldots,\ol{\tau}'_{\infty,b-1})$ follows the distribution $\mu_{(0,x_1,\ldots,x_a,0,y_1,\ldots,y_{b-1})}$.
\end{itemize}
\end{cor}

\begin{proof}
We have
\begin{align*}
\P[sz]{A_{\infty}^{(k)} (a,b,c,\ul{x},\ul{y},\ul{z})}
= \frac{a+b+1}{6^{a+b}12^c}
& \( \prod_{i=1}^a \sum_{n_i \geq 0} \frac{1}{12^{n_i}} \card{\ltrees{n_i}^+(x_i)} \)
  \( \prod_{i=1}^b \sum_{n_i \geq 0} \frac{1}{12^{n_i}} \card{\ltrees{n_i}^{>1}(y_i)} \) \\
& \( \prod_{i=1}^{c+1} \sum_{n_i \geq 0} \frac{1}{12^{n_i}} \card{\ltrees{n_i}^+(z_i)} \)
  \( \prod_{i=1}^{c+1} \sum_{n_i \geq 0} \frac{1}{12^{n_i}} \card{\ltrees{n_i}^{>1}(z_i)} \).
\end{align*}
Using equation \eqref{E: w [Cha-Dur]} and the fact that $\card{\ltrees{m}^{>1}(x)} = \card{\ltrees{m}^+(x-1)}$ for all $m,x \in \N$, this gives
\begin{align*}
\P[sz]{A_{\infty}^{(k)} (a,b,c,\ul{x},\ul{y},\ul{z})}
&= \frac{a+b+1}{6^{a+b}12^c} \(\prod_{i=1}^a 2 w(x_i)\) \(\prod_{i=1}^b 2 w(y_i-1)\) \(\prod_{i=1}^{c+1} 4 w(z_i)w(z_i-1)\) \\
&= \frac{a+b+1}{3^{a+b+c}} \(\prod_{i=1}^a w(x_i)\) \(\prod_{i=1}^b w(y_i-1)\) \(\prod_{i=1}^{c+1} w(z_i)w(z_i-1)\),
\end{align*}
hence the first part of the Lemma. The conditional distributions of the trees $\tau_{\infty,i}$, $\tau'_{\infty,i}$, $\ol{\tau}_{\infty,i}$ and $\ol{\tau}'_{\infty,i}$ are then obtained exactly as in the proof of Proposition \ref{T: Distribution of theta(infty,k)}.
\end{proof}

\subsection{Proof of the joint convergence}

As in Section \ref{S: First convergence}, the main step of the proof of the convergence is to show the convergence of the labels on the branches $\pt{x}_{\infty,i}^{(k)}$, $i \geq 1$ and $\pt{y}_{\infty,i}^{(k)}$, $i \geq 1$. Fix $r \in \N$. For all $k \in \N$, and for all continuous bounded functions $F$, $G$ from $\R^r$ into $\R$, we let
\begin{gather*}
\mathcal{E}_k(F,G) := \E[sz]{F(X_{\infty,1}^{(k)}, \ldots, X_{\infty,r}^{(k)}) G(Y_{\infty,1}^{(k)}, \ldots, Y_{\infty,r}^{(k)}) \ind{a_{\infty}^{(k)}, b_{\infty}^{(k)} \geq r}}.
\end{gather*}

\begin{lem} \label{T: Joint cv labels}
We have the convergence
\begin{gather*}
\mathcal{E}_k(F,G)
\cvg{k \rightarrow \infty}{}
\Efrom{1}{F(\tilde{X}_0, \ldots, \tilde{X}_{r-1})} \Efrom{1}{G(\tilde{X}_0+1, \ldots, \tilde{X}_{r-1}+1)}.
\end{gather*}
\end{lem}

\begin{proof}
As in the previous section, we introduce independent random walks $\hat{X}$, $\hat{Y}$ and $\hat{Z}$ with uniform steps in $\{-1,0,1\}$, and consider associated martingales $\MX$, $\MY$ and $\MZ$ such that for all $j \geq 0$,
\begin{gather*}
\MX_j = \frac{f(\hat{X}_j)}{f(\hat{X}_0)} \prod_{i=0}^{j-1} w(\hat{X}_i) \qquad
\MY_j = \frac{f(\hat{Y}_j)}{f(\hat{Y}_0)} \prod_{i=0}^{j-1} w(\hat{Y}_i) \qquad
\MZ_j = \frac{g(\hat{Z}_j)}{g(\hat{Z}_0)} \prod_{i=0}^{j-1} v(\hat{Z}_i),
\end{gather*}
where $v(x) = w(x) w(x+1) = x(x+4)/(x+2)^2$ and $g(x) = x(x+4)(5x^2 + 20 x + 17)$ for all $x \in \N$. From now on, we work under the assumption $1 \leq r \leq k$. With the above notation, we can write
\begin{align*}
\mathcal{E}_k(F,G)
= \sum_{a,b \geq r-1} \sum_{c \geq 0} \frac{a+b+1}{9} \sum_{k' \geq 1} \mathbb{E}
& \left[ \Efrom[sz]{1}{F(\hat{X}_0, \ldots, \hat{X}_{r-1}) \frac{f(1)w(k'+1)}{f(k'+1)} \MX_{a-1} \ind{\hat{X}_{a-1}=k'+1}} \right. \\
& \left. \Efrom[sz]{1}{G(\hat{Y}_0+1, \ldots, \hat{Y}_{r-1}+1) \frac{f(1)w(k')}{f(k')} \MY_{b-1} \ind{\hat{Y}_{b-1}=k'}} \right. \\
& \left. \Efrom[sz]{k'}{\frac{g(k')v(k-1)}{g(k-1)} \MZ_c \ind{\hat{Z}_c=k-1}} \right].
\end{align*}
Using the Markov property and re-arranging the terms yields
\begin{align*}
\mathcal{E}_k(F,G)
= \sum_{k' \geq 1} \mathbb{E}
& \left[\MX_{r-1} F(\hat{X}_0, \ldots, \hat{X}_{r-1}) \MY_{r-1} G(\hat{Y}_0+1, \ldots, \hat{Y}_{r-1}+1) \frac{f(1)w(k'+1)}{3 f(k'+1)} \vphantom{\sum_{c\geq 0}}\right. \\
& \left. \frac{f(1)w(k')}{3 f(k')} \sum_{a,b \geq 0} (a+b+2r-1)  \Efrom[sz]{\hat{X}_{r-1}}{\MX'_a \ind{\hat{X}'_a=k'+1}} \Efrom[sz]{\hat{Y}_{r-1}}{\MY'_b \ind{\hat{Y}'_b=k'}} \right. \\
& \left. \frac{g(k')v(k-1)}{g(k-1)} \sum_{c \geq 0} \Efrom[sz]{k'}{ \MZ_c \ind{\hat{Z}_c=k-1}} \right],
\end{align*}
where $\hat{X}',\hat{Y}', \MX', \MY'$ are independent copies of $\hat{X},\hat{Y}, \MX, \MY$.
We already have the necessary ingredients in Section \ref{S: First convergence} to study the first factors; the only additional quantity we need to compute is
\begin{gather*}
\SumGV{k'}{k} = \sum_{c \geq 0} \Efrom[sz]{k'}{\MZ_c \ind{\hat{Z}_c=k}} = \sum_{c \geq 0} \Pfrom[sz]{k'}{\dbtilde{Z}_c=k},
\end{gather*}
where $\dbtilde{Z}$ is the image of $\hat{Z}$ under the measure-change given by the martingale $\MZ$, i.e. the Markov process such that $\E{\phi(\dbtilde{Z}_i)} = \E{\MZ_i \phi(\hat{Z}_i)}$ for every continuous bounded function $\phi$.

\begin{lem} \label{T: Values of SumGV}
Fix $k,k' \geq 2$. We have the following equalities:
\begin{itemize}
\item if $k'\leq k$,
\begin{gather*}
\frac{g(k')v(k)}{g(k)} \SumGV{k'}{k} = \frac{3 g(k')}{35 (k+1)(k+2)(k+3)}
\end{gather*}
\item if $k'>k$,
\begin{gather*}
\frac{g(k')v(k)}{g(k)} \SumGV{k'}{k} = \frac{3 g(k)}{35 (k'+1)(k'+2)(k'+3)}.
\end{gather*}
\end{itemize}
\end{lem}

We omit the technical detail of the proof of this result; the ideas are exactly the same as in the proof of Lemma \ref{T: Values of SumFW}. Now
\begin{align*}
\mathcal{E}_k(F,G)
= \sum_{1 \leq x,y \leq r} \(\sum_{k' \geq 1} \mathcal{H}_{x,y,k'}(k)\)
& \Efrom{1}{\MX_{r-1} F(\hat{X}_0, \ldots, \hat{X}_{r-1}) \ind{\hat{X}_{r-1}=x}} \\
& \Efrom{1}{\MY_{r-1} G(\hat{Y}_0+1, \ldots, \hat{Y}_{r-1}+1) \ind{\hat{Y}_{r-1}=y}},
\end{align*}
where
\begin{align*}
\mathcal{H}_{x,y,k'}(k)
= & \frac{f(1)w(k'+1)}{3 f(k'+1)} \frac{f(1)w(k')}{3 f(k')} \frac{g(k')v(k-1)}{g(k-1)} \SumGV{k'}{k-1} \\
& \qquad \times (\SumFWStar{x}{k'+1} \SumFW{y}{k'} + \SumFW{x}{k'+1} \SumFWStar{y}{k'} + (2r-5) \SumFW{x}{k'+1} \SumFW{y}{k'}).
\end{align*}
Therefore, it is enough to show that $\sum_{k' \geq 1} \mathcal{H}_{x,y,k'}(k)$ converges to $1$ as $k \rightarrow \infty$, uniformly in $x,y \leq r$.

Let us first treat the terms for which $k' \geq k$. We have
\begin{gather*}
\frac{g(k')v(k-1)}{g(k-1)} \SumGV{k'}{k-1}
= \frac{3 g(k-1)}{35 (k'+1)(k'+2)(k'+3)}.
\end{gather*}
Moreover, the results of Lemmas \ref{T: Values of SumFW} and \ref{T: Values of SumFWStar} show that, uniformly in $y \leq r$,
\begin{gather*}
\frac{f(1)w(k')}{3 f(k')} \SumFWStar{y}{k'} \cvg{k' \rightarrow \infty}{} 1 \\
\frac{f(1)w(k')}{3 f(k')} \SumFW{y}{k'} \sim_{k' \rightarrow \infty} \frac{2}{(k')^2},
\end{gather*}
and that the same holds with $k'+1$ instead of $k'$ in the left-hand term. As a consequence, we have
\begin{align}
\sum_{k' \geq k} \mathcal{H}_{x,y,k'}(k)
& \sim_{k \rightarrow \infty} \frac{3 g(k-1)}{35} \sum_{k' \geq k} \frac{4}{(k')^5} \nonumber \\
& \sim_{k \rightarrow \infty} \frac{3 k^4}{7} \frac{1}{k^4} = \frac{3}{7}, \label{E: JCV, first sum}
\end{align}
uniformly in $x,y \leq r$.

In second, we consider the terms for which we have $x \vee y < k' \leq k-1$. Lemmas \ref{T: Values of SumFW} and \ref{T: Values of SumFWStar} yield the following estimates, uniformly in $y \leq r$:
\begin{gather*}
\frac{f(1)w(k')}{3 f(k')} \SumFWStar{y}{k'} = 1 - \frac{3}{10(k'+1)(k'+2)} \(C_y + 1 - \(1 \wedge \frac{h(k')}{h(y)}\)\)
\end{gather*}
and
\begin{gather*}
\frac{f(1)w(k')}{3 f(k')} \SumFW{y}{k'} \sim_{k' \rightarrow \infty} \frac{2}{(k')^2}.
\end{gather*}
Putting this together with the result of Lemma \ref{T: Values of SumGV}, we get
\begin{align}
\sum_{k'=x \vee y +1}^{k-2} \mathcal{H}_{x,y,k'}(k)
& \sim_{k \rightarrow \infty} \frac{3}{35 k^3} \sum_{k'=x \vee y +1}^{k-1} 2 \times \frac{2}{(k')^2} \times 5(k')^4 \nonumber \\
& \sim_{k \rightarrow \infty} \frac{12}{7 k^3} \frac{k^3}{3} = \frac{4}{7}. \label{E: JCV, second sum}
\end{align}

The remaining term is
\begin{gather*}
\sum_{k'=1}^{x \vee y} \mathcal{H}_{x,y,k'}(k) = O\( \frac{1}{k^3}\).
\end{gather*}
Putting this together with \eqref{E: JCV, first sum} and \eqref{E: JCV, second sum}, we obtain
\begin{gather*}
\sum_{k' \geq 1} \mathcal{H}_{x,y,k'}(k) \cvg{k \rightarrow \infty}{} \frac{3}{7}+\frac{4}{7} = 1,
\end{gather*}
uniformly in $x,y \leq r$, hence the conclusion.
\end{proof}

To complete the proof of Theorem \ref{T: Joint cv of the rerooted trees}, we finally come back to the trees attached on the branches $\pt{x}_{\infty,i}^{(k)}$, $i \geq 1$ and $\pt{y}_{\infty,i}^{(k)}$, $i \geq 1$, putting together the above result and Corollary \ref{T: Distribution of theta(infty,k) with two marked points}. Let $E^{(k)}(r)$ be the event that $a_{\infty}^{(k)}, b_{\infty}^{(k)} \geq r$, and the trees $(\ol{\tau}_{\infty,i})^{(k)}$ and $(\ol{\tau}'_{\infty,i})^{(k)}$ are finite for all $i \leq r$. Conditionally on $E^{(k)}(r)$, we have the following properties on the spines of $\theta_{\infty}^{(k)}$ and $\theta_{\infty}^{(-k+1)}$:
\begin{itemize}
\item The points $\spine_i(\theta_{\infty}^{(k)})$ and $\pt{x}_{\infty,i}^{(k)}$ are the same for all $i \leq r$, hence $R_i(\theta_{\infty}^{(k)}) = \tau_{\infty,i}^{(k)}$ and $L_i(\theta_{\infty}^{(k)}) = \ol{\tau}_{\infty,i}^{(k)}$ for all $i < r$.
\item The points $\spine_i(\theta_{\infty}^{(-k+1)})$ and $\pt{y}_{\infty,i}^{(k)}$ are the same for all $i \leq r$, hence $R_i(\theta_{\infty}^{(-k+1)}) = (\ol{\tau}'_{\infty,i})^{(k)}$ and $L_i(\theta_{\infty}^{(-k+1)}) = (\tau'_{\infty,i})^{(k)}$ for all $i < r$.
\item As a consequence, the spine labels $(S_i(\theta_{\infty}^{(k)}),S_i(\theta_{\infty}^{(-k+1)})-1)_{1 \leq i \leq r}$ converge in distribution to $(\tilde{X}_i,\tilde{Y}_i)_{0 \leq i \leq r-1}$, with $\tilde{X}_0=\tilde{Y}_0=1$.
\end{itemize}
Further conditioning on $(S_i(\theta_{\infty}^{(k)}),S_i(\theta_{\infty}^{(-k+1)}))_{0 \leq i < r}$, we get that:
\begin{itemize}
\item The subtrees $L_i(\theta_{\infty}^{(k)})$, $0 \leq i < r$ and $R_i(\theta_{\infty}^{(k)})$, $1 \leq i < r$ are independent random variables, with respective distributions $\GWgeom{S_i(\theta_{\infty}^{(k)})}$ and $\GWgeom{S_i(\theta_{\infty}^{(k)})}^+$.
\item  The subtrees $L_i(\theta_{\infty}^{(-k+1)})$, $0 \leq i < r$ and $R_i(\theta_{\infty}^{(-k+1)})$, $1 \leq i < r$ are independent random variables, respectively obtained by adding 1 to the labels of trees distributed according to $\GWgeom{S_i(\theta_{\infty}^{(k)})-1}^+$ and $\GWgeom{S_i(\theta_{\infty}^{(k)})-1}$.
\item The random forests $(L_i(\theta_{\infty}^{(k)},R_i(\theta_{\infty}^{(k)}))_{0\leq i < r}$ and $(L_i(\theta_{\infty}^{(-k+1)},R_i(\theta_{\infty}^{(-k+1)}))_{0\leq i < r}$ are independent.
\end{itemize}
Therefore, it is enough to show that $\P{\ol{E^{(k)}(r)}}$ converges to 0 as $k \rightarrow \infty$. Fix $\epsilon > 0$. We have
\begin{gather*}
\P{\ol{E^{(k)}(r)}} \leq \P{a_{\infty}^{(k)} < r \mbox{ or } b_{\infty}^{(k)} < r} + \E[sz]{1 \wedge \frac{2r+2}{a_{\infty}^{(k)}+b_{\infty}^{(k)}+1}}
\end{gather*}
We know from Lemma \ref{T: Joint cv labels} that the first term converges to 0. More precisely, for all $r' \in \N$, we have
\begin{gather*}
\P{a_{\infty}^{(k)} < r' \mbox{ or } b_{\infty}^{(k)} < r'} \leq \epsilon
\end{gather*}
for all $k$ large enough, hence
\begin{gather*}
\E[sz]{1 \wedge \frac{2r+2}{a_{\infty}^{(k)}+b_{\infty}^{(k)}+1}} \leq \epsilon + \frac{2r+2}{2r'+1}
\end{gather*}
for $k$ large enough. Thus we can choose $r'$ in such a way that for all $k$ large enough, we have
\begin{gather*}
\P{\ol{E^{(k)}(r)}} \leq 3 \epsilon.
\end{gather*}
This concludes the proof.

\section{Convergence of the associated quadrangulations} \label{S: Joint CV of the quadrangulations}

As indicated in the Introduction, the main step of the proof of Theorem \ref{T: Joint CV of the quadrangulations} consists in showing the following result. We use the conventions
\begin{gather*}
\theta_{\infty}^{(\infty)} = \ora{\theta_{\infty}}, \qquad
\theta_{\infty}^{(-\infty)} = \ola{\theta_{\infty}}, \\
\ora{Q}_{\infty}^{(\infty)} = \ora{Q}_{\infty}, \qquad
\ola{Q}_{\infty}^{(\infty)} = \ola{Q}_{\infty}.
\end{gather*}

\begin{prop} \label{T: Ball inclusions}
For all $r \in \N$ and $\epsilon > 0$, there exists $h \in \N$ such that for all $k$ large enough, possibly infinite, we have
\begin{gather} \label{E: R+ ball inclusion}
V\(B_{\ora{Q}_{\infty}^{(k)}}(r)\) \subset V\(B_{\theta_{\infty}^{(k)}}(h)\)
\end{gather}
and
\begin{gather} \label{E: L+ ball inclusion}
V\(B_{\ola{Q}_{\infty}^{(k)}}(r)\) \subset V\(B_{\theta_{\infty}^{(-k+1)}}(h)\) \cup \{\lambda_i: \abs{i}\leq r \}
\end{gather}
with probability at least $1-\epsilon$.
\end{prop}

Let us first see how this result allows us to prove the theorem.

\begin{proof}[Proof of Theorem \ref{T: Joint CV of the quadrangulations}]
Using the Skorokhod representation theorem, we assume that the convergence
\begin{gather*}
(\theta_{\infty}^{(k)},\theta_{\infty}^{(-k+1)}) \cvg{k \rightarrow \infty}{} (\ora{\theta_{\infty}},\ola{\theta_{\infty}}),
\end{gather*}
obtained in Theorem \ref{T: Joint cv of the rerooted trees}, holds almost surely. In particular, it also holds in probability: for all $h \in \N$ and $\epsilon > 0$, we have
\begin{gather*}
\P[sz]{D(\theta_{\infty}^{(k)},\ora{\theta_{\infty}}) \leq \frac{1}{1+h} \mbox{ and } D(\theta_{\infty}^{(-k+1)},\ola{\theta_{\infty}}) \leq \frac{1}{1+h}}
\geq 1- \epsilon
\end{gather*}
for all $k$ large enough, which means that
\begin{gather} \label{E: Applying the joint tree CV}
B_{\theta_{\infty}^{(k)}}(h) = B_{\ora{\theta_{\infty}}}(h)
\qquad \mbox{and} \qquad
B_{\theta_{\infty}^{(-k)}}(h) = B_{\ola{\theta_{\infty}}}(h)
\end{gather}
with probability at least $1-\epsilon$, for all $k$ large enough.

For all $r \in \N$ and $\epsilon > 0$, the above proposition shows that there exists $h_{\epsilon}$ such that the inclusions \eqref{E: R+ ball inclusion} and \eqref{E: L+ ball inclusion} hold with probability at least $1-\epsilon$, for all $k$ large enough. Putting this together with \eqref{E: Applying the joint tree CV} for $h=h_{\epsilon}$, we get that
\begin{gather*}
B_{\ora{Q}_{\infty}^{(k)}}(r) = B_{\Phi(\ora{\theta_{\infty}})} (r)
\qquad \mbox{and} \qquad
B_{\ola{Q}_{\infty}^{(k)}}(r) = B_{\Phi(\ola{\theta_{\infty}})} (r)
\end{gather*}
with probability at least $1-\epsilon$, for all $k$ large enough (possibly infinite). Therefore, we have the convergence
\begin{gather*}
(\ora{Q}_{\infty}^{(k)}, \ola{Q}_{\infty}^{(k)}) \cvg{k \rightarrow \infty}{} (\ora{Q}_{\infty}, \ola{Q}_{\infty})
\end{gather*}
in probability, hence the joint distributional convergence.
\end{proof}

The rest of the section is devoted to the proof of Proposition \ref{T: Ball inclusions}. We first introduce conditions on the ``left-hand side'' and ``right-hand side'' of the trees $\theta_{\infty}^{(k)}$, $\theta_{\infty}^{(-k+1)}$, which are sufficient to get the ball inclusions \eqref{E: R+ ball inclusion} and \eqref{E: L+ ball inclusion}. This is done in Section \ref{S: LR conditions} (see in particular Lemma \ref{T: LR conditions for k<infty}). In Sections \ref{S: Spine labels}, \ref{S: Left-hand condition} and \ref{S: Right-hand condition}, we then show that an ``elementary block'' of these conditions holds with arbitrarily high probability, for all $s$ and $k$ large enough. The corresponding results are stated in Lemmas \ref{T: Left-hand condition} and \ref{T: Right-hand condition}. Finally, Section \ref{S: Conclusion} concludes the proof of the proposition.

\subsection{Conditions on the right-hand and left-hand part of a labeled tree} \label{S: LR conditions}

We first introduce some more detailed notation for the balls in a rooted tree $T$. For all $s \geq 0$, we let $\partial B_T (s)$ denote the ``boundary'' of the ball of radius $s$, defined as
\begin{gather*}
\partial B_T(s) = \{ \pt{v} \in T: \pt{v} \mbox{ has height } s \}.
\end{gather*}
In what follows, the letter $L$ will correspond to the ``left-hand part'' of a tree, and $R$ will be used for the ``right-hand part''. All the following notations are given for the left-hand part, and are also valid for the right-hand part (replacing $L$ by $R$). Assume that $T \in \textbf{S}$, and recall that $L_i(T)$ denotes the subtree of the descendants of $\spine_i(T)$ that are on the left of the spine. We let
\begin{gather*}
L(T) = \bigcup_{i \geq 0} L_i(T),
\end{gather*}
and for all $s \geq 0$,
\begin{gather*}
\LB_T (s) = B_T(s) \cap L(T) = \bigcup_{i=0}^s B_{L_i(T)}(s-i),
\end{gather*}
and
\begin{gather*}
\partial \LB_T(s) = \partial B_T(s) \cap L(T).
\end{gather*}
We also use the natural extensions of this notation to labeled trees.

We are interested in the following subsets of $\lstrees$, for all $r,s,s',h \in \N$:
\begin{gather*}
\mathbb{A}_L (r,s,s',h) = \{ (T,l) \in \lstrees: \bigcup_{i=0}^{s'} L_i(T) \subset B_T(h), \mbox{ and } \exists \pt{v} \in \bigcup_{i=s+1}^{s'} L_i(T) \mbox{ s.t. } l(\pt{v}) = -r \}
\end{gather*}
and
\begin{gather*}
\mathbb{A}_{L +} (r,s) = \{ (T,l) \in \lstrees: \forall \pt{v} \in L(T) \setminus \LB_T(s),\ l(\pt{v}) > r \}.
\end{gather*}
Figure \ref{F: LRConditions} illustrates these definitions. We give a sufficient condition for an inclusion between the balls in $\theta$ and in $\Phi(\theta)$, in terms of these sets $\mathbb{A}_L (r,s,s',h)$, $\mathbb{A}_{L+} (r,s)$, $\mathbb{A}_R (r,s,s',h)$ and $\mathbb{A}_{R+} (r,s)$:

\begin{figure}[t]
\begin{center}
\includegraphics{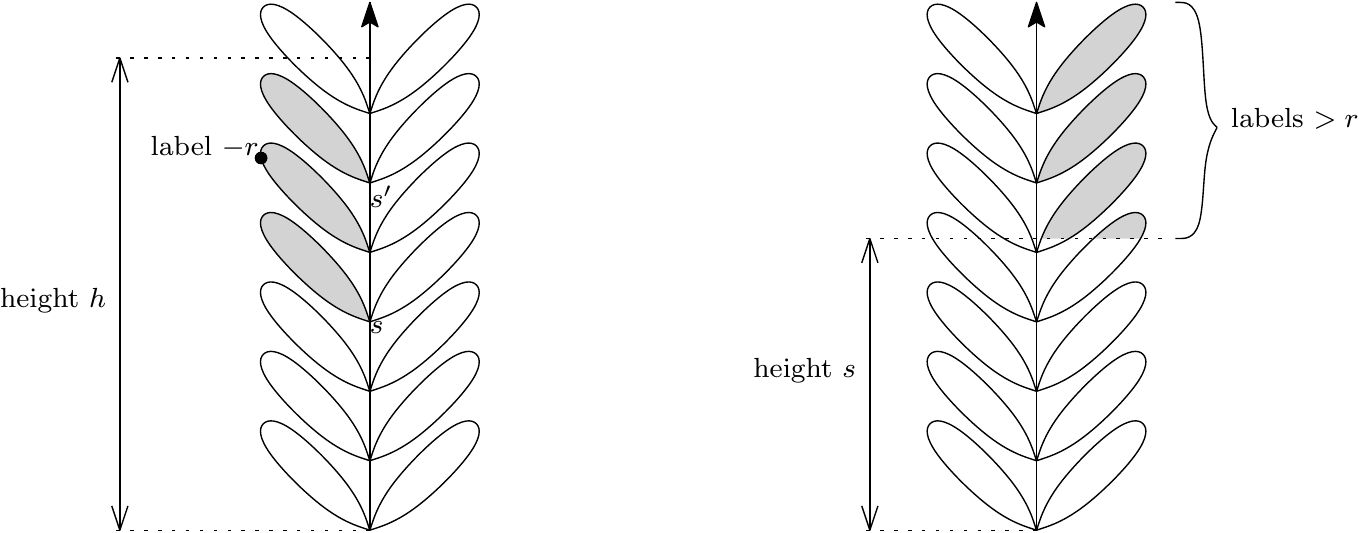}
\caption{An illustration of the conditions $\theta \in \mathbb{A}_L (r,s,s',h)$ (on the left) and $\theta \in \mathbb{A}_{R+} (r,s)$ (on the right).}
\label{F: LRConditions}
\end{center}
\end{figure}

\begin{lem} \label{T: LR conditions for ball inclusions}
Let $r \in \N$.
\begin{enumerate}
\item For all $\theta \in \ora{\lstrees}$, if there exists sequences $(s(r'))_{0 \leq r' \leq r}$ and $(h(r'))_{1 \leq r' \leq r}$ such that
\begin{gather*}
\theta \in \mathbb{A}_L (r',s(r'-1),s(r'),h(r')) \cap \mathbb{A}_{R+} (r',s(r')) \qquad \forall r' \in \{1,\ldots,r\},
\end{gather*}
then we have $V(B_{\Phi(\theta)}(r)) \subset V(B_{\theta}(h(r)))$.
\item  For all $\theta \in \ola{\lstrees}$, if there exists sequences $(s(r'))_{0 \leq r' \leq r}$ and $(h(r'))_{1 \leq r' \leq r}$ such that
\begin{gather*}
\theta \in \mathbb{A}_R (r',s(r'-1),s(r'),h(r')) \cap \mathbb{A}_{L+} (r',s(r')) \qquad \forall r' \in \{1,\ldots,r\},
\end{gather*}
then we have $V(B_{\Phi(\theta)}(r)) \subset V(B_{\theta}(h(r))) \cup \{\lambda_i: \abs{i}\leq r \}$.
\end{enumerate}
\end{lem}

\begin{proof}
Let $\theta \in \ora{\lstrees}$. We show by induction that for all $r \geq 0$, if there exists sequences $(s(r'))_{0 \leq r' \leq r}$ and $(h(r'))_{1 \leq r' \leq r}$ such that
\begin{gather*}
\theta \in \mathbb{A}_L (r,s(r'-1),s(r'),h(r')) \cap \mathbb{A}_{R+} (r',s(r')) \qquad \forall r' \in \{1,\ldots,r\},
\end{gather*}
then we have
\begin{gather*}
V\(B_{\Phi(\theta)}(r)\) \subset V\(\RB_{\theta}(s(r)) \cup \bigcup_{i=0}^{s(r)} L_i(\theta)\).
\end{gather*}
This is enough to prove the first part of the Lemma. Indeed, since $\theta$ belongs to $\mathbb{A}_L (r,s(r-1),s(r),h(r))$, we have $\bigcup_{i=0}^{s(r)} L_i(\theta) \subset B_{\theta} (h(r))$ and $s(r) \leq h(r)$, so
\begin{gather*}
V\(\RB_{\theta}(s(r)) \cup \bigcup_{i=0}^{s(r)} L_i(\theta)\) \subset V(B_{\theta}(h(r))).
\end{gather*}

The result is obviously true for $r=0$. Assume that it holds for a given $r \geq 0$. We order the corners of $\theta$ by writing $\pt{c}_n(\theta) \leq \pt{c}_{n'} (\theta)$ for all $n \leq n'$. For all $r' \leq r+1$, let $\xi_{r'}$ denote the largest corner incident to the vertex $\spine_{s(r')}$. Note that for all $r' \leq r$, for every corner $\pt{c}$ of $\theta$, we have $\pt{c} \leq \xi_{r'}$ if and only if every corner $\tilde{\pt{c}}$ incident to the same vertex as $\pt{c}$ verifies $\tilde{\pt{c}} \leq \xi_{r'}$. The induction hypothesis ensures that for every corner $\pt{c}$ of $\theta$ which is incident to a vertex of $B_{\Phi(\theta)}(r)$, we have $\pt{c} \leq \xi_r$. (This is the case even if the corresponding vertex is in the right-hand part of $\theta$.)

Let $\pt{v} \in V(\theta)$. The vertex $\pt{v}$ belongs to $B_{\Phi(\theta)}(r+1)$ if and only if one of the following conditions holds:
\begin{enumerate}
\item $\pt{v}$ belongs to $B_{\Phi(\theta)}(r)$.
\item There exist a vertex $\pt{v}'$ of $B_{\Phi(\theta)}(r)$, and two corners $\pt{c}$ and $\pt{c}'$, respectively incident to $\pt{v}$ and $\pt{v}'$, such that $\sigma_{\theta}(\pt{c}) = \pt{c}'$.
\item There exist a vertex $\pt{v}'$ of $B_{\Phi(\theta)}(r)$, and two corners $\pt{c}$ and $\pt{c}'$, respectively incident to $\pt{v}$ and $\pt{v}'$, such that $\sigma_{\theta}(\pt{c}') = \pt{c}$.
\end{enumerate}
Respectively, in these three cases, it holds that:
\begin{enumerate}
\item Every corner $\tilde{\pt{c}}$ incident to $\pt{v}$ is such that $\tilde{\pt{c}} \leq \xi_r \leq \xi_{r+1}$.
\item We have $\pt{c} \leq \pt{c}' \leq \xi_r$, so every corner $\tilde{\pt{c}}$ incident to $\pt{v}$ is such that $\tilde{\pt{c}} \leq \xi_r \leq \xi_{r+1}$.
\item The corner $\pt{c}$ is the first corner with label $l(\pt{v}')-1$ after $\pt{c}'$. Since $\pt{v}'$ belongs to $B_{\Phi(\theta)}(r)$, the bound \eqref{E: Lower bound on d_Phi(.)} ensures that
\begin{gather*}
d_{\Phi(\theta)}(\pt{v}_0,\pt{v}') \geq \abs{l(\pt{v}_0) - l(\pt{v}')} = l(\pt{v}')
\end{gather*}
(where $\pt{v}_0$ denotes the root of $\theta$), so $l(\pt{v}')-1 \geq -r-1$. Moreover, we have $\pt{c}' \leq \xi_r$, and since $\theta$ belongs to $\mathbb{A}_L(r+1,s(r),s(r+1),h(r+1))$, there exists a corner with label $-r-1$ between $\xi_r$ and $\xi_{r+1}$. As a consequence, we have $\pt{c} \leq \xi_{r+1}$, and therefore every corner $\tilde{\pt{c}}$ incident to $\pt{v}$ is such that $\tilde{\pt{c}} \leq \xi_{r+1}$.
\end{enumerate}
Thus, we get the inclusion
\begin{gather*}
V\(B_{\Phi(\theta)}(r+1)\) \subset V\(R(\theta) \cup \bigcup_{i=0}^{s(r+1)} L_i(\theta)\).
\end{gather*}
Finally, for every vertex $\pt{v} \in R(\theta) \setminus \RB_{\theta}(s(r+1))$, since $\theta$ belongs to $\mathbb{A}_{R+}(r+1,s(r+1))$, we have $l(\pt{v}) > r+1$, so $\pt{v}$ is at distance at least $r+2$ of the root in $\Phi(\theta)$. This yields
\begin{gather*}
V\(B_{\Phi(\theta)}(r+1)\) \subset V\(\RB_{\theta}(s(r+1)) \cup \bigcup_{i=0}^{s(r+1)} L_i(\theta)\).
\end{gather*}

We now consider the case where $\theta \in \ola{\lstrees}$. Similarly, it is enough to show by induction that for all $r \geq 0$, if there exists sequences $(s(r'))_{0 \leq r' \leq r}$ and $(h(r'))_{1 \leq r' \leq r}$ verifying the hypotheses, then we have
\begin{gather*}
V\(B_{\Phi(\theta)}(r) \setminus \Lambda\) \subset V\(\LB_{\theta}(s(r)) \cup \bigcup_{i=0}^{s(r)} R_i(\theta)\).
\end{gather*}
(Indeed, equation \eqref{E: Lower bound on d_Phi(.)} shows that $V(B_{\Phi(\theta)}(r) \cap \Lambda) \subset \{\lambda_i: \abs{i}\leq r \}$.) Assume that the result holds for a given $r \geq 0$. For all $r' \leq r+1$, let $\xi'_{r'}$ denote the smallest corner incident to the vertex $\spine_{s(r')}$. For every corner $\pt{c}$ of $\theta$ which is incident to a vertex of $B_{\Phi(\theta)}(r)$, we have $\pt{c} \geq \xi_r$. We fix $\pt{v} \in V(\theta)$, and study the same three cases as above. Respectively, we obtain that:
\begin{enumerate}
\item Every corner $\tilde{\pt{c}}$ incident to $\pt{v}$ is such that $\tilde{\pt{c}} \geq \xi'_r \geq \xi'_{r+1}$.
\item The corner $\pt{c}'$ is the first corner with label $l(\pt{v})-1$ after $\pt{c}$ (or a point of $\Lambda$, if such a corner does not exist), and equation \eqref{E: Lower bound on d_Phi(.)} gives that $l(\pt{v})-1=l(\pt{v}') \geq -r$. Since $\theta$ belongs to $\mathbb{A}_R(r+1,s(r),s(r+1),h(r+1))$, there exists a corner with label $-r-1$ which is (strictly) between $\xi'_{r+1}$ and $\xi'_r$. So, if we had $\pt{c} < \xi'_{r+1}$, this would imply $\pt{c}' < \xi'_r$, which is impossible since $\pt{v}'$ is in $B_{\Phi(\theta)}(r)$. Thus, we have $\pt{c} \geq \xi'_{r+1}$, and every corner $\tilde{\pt{c}}$ incident to $\pt{v}$ is such that $\tilde{\pt{c}} \geq \xi'_{r+1}$.
\item Note that since $\pt{v}$ is a vertex of $\theta$, we cannot have $\pt{v}' \in \Lambda$. Thus, we have $\pt{c} \geq \pt{c}' \geq \xi'_r$, so every corner $\tilde{\pt{c}}$ incident to $\pt{v}$ is such that $\tilde{\pt{c}} \geq \xi'_r \geq \xi'_{r+1}$.
\end{enumerate}
This yields the inclusion
\begin{gather*}
V\(B_{\Phi(\theta)}(r+1)\setminus \Lambda\) \subset V\(L(\theta) \cup \bigcup_{i=0}^{s(r+1)} R_i(\theta)\),
\end{gather*}
and the same argument as above concludes the proof.
\end{proof}

Our goal is now to obtain similar conditions on the trees $\theta_{\infty}^{(k)}$ and $\theta_{\infty}^{(-k+1)}$, sufficient to get the ball inclusions \eqref{E: R+ ball inclusion} and \eqref{E: L+ ball inclusion}. Note that we cannot apply the above result directly, since $\theta_{\infty}^{(k)}$ and $\theta_{\infty}^{(-k+1)}$ are elements of $\lstrees^{\ast}(0)$ and $\lstrees^{\ast}(1)$ instead of $\ora{\lstrees}$ and $\ola{\lstrees}$. Moreover, for example in $\theta_{\infty}^{(k)}$, we are not interested in \emph{all} the vertices which are on the right of the spine, but only in those which are on the right of the segment $\segs{\pt{e}_k(\theta_{\infty})}{\pt{e}_0(\theta_{\infty})}$. Informally, the others are ``cut-off'' from the root when we split the quadrangulation $Q_{\infty}$ along the maximal geodesic, so they do not belong to the neighbourhood of $\pt{e}_k(\theta_{\infty})$ in $\ora{Q}_{\infty}^{(k)}$.

Therefore, for all $k \in \N$, we further decompose the trees $\theta_{\infty}^{(k)}$ and $\theta_{\infty}^{(-k+1)}$. Recall the notation introduced in Section \ref{S: Bi-marked theta(n)}. We let
\begin{gather*}
R_{\infty}^{(k)} = \bigcup_{i=1}^{a_{\infty}^{(k)}+c_{\infty}^{(k)}} \tau_{\infty,i}^{(k)}
\qquad \mbox{and} \qquad
R_{\infty}^{(k)} (s) = R_{\infty}^{(k)} \cap B_{\theta_{\infty}^{(k)}} (s)
\quad \forall s \geq 0,
\end{gather*}
and similarly,
\begin{gather*}
L_{\infty}^{(-k+1)} = \bigcup_{i=1}^{b_{\infty}^{(k)}+c_{\infty}^{(k)}} (\tau'_{\infty,i})^{(k)}
\qquad \mbox{and} \qquad
L_{\infty}^{(-k+1)} (s) = L_{\infty}^{(-k+1)} \cap L_{\theta_{\infty}^{(-k+1)}} (s)
\quad \forall s \geq 0.
\end{gather*}
Note that we have, for example, $R_{\infty}^{(k)} \subset R(\theta_{\infty}^{(k)})$ and $R_{\infty}^{(k)} (s) \subset \RB_{\theta_{\infty}^{(k)}}(s)$.
We consider the following events:
\begin{itemize}
\item $\mathcal{A}_{R+}^{(k)} (r,s)$: ``every vertex $\pt{v} \in R_{\infty}^{(k)} \setminus (R_{\infty}^{(k)} (s))$ has label greater than $r$ in $\theta_{\infty}^{(k)}$'',
\item $\mathcal{A}_{L+}^{(-k+1)} (r,s)$: ``every vertex $\pt{v} \in L_{\infty}^{(-k+1)} \setminus (L_{\infty}^{(-k+1)} (s))$ has label greater than $r$ in $\theta_{\infty}^{(-k+1)}$''.
\end{itemize}
For $k=\infty$, we complement this notation by setting
\begin{gather*}
\mathcal{A}_{R+}^{(\infty)} (r,s) = \{ \ora{\theta_{\infty}} \in \mathbb{A}_{R+} (r,s) \} \quad \mbox{and} \quad
\mathcal{A}_{L+}^{(-\infty)} (r,s) = \{ \ola{\theta_{\infty}} \in \mathbb{A}_{L+} (r,s) \}.
\end{gather*}
We can now adapt Lemma \ref{T: LR conditions for ball inclusions} to $\theta_{\infty}^{(k)}$ in the following way:

\begin{lem} \label{T: LR conditions for k<infty}
Let $r \in \N$, and consider two sequences of positive integers $(s(r'))_{0 \leq r' \leq r}$ and $(h(r'))_{1 \leq r' \leq r}$. For all $k \in \N \cup \{\infty\}$, we have that:
\begin{enumerate}
\item Conditionally on $\spine_{s(r)+1}(\theta_{\infty}^{(k)}) \prec \pt{e}_0(\theta_{\infty})$ in $\theta_{\infty}^{(k)}$ and on the event
\begin{gather} \label{E: Cond for ball inclusion, k positive}
\bigcap_{r'=1}^r \(\theta_{\infty}^{(k)} \in \mathbb{A}_L (r',s(r'-1),s(r'),h(r'))\) \cap \mathcal{A}_{R+}^{(k)} (r',s(r')),
\end{gather}
we have
\begin{gather*}
V\(B_{\ora{Q}_{\infty}^{(k)}}(r)\) \subset V\(B_{\theta_{\infty}^{(k)}}(h(r))\)
\end{gather*}
almost surely.
\item Conditionally on $\spine_{s(r)+1}(\theta_{\infty}^{(-k+1)}) \prec \pt{e}_0(\theta_{\infty})$ in $\theta_{\infty}^{(-k+1)}$ and on the event
\begin{gather} \label{E: Cond for ball inclusion, k negative}
\bigcap_{r'=1}^r \(\theta_{\infty}^{(-k+1)} \in \mathbb{A}_R (r',s(r'-1),s(r'),h(r'))\) \cap \mathcal{A}_{L+}^{(-k+1)} (r',s(r')),
\end{gather}
we have
\begin{gather*}
V\(B_{\ola{Q}_{\infty}^{(k)}}(r)\) \subset V\(B_{\theta_{\infty}^{(-k+1)}}(h(r))\) \cup \{\lambda_i: \abs{i}\leq r \}
\end{gather*}
almost surely.
\end{enumerate}
\end{lem}

Figure \ref{F: Conditions finite k} illustrates the ``new'' conditions which appear, compared to the conditions of Lemma \ref{T: LR conditions for ball inclusions} (both are shown for the first case). Note that the condition on the left-hand side of $\theta_{\infty}^{(k)}$ is exactly the same as in Lemma \ref{T: LR conditions for ball inclusions}, already illustrated in Figure \ref{F: LRConditions}.

\begin{figure}[t]
\begin{center}
\includegraphics{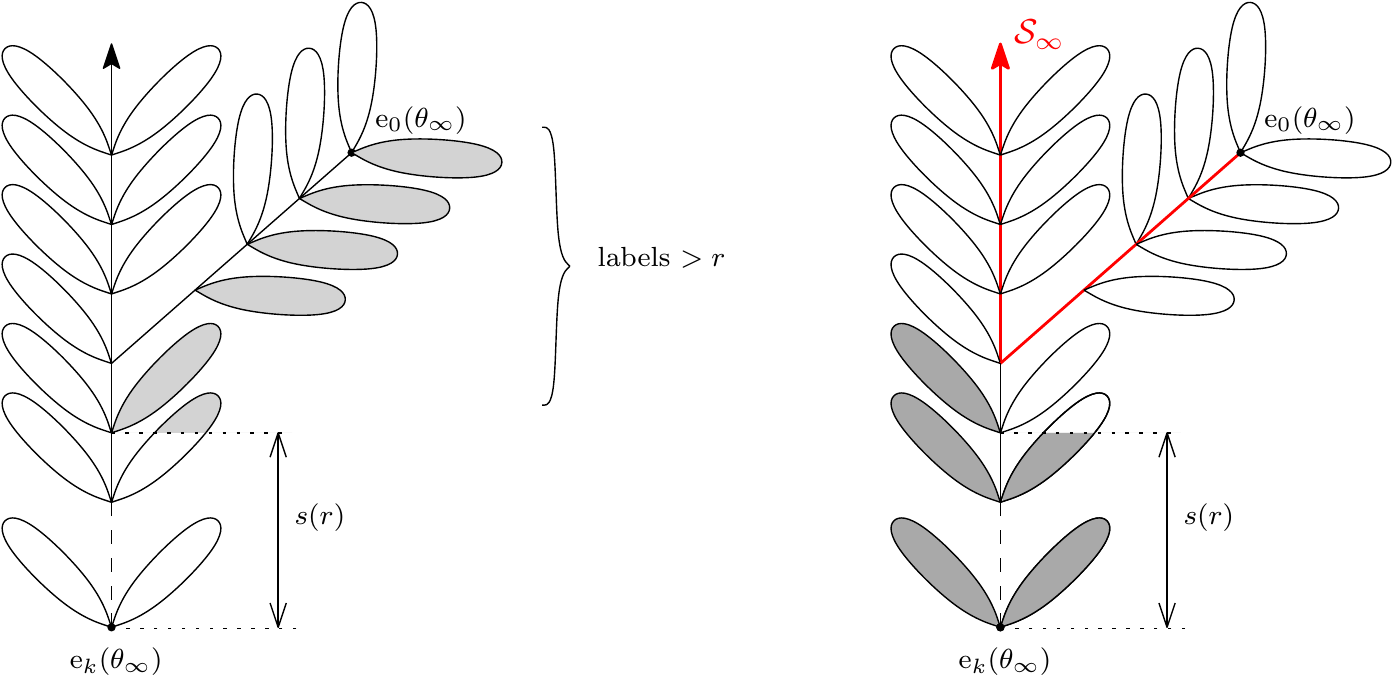}
\caption{Illustration of the event $\mathcal{A}_{R+}(s(r))$ (on the left), and of the additional condition $\spine_{s(r)+1}(\theta_{\infty}^{(k)}) \prec \pt{e}_0(\theta_{\infty})$ in $\theta_{\infty}^{(k)}$ (on the right). The second figure emphasises the fact that under the condition $\spine_{s(r)+1}(\theta_{\infty}^{(k)}) \prec \pt{e}_0(\theta_{\infty})$, the spine $\mathcal{S}_{\infty}$ does not intersect the set $R_{\infty}^{(k)} (s(r)) \cup \bigcup_{i=0}^{s(r)} L_i(\theta_{\infty}^{(k)})$ (and in particular, it does not contain $e_k(\theta_{\infty})$). This will be used in the proof of Lemma \ref{T: LR conditions for k<infty}.}
\label{F: Conditions finite k}
\end{center}
\end{figure}

\begin{proof}
The case where $k=\infty$ is a direct application of Lemma \ref{T: LR conditions for ball inclusions}. From now on, we fix $k \in \N$.

Let $\mathcal{S}_{\infty} = \{\spine_i(\theta_{\infty}): i\geq 0 \}$ be the spine of $\theta_{\infty}$, and $\Gamma'_{\infty} = \{\pt{e}'_{k'}: k' \geq 1 \}$ be the ``copy'' of the infinite geodesic ray we introduced in the definition of the split quadrangulation $\spq(Q_{\infty})$ (see for example Figure \ref{F: Split UIPQ}). The construction of $\spq(Q_{\infty})$ ensures that there are no edges between the vertices of $(\Gamma'_{\infty} \cup R(\theta_{\infty})) \setminus \mathcal{S}_{\infty}$ and the vertices of $L(\theta_{\infty}) \setminus \mathcal{S}_{\infty}$. As a consequence, any geodesic from a point of $(\Gamma'_{\infty} \cup R(\theta_{\infty})) \setminus \mathcal{S}_{\infty}$ to a point of $L(\theta_{\infty}) \setminus \mathcal{S}_{\infty}$ contains a vertex of $\mathcal{S}_{\infty}$. 

Note that we have the following equalities:
\begin{gather*}
R(\theta_{\infty})
 = R(\theta_{\infty}^{(k)}) \setminus R_{\infty}^{(k)}
 = R(\theta_{\infty}^{(-k+1)}) \cup L_{\infty}^{(-k+1)} \\
L(\theta_{\infty})
 = L(\theta_{\infty}^{(k)}) \cup R_{\infty}^{(k)}
 = L(\theta_{\infty}^{(-k+1)}) \setminus L_{\infty}^{(-k+1)}.
\end{gather*}

In the first case, the same induction as in the proof Lemma \ref{T: LR conditions for ball inclusions} shows that conditionally on \eqref{E: Cond for ball inclusion, k positive}, we have
\begin{gather} \label{E: Partial ball inclusion, finite k}
V\(B_{\ora{Q}_{\infty}^{(k)}}(r)\) \cap V\(L(\theta_{\infty})\) \subset V\(R_{\infty}^{(k)} (s(r)) \cup \bigcup_{i=0}^{s(r)} L_i(\theta_{\infty}^{(k)})\).
\end{gather}
Indeed, the first step of the induction shows that there are no vertices belonging to the ball $B_{\ora{Q}_{\infty}^{(k)}}(r)$ after $\spine_{s(r)}(\theta_{\infty}^{(k)})$ in the clockwise order, or equivalently
\begin{gather*}
V\(B_{\ora{Q}_{\infty}^{(k)}}(r)\) \cap V\(L(\theta_{\infty}^{(k)})\) \subset V\(\bigcup_{i=0}^{s(r)} L_i(\theta_{\infty}^{(k)})\),
\end{gather*}
and since the vertices in $R_{\infty}^{(k)} \setminus R_{\infty}^{(k)}(s(r))$ all have labels greater than $r$, we also have
\begin{gather*}
V\(B_{\ora{Q}_{\infty}^{(k)}}(r)\) \cap V\(R_{\infty}^{(k)}\) \subset V\(R_{\infty}^{(k)} (s(r))\).
\end{gather*}
Noting that $L(\theta_{\infty}^{(k)}) \cup R_{\infty}^{(k)} = L(\theta_{\infty})$ yields inclusion \eqref{E: Partial ball inclusion, finite k}.

To conclude the proof of the first point, we only have to show that the vertices of $R(\theta_{\infty}) \setminus \mathcal{S}_{\infty}$ are at distance at least $r+1$ from $\pt{e}_k(\theta_{\infty})$ in $\ora{Q}_{\infty}^{(k)}$. Let $\pt{v} \in V(R(\theta_{\infty}) \setminus \mathcal{S}_{\infty})$, and let $\gamma$ be a geodesic path from $\pt{v}$ to $\pt{e}_k(\theta_{\infty})$ in $\ora{Q}_{\infty}^{(k)}$. The condition $\spine_{s(r)+1}(\theta_{\infty}^{(k)}) \prec \pt{e}_0(\theta_{\infty})$ now has two consequences, as noted in the caption of Figure \ref{F: Conditions finite k}:
\begin{itemize}
\item First, $\pt{e}_k(\theta_{\infty})$ belongs to $L(\theta_{\infty}) \setminus \mathcal{S}_{\infty}$. Thus the geodesic $\gamma$ goes from a point of $R(\theta_{\infty}) \setminus \mathcal{S}_{\infty}$ to a point of $L(R_{\infty}) \setminus \mathcal{S}_{\infty}$, so there exists a vertex $\pt{v}'$ of $\gamma$ which belongs to the spine $\mathcal{S}_{\infty}$ (see the remark we made at the beginning of the proof).
\item Second, the set
\begin{gather*}
R_{\infty}^{(k)} (s(r)) \cup \bigcup_{i=0}^{s(r)} L_i(\theta_{\infty}^{(k)})
\end{gather*}
does not intersect $\mathcal{S}_{\infty}$, so inclusion \eqref{E: Partial ball inclusion, finite k} implies that
\begin{gather} \label{E: Spine and ball do not intersect}
V\(B_{\ora{Q}_{\infty}^{(k)}}(r)\) \cap V\(\mathcal{S}_{\infty}\) = \emptyset.
\end{gather}
\end{itemize}
Putting these two facts together, we get that
\begin{gather*}
d_{\ora{Q}_{\infty}^{(k)}} (\pt{v},\pt{e}_k(\theta_{\infty})) \geq d_{\ora{Q}_{\infty}^{(k)}} (\pt{v}',\pt{e}_k(\theta_{\infty})) \geq r+1.
\end{gather*}

Similarly, in the second case, conditionally on \eqref{E: Cond for ball inclusion, k negative}, we have
\begin{gather*}
V\(B_{\ola{Q}_{\infty}^{(k)}}(r)\) \cap V\(R(\theta_{\infty})\) \subset V\(L_{\infty}^{(k)} (s(r)) \cup \bigcup_{i=0}^{s(r)} R_i(\theta_{\infty}^{(k)})\),
\end{gather*}
and conditionally on $\spine_{s(r)+1}(\theta_{\infty}^{(-k+1)}) \prec \pt{e}_0(\theta_{\infty})$, the latter set does not intersect $\mathcal{S}_{\infty}$, so equation \eqref{E: Spine and ball do not intersect} still holds. Thus we only have to show that the vertices of $L(\theta_{\infty}) \setminus \mathcal{S}_{\infty}$ are at distance at least $r+1$ from $\pt{e}'_k$ in $\ola{Q}_{\infty}^{(k)}$. As above, for every such vertex $\pt{v}$, any geodesic path from $\pt{v}$ to $\pt{e}'_k$ in $\ola{Q}_{\infty}^{(k)}$ intersects $\mathcal{S}_{\infty}$, hence
\begin{gather*}
d_{\ola{Q}_{\infty}^{(k)}} (\pt{v},\pt{e}_k(\theta_{\infty})) \geq r+1.
\end{gather*}
\end{proof}

From now on, we fix $r \in \N$. The goal of the next sections is to show that the above conditions hold with arbitrarily high probability, for $k$ large enough. For condition \eqref{E: Cond for ball inclusion, k positive}, the main ingredients are the following lemmas:

\begin{lem} \label{T: Left-hand condition}
Let $s \in \N$ and $\epsilon > 0$. There exists $s_L = s_L(r,s,\epsilon)$ such that for all $s' \geq s_L$, there exists $h_L(s',\epsilon)$ such that for all $k$ large enough, possibly infinite, we have
\begin{gather*}
\P[sz]{\theta_{\infty}^{(k)} \notin \mathbb{A}_L (r,s,s',h_L(s',\epsilon))} \leq \epsilon.
\end{gather*}
\end{lem}

\begin{lem} \label{T: Right-hand condition}
For all $\epsilon > 0$, there exists $s_R = s_R(r,\epsilon)$ such that for all $k$ large enough, possibly infinite, we have
\begin{gather*}
\P[sz]{\ol{\mathcal{A}_{R+}^{(k)} (r,s_R)}} \leq \epsilon,
\end{gather*}
where $\ol{\mathcal{A}_{R+}^{(k)} (r,s_R)}$ denotes the contrary of the event $\mathcal{A}_{R+}^{(k)} (r,s_R)$.
\end{lem}

The proofs of these results are given in Sections \ref{S: Left-hand condition} and \ref{S: Right-hand condition}, respectively. A first step consists in studying the spine labels of $\ora{\theta_{\infty}}$: this is what we do in Section \ref{S: Spine labels}. In Section \ref{S: Conclusion}, we finally put all these ingredients together to complete the proof of Proposition \ref{T: Ball inclusions}.

\subsection{Two properties of the spine labels} \label{S: Spine labels}

In this section, we show two lemmas on the spine labels $S_i(\ora{\theta_{\infty}})$. The first one gives an upper bound which holds almost surely, for all $i$ large enough. The second one gives a lower bound which holds with high probability.

\begin{lem} \label{T: Upper bound on the spine labels}
There exists a constant $K$ such that almost surely, for all $i$ large enough, we have
\begin{gather*}
S_i(\ora{\theta_{\infty}}) \leq K \sqrt{i \ln(i)}.
\end{gather*}
\end{lem}

\begin{proof}
Recall that the distribution of $(S_i(\ora{\theta_{\infty}}))_{i \geq 0}$ is given in Theorem \ref{T: Joint cv of the rerooted trees}. Let $K > 0$ and $i \geq 1$. Recall that
\begin{gather*}
\P[sz]{S_i(\ora{\theta_{\infty}}) > K \sqrt{i \ln(i)}} = \E[sz]{M_i \ind{\hat{X}_i > K \sqrt{i\ln(i)} }},
\end{gather*}
where $\hat{X}$ is a random walk with uniform steps in $\{-1,0,1\}$ and $M$ is the martingale defined by
\begin{gather*}
M_i = \frac{f(\hat{X}_i)}{f(\hat{X}_0)} \prod_{j=0}^{i-1} w(\hat{X}_j).
\end{gather*}
Note that $M_i \leq f(\hat{X}_i)$ almost surely. Thus, for all $\lambda > 0$, we have
\begin{align} \label{E: Bound on P(S_i > K sqrt(i ln(i)))}
\P[sz]{S_i(\ora{\theta_{\infty}}) > K \sqrt{i \ln(i)}}
& \leq \E[sz]{M_i e^{-\lambda K \sqrt{i\ln(i)}} e^{\lambda \hat{X}_i}} \nonumber \\
& \leq C e^{-\lambda K \sqrt{i\ln(i)}} \E[sz]{\hat{X}_i^4 e^{\lambda \hat{X}_i}},
\end{align}
where $C$ denotes a constant such that $f(x) = x(x+3)(2x+3) \leq C x^4$ for every $x \geq 0$. Now $\E[sz]{\hat{X}_i^4 e^{\lambda \hat{X}_i}}$ is the fourth derivative of $\E[sz]{e^{\lambda \hat{X}_i}}$, and we have
\begin{gather*}
\E[sz]{e^{\lambda \hat{X}_i}} = e^{i \psi(\lambda)},
\end{gather*}
where $\psi(\lambda)$ denotes the Laplace transform of a uniform random variable in $\{-1,0,1\}$. Now we have
\begin{gather*}
\psi(\lambda) = \ln\( \frac{1+2\cosh(\lambda)}{3} \) \leq c \lambda^2
\end{gather*}
for a suitable constant $c>0$, and that the first four derivatives of $\psi$ are bounded. Therefore, there exists a positive constant $C'$ such that
\begin{gather*}
C \E[sz]{\hat{X}_i^4 e^{\lambda \hat{X}_i}} \leq C' i^4 e^{i \psi(\lambda)} \leq C' i^4 e^{ci \lambda^2} \qquad \forall \lambda >0.
\end{gather*}
Putting this together with \eqref{E: Bound on P(S_i > K sqrt(i ln(i)))}, we get
\begin{gather*}
\P[sz]{S_i(\ora{\theta_{\infty}}) > K \sqrt{i \ln(i)}}
\leq C' i^4 e^{ci\lambda^2-\lambda K \sqrt{i\ln(i)}} \qquad \forall \lambda > 0.
\end{gather*}
Choosing the optimal value $\lambda = K \sqrt{\ln(i)}/ (2c\sqrt{i})$ gives
\begin{gather*}
\P[sz]{S_i(\ora{\theta_{\infty}}) > K \sqrt{i \ln(i)}}
\leq C' i^4 e^{-(K^2/2c) \ln(i)} = C' i^{4-K^2/2c}.
\end{gather*}
As a consequence, for all $K$ large enough (such that $4-K^2/2c <-1$), the sum of the probabilities $\P[sz]{S_i(\ora{\theta_{\infty}}) > K \sqrt{i \ln(i)}}$ is finite. Applying the Borel--Cantelli lemma concludes the proof.
\end{proof}

\begin{lem} \label{T: Lower bound on the spine labels}
For all $\eta > 0$, there exists $\delta > 0$ such that for all $s$ large enough, we have
\begin{gather*}
\P[sz]{\exists i \geq \ent{\eta s}: S_i(\ora{\theta_{\infty}}) < \ent{\delta \sqrt{s}}} \leq \eta.
\end{gather*}
\end{lem}

\begin{proof}
Since $(S_i(\ora{\theta_{\infty}}))_{i \geq 0}$ has the same distribution as $(\tilde{X}_i)_{i \geq 0}$ with $\tilde{X}_0=1$, it is enough to show that
\begin{gather*}
\Pfrom[sz]{1}{\exists i \geq \ent{\eta s}: \tilde{X}_i < \ent{\delta \sqrt{s}}} \leq \eta.
\end{gather*}
Recall that, as stated in the introduction, we have the convergence
\begin{gather*}
\( \frac{1}{\sqrt{n}} \tilde{X}_{\ent{nt}}\)_{t \geq 0} \cvg{n \rightarrow \infty}{(d)} (Z_{2t/3})_{t \geq 0},
\end{gather*}
where $Z$ denotes a seven-dimensional Bessel process. As a consequence, there exists constants $\delta_1 > 0$ and $s_1 \in \N$ such that, for all $s \geq s_1$, we have
\begin{gather*}
\Pfrom[sz]{1}{\tilde{X}_{\ent{\eta s}} \leq \sqrt{\eta s} \cdot \delta_1} \leq \frac{\eta}{2}.
\end{gather*}
Fix $s \geq s_1$. Using the Markov property at time $\ent{\eta s}$, for any $\delta > 0$, we can now write
\begin{align}
\Pfrom[sz]{1}{\exists i \geq \ent{\eta s}: \tilde{X}_i < \ent{\delta \sqrt{s}}}
& \leq \frac{\eta}{2} + \sum_{x \geq \sqrt{\eta s} \cdot \delta_1} \Pfrom{1}{\tilde{X}_{\ent{\eta s}} = x} \Pfrom[sz]{x}{\exists i \geq 0: \tilde{X}_i < \ent{\delta \sqrt{s}}} \nonumber \\
& \leq \frac{\eta}{2} + \sum_{x \geq \sqrt{\eta s} \cdot \delta_1} \Pfrom{1}{\tilde{X}_{\ent{\eta s}} = x} \Pfrom[sz]{x}{\exists i \geq 0: \tilde{X}_i < \frac{\delta}{\delta_1 \sqrt{\eta}} x} \nonumber \\
& = \frac{\eta}{2} + \sum_{x \geq \sqrt{\eta s} \cdot \delta_1} \Pfrom{1}{\tilde{X}_{\ent{\eta s}} = x} \Pfrom[sz]{x}{\tilde{T}_{\ent{\delta x/\delta_1 \sqrt{\eta}}} < \infty}, \label{E: Partial lower bound on the spine labels}
\end{align}
where $\tilde{T}_{x'}$ denotes the first hitting time of $x'$ for $\tilde{X}$. It was shown in the proof of Lemma \ref{T: Values of SumFW} that for all $x \geq x'$, we have
\begin{gather*}
\Pfrom{x}{\tilde{T}_{x'} < \infty} = \frac{h(x')}{h(x)},
\end{gather*}
for a given non-constant polynomial $h$. Thus there exists constants $\delta_2$  and $x_2$ such that for all $x \geq x_2$, we have
\begin{gather*}
\Pfrom{x}{\tilde{T}_{\ent{\delta_2 x}} < \infty} \leq \frac{\eta}{2}.
\end{gather*}
Putting this together with \eqref{E: Partial lower bound on the spine labels}, for $\delta = \delta_2 \sqrt{\delta_1 \eta}$ and $s \geq s_1 \wedge (x_2^2/(\eta \delta_1^2))$, we get
\begin{align*}
\Pfrom[sz]{1}{\exists i \geq \ent{\eta s}: \tilde{X}_i < \ent{\delta \sqrt{s}}}
& \leq \frac{\eta}{2} \(1 + \sum_{x \geq \sqrt{\eta s} \cdot \delta_1} \Pfrom{1}{\tilde{X}_{\ent{\eta s}} = x} \) \leq \eta.
\end{align*}
\end{proof}

\subsection{Proof of the left-hand condition} \label{S: Left-hand condition}

In this section, we give the proof of Lemma \ref{T: Left-hand condition}. This result mainly uses the upper bound on the spine labels of $\ora{\theta_{\infty}}$, and the explicit expressions of the distribution of $L_i(\ora{\theta_{\infty}})$, for $i \geq 0$.

\begin{proof}[Proof of Lemma \ref{T: Left-hand condition}]
Since for all $s,s',h$, $\lstrees \setminus \mathbb{A}_L (r,s,s',h)$ is a closed set, we have
\begin{gather*}
\limsup_{k \rightarrow \infty} \P[sz]{\theta_{\infty}^{(k)} \notin \mathbb{A}_L (r,s,s',h)} \leq \P[sz]{\ora{\theta_{\infty}} \notin \mathbb{A}_L (r,s,s',h)},
\end{gather*}
so it is enough to show that the Lemma holds with $\ora{\theta_{\infty}}$ instead of $\theta_{\infty}^{(k)}$. For all $s,s',h \in \N$, we have
\begin{gather*}
\P[sz]{\ora{\theta_{\infty}} \notin \mathbb{A}_L (r,s,s',h)} \leq p_{r,s,s'} + \P[sz]{\exists i \leq s': L_i(T) \mbox{ has height greater than } h-s'},
\end{gather*}
where
\begin{gather*}
p_{r,s,s'} := \P[sz]{\forall i \in \{s+1,\ldots,s'\}, \min_{\pt{x} \in L_i (\ora{\theta_{\infty}})} \ora{l_{\infty}} (\pt{x}) > -r }.
\end{gather*}
Since for all $s'$, there exists $h$ such that the maximum of the heights of the Galton--Watson trees $L_0(\ora{\theta_{\infty}}), \ldots L_{s'}(\ora{\theta_{\infty}})$ is less than $h-s'$ with probability greater than $1-\epsilon/2$, it is enough to prove that the probabilities $p_{r,s,s'}$ converge to $0$ as $s' \rightarrow \infty$.

We first rewrite $p_{r,s,s'}$ using the spine-labels $S_i(\ora{\theta_{\infty}})$:
\begin{align*}
p_{r,s,s'}
& = \E[sz]{\prod_{i=s+1}^{s'} \GWgeom{S_i(\ora{\theta_{\infty}})} \{ (T,l) \in \ltrees{}: l > -r \} } \\
& = \E[sz]{\prod_{i=s+1}^{s'} \GWgeom{r+S_i(\ora{\theta_{\infty}})} \(\ltrees{}^+\) }
  = \E[sz]{\prod_{i=s+1}^{s'} w(r+S_i(\ora{\theta_{\infty}}))}.
\end{align*}
The above product is almost surely decreasing as $s' \rightarrow \infty$. Therefore, we only have to show that
\begin{gather*}
\prod_{i=s+1}^{s'} w(r+S_i(\ora{\theta_{\infty}})) \cvg{s' \rightarrow \infty}{} 0,
\end{gather*}
or equivalently that
\begin{gather} \label{E: Divergent series}
\sum_{i=s+1}^{s'} -\ln\(w(r+S_i(\ora{\theta_{\infty}}))\) \cvg{s' \rightarrow \infty}{} + \infty
\end{gather}
almost surely. Since $S_i(\ora{\theta_{\infty}}) \rightarrow + \infty$ almost surely, we can use the estimate
\begin{gather*}
w(x) = 1 -\frac{2}{x^2} + o\(\frac{1}{x^2}\).
\end{gather*}
This yields
\begin{gather*}
-\ln\(w(r+S_i(\ora{\theta_{\infty}}))\) \sim_{i \rightarrow \infty} \frac{2}{\(r+S_i(\ora{\theta_{\infty}})\)^2}.
\end{gather*}
Lemma \ref{T: Upper bound on the spine labels} now ensures that the right-hand term is $\as$ larger than $2/(K^2 i \ln(i))$ for all $i$ large enough, hence the $\as$ divergence \eqref{E: Divergent series}.
\end{proof}

\subsection{Proof of the right-hand condition} \label{S: Right-hand condition}

This section is devoted to the proof of Lemma \ref{T: Right-hand condition}. Note that the structure of the proof is close to Ménard \cite{Men}. More precisely, the lower bound we already proved in Lemma \ref{T: Lower bound on the spine labels} corresponds to a result Ménard obtains by putting together Lemma 2 and Proposition 5 of \cite{Men}, and Lemma \ref{T: Bound on P(theta not in A-tilde)} corresponds to Lemma 5 of \cite{Men}.

We begin by computing the probability $\P{R_{\infty}^{(k)}(s) = \theta^{\ast}}$, and some conditional probabilities on this event, for all suitable trees $\theta^{\ast}$. More precisely, let $\ltrees{[s]}^R$ denote the set of the labeled trees $(T,l) \in \ltrees{}(0)$ such that:
\begin{itemize}
\item The root of $T$ has exactly one offspring.
\item All labels in $T$ are positive, except the root-label.
\item The height of $T$ is $s$.
\item There are no vertices on the left of the path from the root to $\pt{x}_s$, where $\pt{x}_s$ denotes the leftmost vertex having height $s$. In other words, if $\pt{x}_0, \ldots, \pt{x}_s$ are the vertices on the path from the root to $\pt{x}_s$, then for all $\pt{x} \in T^{\ast} \setminus \{\pt{x}_0, \ldots, \pt{x}_s\}$, we have $\pt{x} > \pt{x}_s$ (where $<$ denotes the depth-first order).
\end{itemize}
Fix $\theta^{\ast} \in \ltrees{[s]}^R$. We let $\pt{x}_s = \pt{y}_1 < \ldots < \pt{y}_{n^{\ast}}$ denote the vertices of $\theta^{\ast}$ which have height $s$. For all $i \in \{0,\ldots,s-1\}$, we let $\tau_i^{\ast}$ denote the subtree formed by the vertices $\pt{x} \in T^{\ast}$ such that $\pt{x}_i \preceq \pt{x}$ and $\pt{x}_{i+1} \npreceq \pt{x}$ (note that $\tau_0^{\ast} = \{\pt{x}_0\}$). Finally, for all suitable $i$, we let $x_i = l^{\ast} (\pt{x}_i)$ and $y_i = l^{\ast}(\pt{y}_i)$. We have the following results:

\begin{lem}
Let $k > s+r$. With the above notation, we have
\begin{gather} \label{E: Distribution of RB(theta(infty,k)) (2)}
\P[sz]{R_{\infty}^{(k)}(s) = \theta^{\ast}}
= W_s(y_1,k) \frac{2^{n^{\ast}-1}}{6^{s-1} 12^{|T^{\ast}|-s}}\prod_{j=2}^{n^{\ast}} w(y_j),
\end{gather}
where
\begin{gather*}
W_s(x,k) = \frac{f(x)}{f(1)} \(1- \frac{C_x-s+1}{10(k+1)(k+2)}\) \qquad \forall x \leq s < k.
\end{gather*}
Moreover, this yields the conditional probabilities
\begin{gather} \label{E: P(no small label above y_1 on the right)}
\P[sz]{\min_{\bigcup_{i \geq s} \tau_{\infty,i}^{(k)}} l_{\infty}^{(k)} > r \bigg| R_{\infty}^{(k)}(s) = \theta^{\ast} }
= \frac{W_s(y_1-r,k-r)}{W_s(y_1,k)}
\end{gather}
and
\begin{gather} \label{E: P(no small label above y_j, j>2)}
\P[sz]{\min_{\pt{y}_j \preceq \pt{v}} l_{\infty}^{(k)}(\pt{v}) > r \bigg| R_{\infty}^{(k)}(s) = \theta^{\ast} }
= \frac{w(y_j-r)}{w(y_j)} \qquad \forall j \in \{2,\ldots,n^{\ast}\}.
\end{gather}
\end{lem}

Note that it is easy to see that these equations also hold for $k=\infty$, with $RB_{\ora{\theta_{\infty}}}(s)$ instead of $R_{\infty}^{(k)}(s)$ and $W_s(x,\infty) = f(x)/f(1)$ for all $x \leq s$.

\begin{proof}
Note that we have $x_s=y_1$; in the first two steps of the proof, it is more natural to use the notation $x_s$. The characterization of the distribution of $\theta_{\infty}^{(k)}$ given in Proposition \ref{T: Distribution of theta(infty,k)} yields
\begin{gather} \label{E: Distribution of RB(theta(infty,k)) (1)}
\P[sz]{R_{\infty}^{(k)}(s) = \theta^{\ast}}
= \P[sz]{(X_{\infty,1}^{(k)}, \ldots, X_{\infty,s}^{(k)}) = (x_1,\ldots,x_s) } \prod_{i=1}^{s-1} \GWgeom{x_i}^+ \( \theta: B_{\theta}(s-i) = \tau_i^{\ast} \).
\end{gather}
Furthermore, the computations of Section \ref{S: First dist eqns} show that
\begin{align*}
\P[sz]{X_{\infty,i}^{(k)} = x_i,\ \forall i \leq s }
& = \sum_{m \geq s} \frac{m+1}{3^s} \prod_{i=1}^{s-1} w(x_i)\ \Efrom[sz]{x_s}{\prod_{i=0}^{m-s} w(\hat{X}_i) \ind{\hat{X}_{m-s}=k}} \\
& = 3^{-s} \(\prod_{i=1}^{s-1} w(x_i)\) \frac{w(k)f(x_s)}{f(k)} \sum_{m \geq 0} (m+s+1) \Pfrom[sz]{x_s}{\tilde{X}_m=k} \\
& = 3^{-s} \(\prod_{i=1}^{s-1} w(x_i)\) \frac{w(k)f(x_s)}{f(k)} \(\SumFWStar{x_s}{k}+ (s-1)\SumFW{x_s}{k} \).
\end{align*}
Using the expressions obtained in Lemma \ref{T: Values of SumFW} and Lemma \ref{T: Values of SumFWStar}, and the hypothesis $s < k$, this gives
\begin{gather*}
\P[sz]{X_{\infty,i}^{(k)} = x_i,\ \forall i \leq s }
= 3^{-s+1} \(\prod_{i=1}^{s-1} w(x_i)\) \frac{f(x_s)}{f(1)} \(1-\frac{C_{x_s}-s+1}{10(k+1)(k+2)} \).
\end{gather*}
Besides, for all $i \leq s-1$, we have
\begin{gather*}
\GWgeom{x_i}^+ \( \theta: B_{\theta}(s-i) = \tau_i^{\ast} \)
= \frac{1}{2 w(x_i) 12^{|\tau_i^{\ast}|}} \prod_{j: \pt{v}_j \in \tau_i^{\ast}} 2 w(y_j).
\end{gather*}
Equation \eqref{E: Distribution of RB(theta(infty,k)) (1)} can now be rewritten as
\begin{align*}
\P[sz]{R_{\infty}^{(k)}(s) = \theta^{\ast}}
& = \frac{f(x_s)}{f(1)} \(1-\frac{C_{x_s}-s+1}{10(k+1)(k+2)} \) \prod_{i=1}^{s-1} \( \frac{w(x_i)}{6 w(x_i) 12^{|\tau_i^{\ast}|}} \prod_{j: \pt{v}_j \in \tau_i^{\ast}} 2 w(y_j) \) \nonumber \\
& = \frac{f(x_s)}{f(1)} \(1-\frac{C_{x_s}-s+1}{10(k+1)(k+2)} \) \frac{2^{n^{\ast}-1}}{6^{s-1} 12^{|T^{\ast}|-s}}\prod_{j=2}^{n^{\ast}} w(y_j),
\end{align*}
hence the first result of the lemma.

To get the conditional probability \eqref{E: P(no small label above y_1 on the right)}, we have to compute
\begin{gather*}
\P[sz]{\(R_{\infty}^{(k)}(s) = \theta^{\ast}\) \cap \(\min_{\bigcup_{i \geq s} \tau_{\infty,i}^{(k)}} l_{\infty}^{(k)} > r \) }.
\end{gather*}
Using the same decomposition as above, this probability can be written as
\begin{align*}
& \sum_{m \geq s} \frac{m+1}{3^m} \(\prod_{i=1}^{s-1} w(x_i)\) \sum_{\ul{x}' \in \mathcal{M}^+_{m-s,x_s \rightarrow k}} \(\prod_{i=0}^{m-s} \GWgeom{x'_i}^+ \((T,l): \min_T l > r\) \) \\
& \qquad \qquad \qquad \qquad \qquad \qquad \quad 
\prod_{i=1}^{s-1} \GWgeom{x_i}^+ \( \theta: B_{\theta}(s-i) = \tau_i^{\ast} \),
\end{align*}
or equivalently,
\begin{gather*}
\(\sum_{m \geq 0} \frac{m+s+1}{3^m} \sum_{\ul{x}' \in \mathcal{M}^+_{m,x_s-r \rightarrow k-r}}  \prod_{i=0}^m w(x'_i) \)
\frac{2^{n^{\ast}-1}}{6^{s-1} 12^{|T^{\ast}|-s}}\prod_{j=2}^{n^{\ast}} w(y_j).
\end{gather*}
Thus, we get
\begin{align*}
& \P[sz]{\(R_{\infty}^{(k)}(s) = \theta^{\ast}\) \cap \(\min_{\bigcup_{i \geq s} \tau_{\infty,i}^{(k)}} l_{\infty}^{(k)} > r \) } \\
& \qquad \qquad
= \frac{w(k-r)f(x_s-r)}{3f(k-r)} \( \SumFWStar{x_s-r}{k-r} + (s-1)\SumFW{x_s-r}{k-r} \) \frac{2^{n^{\ast}-1}}{6^{s-1} 12^{|T^{\ast}|-s}}\prod_{j=2}^{n^{\ast}} w(y_j) \\
& \qquad \qquad
= \frac{f(x_s-r)}{f(1)} \( 1- \frac{C_{x_s-r}-s+1}{10(k-r+1)(k-r+2)} \) \frac{2^{n^{\ast}-1}}{6^{s-1} 12^{|T^{\ast}|-s}}\prod_{j=2}^{n^{\ast}} w(y_j).
\end{align*}
This completes the proof of equation \eqref{E: P(no small label above y_1 on the right)}.

Finally, for all $j^{\ast} \in \{2,\ldots,n^{\ast}\}$, we have
\begin{align*}
& \P[sz]{\(R_{\infty}^{(k)}(s) = \theta^{\ast}\) \cap \(\min_{\pt{y}_j \preceq \pt{v}} l_{\infty}^{(k)}(\pt{v}) > r \) } \\
& \qquad \qquad
= 3^{-s+1} \(\prod_{i=1}^{s-1} w(x_i)\) W(y_1,k) \prod_{i=1}^{s-1} \frac{1}{2 w(x_i) 12^{|\tau_i^{\ast}|}} \(\prod_{\substack{j: \pt{v}_j \in \tau_i^{\ast} \\ j \neq j^{\ast}}} 2 w(y_j)\) \times 2 w(y_{j^{\ast}}-r) \\
& \qquad \qquad
= W_s(y_1,k) \frac{2^{n^{\ast}-1}}{6^{s-1} 12^{|T^{\ast}|-s}} w(y_{j^{\ast}}-r) \prod_{\substack{2 \leq j \leq n^{\ast} \\ j \neq j^{\ast}}} w(y_j),
\end{align*}
hence equation \eqref{E: P(no small label above y_j, j>2)}.
\end{proof}

The second step consists in studying the vertices of $R_{\infty}^{(k)}$ which are exactly at height $s$: we give an upper bound on the expectation of the number of such vertices, and show that with high probability, for $k$ large enough, these vertices have labels greater than $s^{\alpha}$, for $\alpha \in (0,1/2)$. Precise statements are given in Lemmas \ref{T: Bound on E[nb of vertices at height s]} and \ref{T: Bound on P(theta not in A-tilde)} below. Note that for all $k$, we have
\begin{gather*}
\partial R_{\infty}^{(k)}(s):= \lbrace \pt{v} \in R_{\infty}^{(k)}: \pt{v} \mbox{ has height } s \rbrace \subset \partial \RB_{\theta_{\infty}^{(k)}} (s).
\end{gather*}

\begin{lem} \label{T: Bound on E[nb of vertices at height s]}
For all $s \geq 1$ and $k \in \N$, we have
\begin{gather*}
\E[sz]{\card{\partial R_{\infty}^{(k)}(s)}} \leq s.
\end{gather*}
\end{lem}

\begin{proof}
For all $s \geq 1$ and $k \in \N$, we have
\begin{align*}
\E[sz]{\card{\partial R_{\infty}^{(k)}}(s)}
& = \sum_{i=1}^s \E[sz]{\card{\partial B_{\tau_{\infty,i}^{(k)}}(s-i)}} \\
& = \sum_{i=1}^s \E[sz]{ \frac{1}{w(X_i)} \E{\card{\partial B_{\tau}(s-i)}}},
\end{align*}
where $\tau$ denotes a Galton--Watson tree with offspring distribution $\geom(1/2)$. For all $h \geq 0$, we have $\E{\card{\partial B_{\tau}(h)}}=1$. As a consequence, the above equality gives
\begin{gather*}
\E[sz]{\card{\partial R_{\infty}^{(k)}}(s)}
= \sum_{i=1}^s \E[sz]{ \frac{1}{w(X_i)}}
\leq \sum_{i=1}^s 1 = s.
\end{gather*}
\end{proof}

We now consider the set
\begin{gather*}
\tilde{\mathbb{A}}_{R+} (r,s,\alpha) = \{ (T,l) \in \lstrees: \forall \pt{v} \in \partial \RB_T(s),\ l(\pt{v}) > \ent{s^{\alpha}} \}.
\end{gather*}

\begin{lem} \label{T: Bound on P(theta not in A-tilde)}
Fix $\alpha < 1/2$. For all $s$ large enough, there exists $k_1(s)$ such that for all $k \geq k_1(s)$, possibly infinite, we have
\begin{gather*}
\P[sz]{\theta_{\infty}^{(k)} \notin \tilde{\mathbb{A}}_{R+} (r,s,\alpha)} \leq \epsilon.
\end{gather*}
\end{lem}

\begin{proof}
First note that since $\lstrees \setminus \tilde{\mathbb{A}}_{R+} (r,s,\alpha)$ is a closed set, we have
\begin{gather*}
\limsup_{k \rightarrow \infty} \P[sz]{\theta_{\infty}^{(k)} \notin \tilde{\mathbb{A}}_{R+} (r,s,\alpha)}
\leq \P[sz]{\ora{\theta_{\infty}} \notin \tilde{\mathbb{A}}_{R+} (r,s,\alpha)},
\end{gather*}
so it is enough to show that the property holds for $k = \infty$. Moreover, the same arguments as in the proof of \cite[Lemma 5]{Men} show that for all $\eta \in (0,1/2)$, for all $s$ large enough, we have
\begin{gather*}
\P[sz]{\exists i \leq \ent{\eta s}-1: R_i(\ora{\theta_{\infty}}) \cap \partial B_{\ora{\theta_{\infty}}}(s) \neq \emptyset } \leq 4 \eta.
\end{gather*}
Thus, letting $I_{\eta}(s) = \{\ent{\eta s}, \ldots, s\}$, we have
\begin{gather*}
\P[sz]{\ora{\theta_{\infty}} \notin \tilde{\mathbb{A}}_{R+} (r,s,\alpha)}
\leq 4 \eta + \P[sz]{\exists i \in I_{\eta}(s): \min_{R_i(\ora{\theta_{\infty}})} \ora{l_{\infty}} \leq \ent{s^{\alpha}}}.
\end{gather*}
Lemma \ref{T: Lower bound on the spine labels} now ensures that for $\delta > 0$ and $s$ large enough, this probability is less than
\begin{gather} \label{E: Partial bound on P(theta not in A-tilde)}
5 \eta + \P[sz]{\( \exists i \in I_{\eta}(s): \min_{R_i(\ora{\theta_{\infty}})} \ora{l_{\infty}} \leq \ent{s^{\alpha}}\) \cap \(\forall i \in I_{\eta}(s), S_i(\ora{\theta_{\infty}}) \geq \ent{\delta \sqrt{s}}\)}.
\end{gather}
For all $(x_i)_{\ent{\eta s} \leq i \leq s}$, we have
\begin{align*}
& \P[sz]{ \exists i \in I_{\eta}(s): \min_{R_i(\ora{\theta_{\infty}})} \ora{l_{\infty}} \leq \ent{s^{\alpha}} \bigg| S_i(\ora{\theta_{\infty}}) = x_i\ \forall i \in I_{\eta}(s)} \\
& \qquad \qquad \qquad \qquad \qquad \qquad \qquad \qquad
\leq \sum_{i=\ent{\eta s}}^s \P[sz]{\min_{R_i(\ora{\theta_{\infty}})} \ora{l_{\infty}} \leq \ent{s^{\alpha}} \bigg| S_i(\ora{\theta_{\infty}}) = x_i} \\
& \qquad \qquad \qquad \qquad \qquad \qquad \qquad \qquad
\leq \sum_{i=\ent{\eta s}}^s \GWgeom{x_i}^+ \( (T,l) \in \ltrees{}^+: \min_T l \leq \ent{s^{\alpha}} \) \\
& \qquad \qquad \qquad \qquad \qquad \qquad \qquad \qquad
\leq \sum_{i=\ent{\eta s}}^s \frac{w(x_i)-w(x_i-\ent{s^{\alpha}})}{w(x_i)}.
\end{align*}
Furthermore, if we choose the integers $x_i$ in such a way that $x_i \geq \ent{\delta \sqrt{s}}$ for all $i$, we have
\begin{gather*}
\frac{w(x_i)-w(x_i-\ent{s^{\alpha}})}{w(x_i)} = \frac{4 \ent{s^{\alpha}}}{x_i^3} + o\(\frac{s^{\alpha}}{x_i^3}\) \leq \frac{4 s^{\alpha-3/2}}{\delta^3} + o\(s^{\alpha-3/2}\),
\end{gather*}
so
\begin{gather*}
\P[sz]{ \exists i \in I_{\eta}(s): \min_{R_i(\ora{\theta_{\infty}})} \ora{l_{\infty}} \leq \ent{s^{\alpha}} \bigg| S_i(\ora{\theta_{\infty}}) = x_i\ \forall i \in I_{\eta}(s)}
\leq \frac{4 s^{\alpha-1/2}}{\delta^3} + o\(s^{\alpha-1/2}\).
\end{gather*}
Putting this together with \eqref{E: Partial bound on P(theta not in A-tilde)}, we finally get
\begin{gather*}
\P[sz]{\ora{\theta_{\infty}} \notin \tilde{\mathbb{A}}_{R+} (r,s,\alpha)}
\leq 5 \eta + \frac{4 s^{\alpha-1/2}}{\delta^3} + o\(s^{\alpha-1/2}\) \leq 6 \eta
\end{gather*}
for all $s$ large enough.
\end{proof}

We are now ready to give the proof of Lemma \ref{T: Right-hand condition}.

\begin{proof}[Proof of Lemma \ref{T: Right-hand condition}]
Fix $\alpha \in (1/3,1/2)$. Lemma \ref{T: Bound on P(theta not in A-tilde)} show that for all $s$ large enough and $k \geq k_1(s)$ (possibly infinite), we have
\begin{gather*}
\P[sz]{\ol{\mathcal{A}_{R+}^{(k)} (r,s)}}
\leq 2 \epsilon + \P[sz]{\ol{\mathcal{A}_{R+}^{(k)} (r,s)} \cap \( \theta_{\infty}^{(k)} \in \tilde{\mathbb{A}}_{R+} (r,s,\alpha)\) }.
\end{gather*}
Letting 
\begin{gather*}
\Theta(s,\alpha) = \{ (T,l) \in \ltrees{[s]}^R: \min_{\partial B_T(s)} l > \ent{s^{\alpha}}  \},
\end{gather*}
we get that
\begin{gather} \label{E: Bound on P(theta(k) not in A-R(k)) (1)}
\P[sz]{\ol{\mathcal{A}_{R+}^{(k)} (r,s)}}
\leq 2 \epsilon + \sum_{\theta^{\ast} \in \Theta(s,\alpha)} \P[sz]{\(R_{\infty}^{(k)}(s) = \theta^{\ast}\) \cap \ol{\mathcal{A}_{R+}^{(k)} (r,s)}}.
\end{gather}
Fix $\theta^{\ast} \in \Theta_{\epsilon}(s,\alpha)$, and let $y_1,\ldots,y_{n^{\ast}}$ denote the labels of the vertices of height $s$ (from left to right) in $\theta^{\ast}$. Note that the condition $\theta^{\ast} \in \Theta_{\epsilon}(s,\alpha)$ means that we have $\ent{s^{\alpha}} < y_i \leq s$ for all $i \in \{1,\ldots,n^{\ast}\}$. Moreover, equations \eqref{E: P(no small label above y_1 on the right)} and \eqref{E: P(no small label above y_j, j>2)} show that
\begin{align*}
\P[sz]{\ol{\mathcal{A}_{R+}^{(k)} (r,s)}\ \bigg|\ R_{\infty}^{(k)}(s) = \theta^{\ast}}
\leq 1-\frac{W_s(y_1-r,k-r)}{W_s(y_1,k)} + \sum_{j=2}^{n^{\ast}} \( 1-\frac{w(y_j-r)}{w(y_j)} \).
\end{align*}
For all $y \leq s$, we have
\begin{gather*}
\frac{W_s(y-r,k-r)}{W_s(y,k)}
= \frac{f(y-r)}{f(y)} \(1+ \frac{\frac{C_{y-r}-s+1}{10(k-r+1)(k-r+2)} - \frac{C_y-s+1}{10(k+1)(k+2)}}{1 - \frac{C_y-s+1}{10(k+1)(k+2)}}\)
\geq \frac{f(y-r)}{f(y)},
\end{gather*}
so 
\begin{gather*}
0 \leq 1-\frac{W_s(y-r,k-r)}{W_s(y,k)} \leq 1-\frac{f(y-r)}{f(y)} \leq \epsilon
\end{gather*}
for all $s$ large enough and $y \in \{\ent{s^{\alpha}},\ldots,s\}$. Besides, uniformly in $y > \ent{s^{\alpha}}$, we have
\begin{gather*}
1-\frac{w(y-r)}{w(y)} \leq \frac{4r}{s^{3\alpha}} + o\(\frac{r}{s^{3\alpha}}\).
\end{gather*}
This yields
\begin{gather*}
\P[sz]{\ol{\mathcal{A}_{R+}^{(k)} (r,s)}\ \bigg|\ R_{\infty}^{(k)}(s) = \theta^{\ast}}
\leq \epsilon + n^{\ast} \( \frac{4r}{s^{3\alpha}} + o\(\frac{r}{s^{3\alpha}}\) \).
\end{gather*}
Putting this into \eqref{E: Bound on P(theta(k) not in A-R(k)) (1)}, we obtain
\begin{align*}
\P[sz]{\ol{\mathcal{A}_{R+}^{(k)} (r,s)}}
& \leq 3 \epsilon + \( \frac{4r}{s^{3\alpha}} + o\(\frac{r}{s^{3\alpha}}\) \) \sum_{\theta^{\ast} \in \Theta(s,\alpha)} \card{\partial B_{\theta^{\ast}}(s)} \P[sz]{R_{\infty}^{(k)}(s) = \theta^{\ast}} \\
& = 3 \epsilon + \( \frac{4r}{s^{3\alpha}} + o\(\frac{r}{s^{3\alpha}}\) \) \E[sz]{\card{\partial R_{\infty}^{(k)}(s)}}.
\end{align*}
Lemma \ref{T: Bound on E[nb of vertices at height s]} now implies that
\begin{gather*}
\P[sz]{\ol{\mathcal{A}_{R+}^{(k)} (r,s)}}
\leq 3 \epsilon + \frac{4r}{s^{3\alpha-1}} + o\(\frac{r}{s^{3\alpha-1}}\).
\end{gather*}
Since we took $\alpha > 1/3$, this concludes the proof.
\end{proof}

\subsection{Proof of Proposition \ref{T: Ball inclusions}} \label{S: Conclusion}

We can now prove Proposition \ref{T: Ball inclusions} by putting together the results of Lemmas \ref{T: Left-hand condition}, \ref{T: Right-hand condition} and \ref{T: LR conditions for k<infty}, and using the symmetry between the definitions of $\theta_{\infty}^{(k)}$ and $\theta_{\infty}^{(-k)}$.

\begin{proof}[Proof of Proposition \ref{T: Ball inclusions}]
Let $r \in \N$. For all $\epsilon \geq 0$, we consider the sequences $(s_{\epsilon}(r'))_{r' \geq 0}$ and $(h_{\epsilon}(r'))_{r' \geq 1}$ defined by $s_{\epsilon}(0)=0$, and for all $r' \geq 1$:
\begin{gather*}
s_{\epsilon}(r')=s_R(r',2^{-r'-1} \epsilon) \vee s_L(r',s_{\epsilon}(r'-1), 2^{-r'-1}\epsilon) \\
h_{\epsilon}(r')=h_L(s_{\epsilon}(r'),2^{-r'-1}\epsilon),
\end{gather*}
where $s_L$, $s_R$ and $h_R$ are the quantities introduced in Lemmas \ref{T: Left-hand condition} and \ref{T: Right-hand condition}. Note that for all $r'$, we have $\mathcal{A}_{R+}(r',s_{\epsilon}(r')) \subset \mathcal{A}_{R+}(r',s_R(r',2^{-r'-1} \epsilon))$. Thus, Lemmas \ref{T: Left-hand condition} and \ref{T: Right-hand condition} show that for all $r' \in \N$, for all $k$ large enough, we have
\begin{gather*}
\P[sz]{\(\theta_{\infty}^{(k)} \notin \mathbb{A}_L(r',s_{\epsilon}(r'-1),s_{\epsilon}(r'),h_{\epsilon}(r')) \) \cup \ol{\mathcal{A}_{R+}^{(k)}(r',s_{\epsilon}(r'))} } \leq 2^{-r'} \epsilon,
\end{gather*}
and as a consequence,
\begin{gather*}
\P[sz]{\bigcup_{r'=1}^r \(\theta_{\infty}^{(k)} \notin \mathbb{A}_L(r',s_{\epsilon}(r'-1),s_{\epsilon}(r'),h_{\epsilon}(r')) \) \cup \ol{\mathcal{A}_{R+}^{(k)}(r',s_{\epsilon}(r'))} } \leq \epsilon
\end{gather*}
for all $k$ large enough. Moreover, recalling the notation of Proposition \ref{T: Distribution of theta(infty,k)}, we have
\begin{gather*}
\spine_{s_{\epsilon}(r)+1}(\theta_{\infty}^{(k)}) \nprec \pt{e}_0(\theta_{\infty})
\end{gather*}
if and only if $I_{\infty}^{(k)} \leq s_{\epsilon}(r)$, which happens with probability at most $s_{\epsilon}(r)/(k+1)$. Therefore, for all $k$ large enough, the conditions stated in the first part of Lemma \ref{T: LR conditions for k<infty} hold with probability at least $1-2\epsilon$.

Finally, we can see from the symmetry between the definitions of $\theta_{\infty}^{(k)}$ and $\theta_{\infty}^{(-k)}$ that for all $r',s,s',h \in \N$, we have
\begin{gather*}
\P[sz]{\theta_{\infty}^{(-k+1)} \notin \mathbb{A}_R(r',s,s',h)} = \P[sz]{\theta_{\infty}^{(k-1)} \notin \mathbb{A}_L(r'+1,s,s',h)}
\end{gather*}
and
\begin{gather*}
\P[sz]{\ol{\mathcal{A}_{L+}^{(-k+1)}(r',s)}} = \P[sz]{\ol{\mathcal{A}_{R+}^{(k-1)}(r'-1,s)}}.
\end{gather*}
Thus, letting $\tilde{s}_{\epsilon}(0)=0$ and, for all $r' \geq 1$,
\begin{gather*}
\tilde{s}_{\epsilon}(r')=s_R(r'-1,2^{-r'-1} \epsilon) \vee s_L(r'+1,\tilde{s}_{\epsilon}(r'-1), 2^{-r'-1}\epsilon) \\
\tilde{h}_{\epsilon}(r')=h_L(\tilde{s}_{\epsilon}(r'),2^{-r'-1}\epsilon),
\end{gather*}
the probability
\begin{gather*}
\P[sz]{\bigcap_{r'=1}^r \(\theta_{\infty}^{(-k+1)} \notin \mathbb{A}_R(r',\tilde{s}_{\epsilon}(r'-1),\tilde{s}_{\epsilon}(r'),\tilde{h}_{\epsilon}(r')) \) \cap \mathcal{A}_{L+}^{(-k+1)}(r',\tilde{s}_{\epsilon}(r')) }
\end{gather*}
is equal to
\begin{gather*}
\P[sz]{\bigcap_{r'=1}^r \(\theta_{\infty}^{(k-1)} \notin \mathbb{A}_L(r'+1,\tilde{s}_{\epsilon}(r'-1),\tilde{s}_{\epsilon}(r'),\tilde{h}_{\epsilon}(r')) \) \cap \mathcal{A}_{R+}^{(k-1)}(r'-1,\tilde{s}_{\epsilon}(r')) } \leq \epsilon.
\end{gather*}
Similarly as above, this implies that the conditions stated in the second part of Lemma \ref{T: LR conditions for k<infty} hold with probability at least $1-2\epsilon$.

Therefore, Lemma \ref{T: LR conditions for k<infty} shows that for $h = h_{\epsilon}(r) \vee \tilde{h}_{\epsilon}(r)$, the inclusions \eqref{E: R+ ball inclusion} and \eqref{E: L+ ball inclusion} hold with probability at least $1-4\epsilon$, for all $k$ large enough.
\end{proof}

\bibliographystyle{plain}
\bibliography{../Biblio/Bibliographie}

\end{document}